\documentclass[a4paper,11pt]{article}
\usepackage{makeidx}
\usepackage[mathscr]{eucal}
\usepackage[dvips]{color}
\usepackage{amssymb, amsfonts, amsmath, amsthm, graphics} 

\newtheorem{proposition}{Proposition}[section]
\newtheorem{theorem}{Theorem}[section]
\newtheorem{lemma}[proposition]{Lemma}
\newtheorem{remark}{Remark}[section]

\numberwithin{equation}{section}

\title{Existence of a ground state and scattering for a nonlinear Schr\"odinger equation with critical growth}
\author{Takafumi Akahori, Slim Ibrahim, Hiroaki Kikuchi and Hayato Nawa}
\date{}

\begin{document}
\maketitle

\begin{abstract}
We study the energy-critical focusing nonlinear Schr\"odinger equation with an energy-subcritical perturbation. We show the existence of a ground state in the four or higher dimensions.  Moreover, we give a sufficient and necessary condition for a solution to scatter, in the spirit of Kenig-Merle \cite{Kenig-Merle}.   
\end{abstract}

\section{Introduction}
In this paper, we study the following nonlinear Schr\"odinger equation. 
\begin{equation}\tag{NLS}
\label{11/06/12/9:08}
i\frac{\partial \psi}{\partial t}+\Delta \psi 
+
f(\psi) 
+
|\psi|^{2^{*}-2}\psi
=0,
\end{equation}
where $\psi=\psi(x,t)$ is a complex-valued function on $\mathbb{R}^{d}\times \mathbb{R}$ ($d\ge 3$), $\Delta$ is the Laplace operator on $\mathbb{R}^{d}$, $2^{*}:=2+\frac{4}{d-2}$ and $f\colon \mathbb{C}\to \mathbb{C}$ is a continuously differentiable function in the $\mathbb{R}^{2}$-sense. We specify the nonlinearity $f$ later (see the assumptions \eqref{11/07/18/9:45}--\eqref{11/05/02/8:40} below); Especially, we assume the Hamiltonian structure (see \eqref{11/07/18/9:25} below), so that there exists a function $F \in C^{2}(\mathbb{C},\mathbb{R})$ such that the Hamiltonian for \eqref{11/06/12/9:08} is given by   
\begin{equation}
\label{11/07/19/11:34}
\mathcal{H}(u)
:=
\frac{1}{2}
\left\|\nabla u \right\|_{L^{2}}^{2}
-\frac{1}{2}\int_{\mathbb{R}^{d}}F(u)
-\frac{1}{2^{*}}\left\| u \right\|_{L^{2^{*}}}^{2^{*}}.
\end{equation}
Moreover, we assume that $f$ satisfies the mass-supercritical and energy-subcritical condition (see \eqref{11/04/30/22:54} and \eqref{11/05/02/8:40}). Hence, the equation  \eqref{11/06/12/9:08} is considered to be a perturbed one of 
\begin{equation}\tag{${\rm NLS}_{0}$}
\label{11/07/17/16:01}
i\frac{\partial \psi}{\partial t}+\Delta \psi 
+
|\psi|^{2^{*}-2}\psi
=0.
\end{equation}
Here, the Hamiltonian for \eqref{11/07/17/16:01} is 
\begin{equation}\label{11/07/31/18:19}
\mathcal{H}_{0}(u):=\frac{1}{2}\left\|\nabla u \right\|_{L^{2}}^{2}-\frac{1}{2^{*}}\left\| u \right\|_{L^{2^{*}}}^{2^{*}}. 
\end{equation}

It is well-known that the equation \eqref{11/07/17/16:01} dose not have an oscillatory standing wave; In contrast, \eqref{11/07/17/16:01} has the non-oscillatory solution  
\begin{equation}\label{11/04/10/13:23}
W(x):=\left( \frac{\sqrt{d(d-2)}}{1+|x|^{2}} \right)^{\frac{d-2}{d}}, 
\qquad 
x \in \mathbb{R}^{d}. 
\end{equation}
 
Our first aim is to show that a suitable perturbation $f$ gives rise to an oscillatory standing wave for $d\ge 4$. In other words, we intend to prove that for $d\ge 4$ and $\omega>0$, there exists a solution to the elliptic equation  
\begin{equation}\label{11/05/01/17:30}
- \Delta u + \omega u 
-f(u) 
- |u|^{2^{*}-2}u = 0, 
\qquad 
u \in H^{1}(\mathbb{R}^{d})\setminus \{0\}.
\end{equation}
In particular, we show the existence of a ground state for $d\ge 4$ (see Theorem \ref{11/05/01/18:28} and Remark \ref{11/05/01/18:35} below); A ground state means a solution to \eqref{11/05/01/17:30} which minimizes the action $\mathcal{S}_{\omega}$ among the solutions, where 
\begin{equation}
\label{11/05/01/0:00}
\mathcal{S}_{\omega}(u)
:=
\frac{\omega}{2} \left\|u \right\|_{L^{2}}^{2}
+
\mathcal{H}(u).
\end{equation}
\par 
We remark that, in \cite{AIKN}, the same authors considered the case $f(z)=\mu |z|^{p-1}z$ with $\mu \in \mathbb{R}$ and $2_{*}-1 <p <2^{*}-1$ ($2_{*}:=2+\frac{4}{d}$), and proved that if $d\ge 4$, then for any $\mu>0$ and $\omega>0$, there exists a ground state; on the other hand, if $d\ge 3$ and $\mu\le 0$, or $d=3$ and 
 $\displaystyle{\mu /\omega^{\frac{d-2}{4}\{ 2^{*}-(p+1) \}}}$ is sufficiently small, then the equation \eqref{11/05/01/17:30} has no solution. 
\par 
In this paper, we extend the result in \cite{AIKN} to a wider class of perturbations including 
\begin{equation}\label{11/10/11/17:12}
f(z)=\mu_{1}|z|^{p_{1}-1}z+\cdots +\mu_{k}|z|^{p_{k}-1}z,
\end{equation}
where $k \in \mathbb{N}$, $\mu_{1},\ldots, \mu_{k}>0$ and  $2_{*}-1<p_{1}< \cdots < p_{k}< 2^{*}-1$. 
\par 
Our second aim is to give a necessary and sufficient condition for solutions to scatter, in the spirit of Kenig-Merle \cite{Kenig-Merle}, for $d\ge 5$ (see Theorem \ref{11/06/10/17:09} below); Precisely, we introduce a set $A_{\omega,+}$ (see \eqref{11/08/24/12:02} below) and prove that any solution starting from $A_{\omega,+}$ exists globally in time and asymptotically behaves like a free solution in the distant future and past. Although we can introduce the set $A_{\omega,+}$ for $d\ge 3$, the scattering result is open in $d=3,4$, as well as the equation \eqref{11/07/17/16:01} (see \cite{Killip-Visan}).  
\par 
Now, we state our assumption of the perturbation $f$. We first assume that 
\begin{equation}\label{11/07/18/9:45}\tag{A0}
f(0)=\frac{\partial f}{\partial z}(0)
=\frac{\partial f}{\partial \bar{z}}(0)=0.
\end{equation}
For the Hamiltonian structure and mass conservation law, we assume that there exists a real-valued function $F \in C^{2}(\mathbb{C},\mathbb{R})$ such that
\begin{equation}\label{11/07/18/9:25}\tag{A1}
F(0)=0,
\qquad 
\frac{\partial F}{\partial \bar{z}}=f,
\qquad 
\Im{\biggm[ z\frac{\partial F}{\partial z}\biggm]}=0,
\end{equation}
so that $\bar{z}f(z) \in \mathbb{R}$ for any $z \in \mathbb{C}$. Besides, we assume that 
\begin{equation}\label{11/07/18/9:29}\tag{A2}
F\ge 0.
\end{equation}
This assumption \eqref{11/07/18/9:29} rules out the case $F(z)=-\frac{1}{p+1}|z|^{p+1}$ (or $f(z)=-|z|^{p-1}z$) for which the Pohozaev identity shows that there is no solution to \eqref{11/05/01/17:30}. \par  
To ensure the existence of ground state, we further need the monotonicity and convexity conditions, like in \cite{IMN}: Define the operator $D$ by   
\begin{equation}\label{11/07/18/20:36}
Dg(z)
:= 
z \frac{\partial g}{\partial z}(z)
+
\bar{z} \frac{\partial g}{\partial \bar{z}}(z)
\qquad 
\mbox{for $g \in C^{1}(\mathbb{C},\mathbb{C})$}.
\end{equation} 
Then, we assume that there exists $\varepsilon_{0}>0$ such that
\begin{align}
\label{11/04/30/22:39}\tag{A3}
&(D-2_{*}-\varepsilon_{0})F \ge 0, 
\\[6pt]
\label{11/04/30/22:40}\tag{A4}
&(D-2)(D-2_{*}-\varepsilon_{0})F \ge 0.
\end{align}
The conditions  \eqref{11/07/18/9:29}, \eqref{11/04/30/22:39} and  \eqref{11/04/30/22:40} imply  that  
\begin{equation}
\label{11/05/01/21:31}
D^{2}F\ge (2_{*}+\varepsilon_{0})DF \ge (2_{*}+\varepsilon_{0})^{2} F \ge 0.
\end{equation}
Finally, we make an assumption so that $f$ satisfies the mass-supercritical and energy-subcritical growth, like in \cite{IMN}: Fix a cut-off function $\chi \in C^{\infty}(\mathbb{R})$ such that $\chi(u)=1$ for $|u|\le 1$ and $\chi(u)=0$ for $|u|\ge 2$, and put $f_{\le 1}(u):=\chi(u)f(u)$ and $f_{\ge 1}:=f-f_{\le 1}$. Then, we assume that there exist $p_{1}$ and $p_{2}$ such that  $2_{*}-1<p_{1}\le p_{2}<2^{*}-1$, and   
\begin{equation}\tag{A5}
\label{11/04/30/22:54}
\left\{ 
\begin{array}{rll}
\displaystyle{
\left| \frac{\partial f_{\le 1}}{\partial z}(u) \right|
+
\left| \frac{\partial f_{\le 1}}{\partial \bar{z}}(u) \right| 
}
&\lesssim |u|^{p_{1}-1}
&\mbox{if $d=3,4$},
\\[12pt]
\displaystyle{
\left| \frac{\partial f_{\le 1}}{\partial z}(u)
-
\frac{\partial f_{\le 1}}{\partial z}(v)
 \right|
+
\left| \frac{\partial f_{\le 1}}{\partial \bar{z}}(u)
-
\frac{\partial f_{\le 1}}{\partial \bar{z}}(v)
 \right|
}
&\lesssim 
(|u|+|v|)^{p_{1}-2}|u-v|
&\mbox{if $d\ge 5,\ p_{1}\ge 2$},
\\[12pt]
\displaystyle{
\left| \frac{\partial f_{\le 1}}{\partial z}(u)
-
\frac{\partial f_{\le 1}}{\partial z}(v)
 \right|
+
\left| \frac{\partial f_{\le 1}}{\partial \bar{z}}(u)
-
\frac{\partial f_{\le 1}}{\partial \bar{z}}(v)
 \right|}
&\lesssim |u-v|^{p_{1}-1}
&\mbox{if $d\ge 5,\ p_{1}< 2$}
\end{array}
\right.
\end{equation}
and 
\begin{equation}\tag{A6}
\label{11/05/02/8:40}
\left\{ 
\begin{array}{rll}
\displaystyle{
\left| \frac{\partial f_{\ge 1}}{\partial z}(u)\right|
+
\left| \frac{\partial f_{\ge 1}}{\partial \bar{z}}(u)
 \right|}
& \lesssim |u|^{p_{2}-1}
&\mbox{if $d=3,4$},
\\[12pt]
\displaystyle{
\left| \frac{\partial f_{\ge 1}}{\partial z}(u)
-
\frac{\partial f_{\ge 1}}{\partial z}(v)
 \right|
+
\left| \frac{\partial f_{\ge 1}}{\partial \bar{z}}(u)
-
\frac{\partial f_{\ge 1}}{\partial \bar{z}}(v)
 \right|
}
&\lesssim (|u|+|v|)^{p_{2}-2}|u-v|
&\mbox{if $d\ge 5,\ p_{2}\ge 2$},
\\[12pt]
\displaystyle{
\left| \frac{\partial f_{\ge 1}}{\partial z}(u)
-
\frac{\partial f_{\ge 1}}{\partial z}(v)
 \right|
+
\left| \frac{\partial f_{\ge 1}}{\partial \bar{z}}(u)
-
\frac{\partial f_{\ge 1}}{\partial \bar{z}}(v)
 \right|}
&\lesssim |u-v|^{p_{2}-1}
&\mbox{if $d\ge 5,\ p_{2}< 2$}.
\end{array}
\right.
\end{equation}

As mentioned above, we prove the existence of a ground state to the equation \eqref{11/05/01/17:30} under the assumptions \eqref{11/07/18/9:45}--\eqref{11/05/02/8:40}. To this end, for any $\omega>0$, we introduce a variational value $m_{\omega}$: 
\begin{equation}\label{11/04/30/23:55}
m_{\omega}
:=
\inf\left\{ \mathcal{S}_{\omega}(u) 
\biggm| u\in H^{1}(\mathbb{R}^{d})\setminus \{0\},
\ \mathcal{K}(u)=0
\right\},
\end{equation}
where 
\begin{equation}\label{11/04/30/23:58}
\begin{split}
\mathcal{K}(u)
&:=
\frac{d}{d\lambda}\mathcal{H}(T_{\lambda}u)\biggm|_{\lambda=1}
=
\frac{d}{d\lambda}\mathcal{S}_{\omega}(T_{\lambda}u)\biggm|_{\lambda=1}
\\[6pt]
&=
\left\|\nabla u \right\|_{L^{2}}^{2}
-
\frac{d}{4}
\int_{\mathbb{R}^{d}}
\bigm( DF- 2F \bigm)(u)
-
\left\| u \right\|_{L^{2^{*}}}^{2^{*}}
\end{split}
\end{equation}
and 
$T_{\lambda}$ is the $L^{2}$-scaling operator, i.e.,  
\begin{equation}\label{11/04/30/23:59}
(T_{\lambda}u)(x):=\lambda^{\frac{d}{2}}u(\lambda x).
\end{equation}
It is well-known(see, e.g., \cite{Berestycki-Cazenave, LeCoz}) that 
 the minimizer of \eqref{11/04/30/23:55} becomes a ground state to \eqref{11/05/01/17:30}. Thus, it suffices to show the existence of the minimizer. 
\par 
In order to find the minimizer of the variational problem \eqref{11/04/30/23:55}, we need two auxiliary variational problems; The first one is  
\begin{equation}\label{11/04/30/23:54}
\widetilde{m}_{\omega}
=
\inf\left\{ \mathcal{I}_{\omega}(u) 
\biggm| u\in H^{1}(\mathbb{R}^{d})\setminus \{0\},
\ \mathcal{K}(u)\le 0
\right\}, 
\end{equation}
where 
\begin{equation}\label{11/04/30/23:56}
\begin{split}
\mathcal{I}_{\omega}(u)
&:=
\mathcal{S}_{\omega}(u)
- \frac{1}{2}
\mathcal{K}(u)
\\[6pt]
&=
\frac{\omega}{2}\left\| u \right\|_{L^{2}}^{2}
+
\frac{d}{8}\int_{\mathbb{R}^{d}}
\bigm(DF- 2_{*}F \bigm)(u)
+
\frac{1}{d}\left\| u \right\|_{L^{2^{*}}}^{2^{*}},
\end{split}
\end{equation} 
and the other one is 
\begin{equation}\label{11/04/09/17:01}
\sigma
:=
\inf\left\{
\left\| \nabla u \right\|_{L^{2}}^{2}
\biggm|
u \in \dot{H}^{1}(\mathbb{R}^{d}), 
\ 
\left\| u \right\|_{L^{2^{*}}}=1
\right\}. 
\end{equation}
An advantage of the problem \eqref{11/04/30/23:54} is that the functional $\mathcal{I}_{\omega}$ is positive thanks to \eqref{11/04/30/22:39}, and the constraint $\mathcal{K}\le 0$ is stable under the Schwarz symmetrization. Moreover, we have;  
\begin{proposition}\label{11/05/01/17:50}
Assume $d\ge 3$, $\omega>0$ and the conditions \eqref{11/07/18/9:45}--\eqref{11/05/02/8:40}. Then, it holds that 
\\[6pt]
{\rm (i)} $m_{\omega}=\widetilde{m}_{\omega}$,
\\[6pt]
{\rm (ii)} Any minimizer of the variational problem for $\widetilde{m}_{\omega}$ is also a minimizer for $m_{\omega}$.
\end{proposition}

The reason why we need the variational problem \eqref{11/04/09/17:01} is the following relation between $m_{\omega}$ and $\sigma$: 
\begin{lemma}
\label{11/05/01/23:05}
Assume $d\ge 4$, $\omega>0$ and the conditions \eqref{11/07/18/9:45}--\eqref{11/05/02/8:40}. Then, we have 
\begin{equation}\label{11/04/09/16:59}
m_{\omega}< \frac{1}{d}\sigma^{\frac{d}{2}}.
\end{equation} 
\end{lemma}

Here, it is worthwhile noting that the function $W$ given in 
\eqref{11/04/10/13:23} relates to the value $\sigma$: 
\begin{equation}\label{11/04/10/13:24}
\sigma^{\frac{d}{2}}
=\left\| T_{\lambda}'\nabla W \right\|_{L^{2}}^{2}
=\left\| T_{\lambda}'W \right\|_{L^{2^{*}}}^{2^{*}}
\quad 
\mbox{for any $\lambda>0$},
\end{equation}
where $T_{\lambda}'$ denotes the $\dot{H}^{1}$-scaling operator, i.e., 
\begin{equation}\label{11/10/12/11:55}
(T_{\lambda}'u)(x):=\lambda^{\frac{d-2}{2}}u(\lambda x)
.
\end{equation}
\par 
Since the proof of Lemma \ref{11/05/01/23:05} is similar to the one of Lemma 2.2 in \cite{AIKN}, we omit it.  
\par 
Using Lemma \ref{11/05/01/23:05}, we can find:   
\begin{theorem}\label{11/05/01/18:28}
Assume $d\ge 4$, $\omega>0$ and the conditions \eqref{11/07/18/9:45}--\eqref{11/05/02/8:40}. Then, there exists a minimizer of the variational problem for $m_{\omega}$. 
\end{theorem}

\begin{remark}\label{11/05/01/18:35}
{\rm (i)} 
We find from Proposition \ref{11/05/01/17:50} (i) that 
$m_{\omega}=\widetilde{m}_{\omega}>0$.
\\[6pt]
{\rm (ii)} 
It is necessary that $d\ge 4$; For, when $f(u)=\mu |u|^{p_{1}-1}u$ with a sufficiently small $\mu>0$, there is no minimizer of the variational problem for $m_{\omega}$ (see \cite{AIKN}). 
\\[6pt]
{\rm (iii)} As mentioned above, a minimizer of the variational problem for $m_{\omega}$ becomes a ground state to \eqref{11/05/01/17:30} (see \cite{Berestycki-Cazenave, LeCoz}). 
\end{remark}

In order to state our scattering result, we introduce a set $A_{\omega,+}$:
\begin{equation}\label{11/08/24/12:02}
A_{\omega,+}:=
\left\{ 
 u \in H^{1}(\mathbb{R}^{d}) \biggm| 
\mathcal{S}_{\omega}(u)<m_{\omega},
\ 
\mathcal{K}(u)>0. 
\right\}.
\end{equation}
Then, we have: 
\begin{theorem}\label{11/06/10/17:09}
Assume $d\ge 5$, $\omega>0$ and the conditions \eqref{11/07/18/9:45}--\eqref{11/05/02/8:40}. Then, any solution $\psi$ to \eqref{11/06/12/9:08} starting from $A_{\omega, +}$ exists globally in time. Furthermore, the solution $\psi$ scatters in $H^{1}(\mathbb{R}^{d})$, i.e., there exist $\phi_{+}, \phi_{-} \in H^{1}(\mathbb{R}^{d})$ such that 
\begin{equation}\label{11/12/02/9:55}
\lim_{t\to +\infty}\left\| \psi(t)-e^{it\Delta}\phi_{+}\right\|_{H^{1}}=
\lim_{t\to -\infty}\left\|\psi(t)-e^{it\Delta}\phi_{-} \right\|_{H^{1}}=0.
\end{equation}
\end{theorem}

This paper is organised as follows. In Section \ref{11/12/04/11:44}, we summarise basic properties of the functionals $\mathcal{S}_{\omega}$, $\mathcal{K}$ and so on. In Section \ref{11/05/01/22:50}, we give proofs of Proposition \ref{11/05/01/17:50} and 
Theorem \ref{11/05/01/18:28}. In Section \ref{10/10/04/21:43}, we  introduce Strichartz type spaces and discuss the well-posedness for the equation \eqref{11/06/12/9:08}. In Section \ref{10/11/04/14:21}, we give a long-time perturbation theory which plays an important role to prove the scattering result(Theorem \ref{11/06/10/17:09}). In Section \ref{11/06/10/17:05}, we prove Theorem \ref{11/06/10/17:09} by showing the existence of the so-called critical element in a reductive absurdity. 
\\
\par 
Finally, we give several notation used in this paper:
\\[6pt]
{\bf Notation}. 
\begin{enumerate}
\item
\begin{equation}\label{11/12/02/9:51}
2_{*}:=2+\frac{4}{d},
\qquad 
2^{*}:=2+\frac{4}{d-2}
\end{equation} 
\item
\begin{equation}\label{11/04/30/8:50}
s_{p}:=\frac{d}{2}-\frac{2}{p-1}
\quad 
\mbox{for $p>1$}.
\end{equation}
\item 
We denote the H\"older conjugate of $q \in (1,\infty)$ by $q'$, i.e, $q'=\frac{q}{q-1}$.
\item
Let $A$ and $B$ be two positive quantities. The notation $A\lesssim B$ means that there exists a constant $C>0$ such that $A\le CB$, where 
 $C$ can depend on $d$, $\varepsilon_{0}$ in \eqref{11/04/30/22:39}, $p_{1}$ in \eqref{11/04/30/22:54}, $p_{2}$ in \eqref{11/05/02/8:40} and a given $\omega$. 
\end{enumerate}
\section{Preliminaries}
\label{11/12/04/11:44}
In this section, we give basic properties of functionals $\mathcal{S}_{\omega}$, $\mathcal{K}$ and so on. 
\par 
We first summarize easy fact of calculation, without the proofs:  
\begin{align}
\label{11/06/28/9:59}
&\frac{d}{d\lambda} \mathcal{S}_{\omega}(T_{\lambda}u)
=
\frac{d}{d\lambda} \mathcal{H}(T_{\lambda}u)
=\frac{1}{\lambda}\mathcal{K}(T_{\lambda}u),
\\[6pt]
\label{11/07/20/18:30}
&\frac{d}{d\lambda}\int_{\mathbb{R}^{d}}
\hspace{-3pt}
\big( DF -\alpha F \big)(T_{\lambda}u)
=
\frac{d}{2\lambda}
\int_{\mathbb{R}^{d}}\hspace{-6pt}
\big( D^{2}F -(2+\alpha)DF+ 2\alpha F \big)(T_{\lambda}u)
\quad  \mbox{for any $\alpha \in \mathbb{R}$}.
\end{align}
We also see from \eqref{11/07/20/18:30} together with \eqref{11/04/30/22:39} and \eqref{11/04/30/22:40}  that 
\begin{align}
\label{11/07/20/18:05}
&\frac{d}{d\lambda }\int_{\mathbb{R}^{d}}
\bigm( DF -2 F \bigm)(T_{\lambda}u)
\ge 0, 
\\[6pt]
\label{11/07/20/21:39}
&\frac{d}{d\lambda }\int_{\mathbb{R}^{d}}
\bigm( DF -2_{*} F \bigm)(T_{\lambda}u)
\ge 0,
\\[6pt]
\label{11/09/05/10:10}
&\frac{d}{d\lambda }\left\{ \frac{1}{\lambda} \int_{\mathbb{R}^{d}}
\bigm( DF -2 F \bigm)(T_{\lambda}u) \right\}
\ge 0,
\\[6pt]
\label{11/09/05/10:54} 
&\frac{d}{d\lambda }\left\{ \frac{1}{\lambda^{2}} \int_{\mathbb{R}^{d}}
\bigm( DF -2 F \bigm)(T_{\lambda}u) \right\}
\ge 0.
\end{align}

Next, we give important properties of the functionals $\mathcal{K}$ and  $\mathcal{I}_{\omega}$:
\begin{lemma}\label{11/07/20/18:01}
Assume $d\ge 3$ and conditions \eqref{11/07/18/9:45}--\eqref{11/05/02/8:40}. Then, we have:
\\
{\rm (i)} For any non-trivial function $u \in H^{1}(\mathbb{R}^{d})$, there exists  a unique $\lambda(u)>0$ such that
\begin{equation}
\label{11/05/01/22:06}
\mathcal{K}(T_{\lambda}u)
\left\{
\begin{array}{lcl}
>0 & \mbox{if}& 0< \lambda <\lambda(u),
\\[6pt]
=0 & \mbox{if}& \lambda=\lambda(u),
\\[6pt]
<0 & \mbox{if}&  \lambda >\lambda(u).
\end{array} 
\right.
\end{equation}
{\rm (ii)}
\begin{equation}\label{11/05/01/22:43}
\frac{d}{d\lambda} \mathcal{I}_{\omega}(T_{\lambda}u)>0
\qquad 
\mbox{for any $\lambda>0$}, 
\end{equation}
so that $\mathcal{I}_{\omega}(T_{\lambda}u)$ is monotone increasing with respect to $\lambda>0$.
\\
{\rm (iii)} 
\begin{equation}\label{11/09/01/15:50}
\frac{d^{2}}{d^{2}\lambda} \mathcal{S}_{\omega}(T_{\lambda}u)
\le 
\frac{1}{\lambda^{2}}\mathcal{K}(T_{\lambda}u)
\qquad 
\mbox{for any $\lambda>0$}
\end{equation}
and $\mathcal{S}_{\omega}(T_{\lambda}u)$ is concave on $[\lambda(u),\infty)$.
\end{lemma}
\begin{proof}[Proof of Lemma \ref{11/07/20/18:01}]
We see from \eqref{11/04/30/22:39} and \eqref{11/04/30/22:40} that $\mathcal{K}(T_{\lambda}u)>0$ for any sufficiently small $\lambda>0$. 
We also have   
\begin{equation}\label{11/09/05/10:57}
\mathcal{K}(T_{\lambda}u)\le \lambda^{2} \left\| \nabla u \right\|_{L^{2}}^{2}
- 
\lambda^{2^{*}} \left\|  u \right\|_{L^{2^{*}}}^{2^{*}}
\to -\infty \qquad \mbox{as $\lambda\to \infty$}.
\end{equation}
Hence, there exists $\lambda_{0}>0$ such that $\mathcal{K}(T_{\lambda_{0}}u)=0$. 
Since  \eqref{11/09/05/10:54} shows $\frac{1}{\lambda^{2}}\mathcal{K}(T_{\lambda}u)$ is monotone decreasing with respect to $\lambda$, we find that {\rm (i)} holds.  
\par 
Using \eqref{11/07/20/21:39},  we easily obtain \eqref{11/05/01/22:43}.
\par 
A direct calculation together with \eqref{11/06/28/9:59}, \eqref{11/07/20/18:30} and \eqref{11/04/30/22:40} gives  \eqref{11/09/01/15:50}.  Moreover, \eqref{11/09/01/15:50} together with \eqref{11/05/01/22:06} shows the concavity of $\mathcal{S}_{\omega}(T_{\lambda}u)$.
\end{proof}

\section{Variational problems}
\label{11/05/01/22:50}
In this section, we give the proofs of Proposition \ref{11/05/01/17:50} and Theorem \ref{11/05/01/18:28}. 
\par 
First, we prove Proposition \ref{11/05/01/17:50}: 
\begin{proof}[Proof of Proposition \ref{11/05/01/17:50}]
{\rm (i)} We shall prove $m_{\omega}=\widetilde{m}_{\omega}$. Let $\{u_{n}\}$ be a minimizing sequence of the variational problem for $\widetilde{m}_{\omega}$, i.e., $\{u_{n}\}$ is a sequence in $H^{1}(\mathbb{R}^{d})\setminus \{0\}$ such that 
\begin{align}
\label{11/04/09/16:00}
&\lim_{n\to \infty}\mathcal{I}_{\omega}(u_{n})=\widetilde{m}_{\omega},
\\[6pt]
\label{11/04/09/16:01}
&\mathcal{K}(u_{n})\le 0 
\quad 
\mbox{for any $n \in \mathbb{N}$}.
\end{align} 
Then, it follows from (\ref{11/05/01/22:06}) in Lemma \ref{11/07/20/18:01}  that for each $n \in \mathbb{N}$, there  exists $\lambda_{n}\in (0, 1]$ such that $\mathcal{K}(T_{\lambda_{n}}u_{n})=0$. This together with (\ref{11/05/01/22:43}) leads us to that  
\begin{equation}
\label{11/04/09/16:43}
m_{\omega}
\le 
\mathcal{S}_{\omega}(T_{\lambda_{n}}u_{n})
=
\mathcal{I}_{\omega}(T_{\lambda_{n}}u_{n})
\le 
\mathcal{I}_{\omega}(u_{n})
=
\widetilde{m}_{\omega}+o_{n}(1).
\end{equation} 
Hence, taking $n\to \infty$, we have $m_{\omega} \le \widetilde{m}_{\omega}$. On the other hand, since    
\begin{equation}\label{11/03/04/4:55}
\widetilde{m}_{\omega}
\le \inf_{{u\in H^{1}
\setminus \{0\}}\atop {\mathcal{K}(u)=0}}
\mathcal{I}_{\omega}(u)
=
\inf_{{u\in H^{1}
\setminus \{0\}}\atop {\mathcal{K}(u)=0}}
\mathcal{S}_{\omega}(u)
=
m_{\omega},
\end{equation}
we have $\widetilde{m} \le m$. Hence, it holds that $m_{\omega}=\widetilde{m}_{\omega}$. 
\\[6pt] 
{\rm (ii)} We shall show that any minimizer of the variational problem for $\widetilde{m}_{\omega}$ is also a one for $m_{\omega}$. Let $Q_{\omega}$ be a minimizer for $\widetilde{m}_{\omega}$, i.e., $Q_{\omega} \in H^{1}(\mathbb{R}^{d})\setminus \{0\}$ with $\mathcal{K}(Q_{\omega})\le 0$ and $\mathcal{I}_{\omega}(Q_{\omega})=\widetilde{m}_{\omega}$. Since $\mathcal{I}_{\omega}=\mathcal{S}_{\omega}-\frac{1}{2}\mathcal{K}$ and $m_{\omega}=\widetilde{m}_{\omega}$, it is sufficient to show that $\mathcal{K}(Q_{\omega})=0$. Suppose the contrary that $\mathcal{K}(Q_{\omega})<0$. Then, it follows from (\ref{11/05/01/22:06}) that there exists $0<\lambda_{0}<1$ such that $\mathcal{K}(T_{\lambda_{0}}Q_{\omega})=0$. Moreover, we  see from the definition of $\widetilde{m}_{\omega}$ and (\ref{11/05/01/22:43}) that  
\begin{equation}\label{11/04/30/9:41}
\widetilde{m}_{\omega} \le \mathcal{I}(T_{\lambda_{0}}Q_{\omega})
< \mathcal{I}_{\omega}(Q_{\omega})=\widetilde{m}_{\omega},
\end{equation}
which is a contradiction. Hence, $\mathcal{K}(Q_{\omega})=0$. 
\end{proof}

Next, we give the prove of Theorem \ref{11/05/01/18:28}:  
\begin{proof}[Proof of Theorem \ref{11/05/01/18:28}]
In view of Proposition \ref{11/05/01/17:50} (ii), it is sufficient to show the existence of a minimizer of the variational problem for $\widetilde{m}_{\omega}$. Let $\{u_{n}\}$ be a minimizing sequence for $\widetilde{m}_{\omega}$. We denote the Schwarz symmetrization of $u_{n}$ by $u_{n}^{*}$. Then, we have   
\begin{align}
\label{11/03/04/2:13}
&\mathcal{K}(u_{n}^{*})\le 0
\quad 
\mbox{for any $n \in \mathbb{N}$},
\\[6pt]
\label{11/03/04/1:30}
&\lim_{n\to \infty}\mathcal{I}_{\omega}(u_{n}^{*})=\widetilde{m}_{\omega}.
\end{align}
Besides, extracting some subsequence, we may assume that 
\begin{equation}\label{11/04/10/17:40}
\mathcal{I}_{\omega}(u_{n}^{*}) \le 1+\widetilde{m}_{\omega}
\qquad 
\mbox{for any $n \in \mathbb{N}$},
\end{equation}
which together with (\ref{11/05/01/21:31}) gives us the boundedness of $\{u_{n}^{*}\}$ in $L^{2}(\mathbb{R}^{d})$ and $L^{2^{*}}(\mathbb{R}^{d})$: 
\begin{equation}\label{11/05/01/23:29}
\sup_{n \in \mathbb{N}}\left\| u_{n}^{*} \right\|_{L^{2}}
\lesssim 1,
\qquad 
\sup_{n \in \mathbb{N}}\left\| u_{n}^{*} \right\|_{L^{2^{*}}}
\lesssim 1 .
\end{equation}
Moreover, using the Sobolev embedding, \eqref{11/03/04/2:13},  the boundedness in $L^{2}$ and the growth conditions \eqref{11/04/30/22:54} and \eqref{11/05/02/8:40}, we obtain that 
\begin{equation}\label{11/02/25/7:11}
\begin{split}
\left\| u_{n}^{*} \right\|_{L^{2^{*}}}^{2}
\lesssim 
\left\| \nabla u_{n}^{*} \right\|_{L^{2}}^{2}
\lesssim  
\left\| u_{n}^{*} \right\|_{L^{p_{1}+1}}^{p_{1}+1}
+
\left\| u_{n}^{*} \right\|_{L^{2^{*}}}^{2^{*}}
\lesssim 
\left\| u_{n}^{*} \right\|_{L^{2^{*}}}^{\frac{d(p_{1}-1)}{2}}
+
\left\| u_{n}^{*} \right\|_{L^{2^{*}}}^{2^{*}}.
\end{split}
\end{equation}
This together with \eqref{11/05/01/23:29} shows that 
\begin{equation}\label{11/03/04/1:40}
\sup_{n \in \mathbb{N}}\left\| u_{n}^{*} \right\|_{H^{1}}
\lesssim 1.
\end{equation}

Now, since $\{u_{n}^{*}\}$ is  radially symmetric and bounded in $H^{1}(\mathbb{R}^{d})$, there exists a radially symmetric function $Q \in H^{1}(\mathbb{R}^{d})$ such that 
\begin{align}
\label{11/02/25/7:24}
&\lim_{n\to \infty}u_{n}^{*}=Q 
\quad 
\mbox{weakly in $H^{1}(\mathbb{R}^{d})$},
\\[6pt]
\label{11/04/10/17:52}
&\lim_{n\to \infty}u_{n}^{*}=Q 
\quad 
\mbox{strongly in $L^{q}(\mathbb{R}^{d})$ for $2<q<2^{*}$},
\\[6pt]
\label{11/03/04/1:49}
&\lim_{n\to \infty}u_{n}^{*}(x)=Q(x) 
\quad 
\mbox{for almost all $x \in \mathbb{R}^{d}$},
\\[6pt]
\label{11/09/13/17:02}
&\mathcal{I}_{\omega}(u_{n}^{*})- \mathcal{I}_{\omega}(u_{n}^{*}-Q)
-
\mathcal{I}_{\omega}(Q)
=o_{n}(1),
\\[6pt]
\label{11/09/13/17:03}
&\mathcal{K}(u_{n}^{*})- \mathcal{K}(u_{n}^{*}-Q)
-
\mathcal{K}(Q)
=o_{n}(1).
\end{align}
This function $Q$ is a candidate for the minimizer. 
\par 
We shall first show that $Q$ is non-trivial. Suppose the contrary that $Q$ is trivial. Then, it follows from (\ref{11/03/04/2:13}) and (\ref{11/04/10/17:52}) that 
\begin{equation}\label{11/04/10/18:14}
0\ge \limsup_{n\to \infty}\mathcal{K}(u_{n}^{*})
=2\limsup_{n\to \infty}\left\{ 
\left\| \nabla u_{n}^{*} \right\|_{L^{2}}^{2}
-\left\| u^{*} \right\|_{L^{2^{*}}}^{2^{*}}
\right\},
\end{equation}
so that   
\begin{equation}\label{11/04/10/18:18}
\limsup_{n\to \infty}\left\| \nabla u_{n}^{*} \right\|_{L^{2}}^{2}
\le 
\liminf_{n\to \infty}\left\| u_{n}^{*} \right\|_{L^{2^{*}}}^{2^{*}}.
\end{equation}
Moreover, this together with the definition of $\sigma$ (see (\ref{11/04/09/17:01})) gives us that 
\begin{equation}\label{11/04/10/18:22}
\limsup_{n\to \infty}\left\| \nabla u_{n}^{*} \right\|_{L^{2}}^{2}
\ge 
\sigma 
\liminf_{n\to \infty}\left\| u_{n}^{*} \right\|_{L^{2^{*}}}^{2}
\ge 
\sigma \limsup_{n\to \infty}
\left\| \nabla u_{n}^{*} \right\|_{L^{2}}^{\frac{2(d-2)}{d}},
\end{equation} 
so that 
\begin{equation}\label{11/04/18/8:57}
\sigma^{\frac{d}{2}}
\le \limsup_{n\to \infty}\left\| \nabla u_{n}^{*} \right\|_{L^{2}}^{2}
\le 
\liminf_{n\to \infty}\left\| u_{n}^{*} \right\|_{L^{2^{*}}}^{2^{*}}.
\end{equation}
Hence, if $Q$ is trivial, then we see from  Proposition \ref{11/05/01/17:50}, \eqref{11/04/30/22:39} and \eqref{11/04/18/8:57} that 
\begin{equation}\label{11/09/13/16:25}
m_{\omega}=\widetilde{m}_{\omega}
=
\lim_{n\to \infty}\mathcal{I}_{\omega}(u_{n}^{*})
\ge 
\liminf_{n\to \infty}
\frac{1}{d} \left\| u_{n}^{*} \right\|_{L^{2^{*}}}^{2^{*}}
\ge \frac{1}{d}\sigma^{\frac{d}{2}}.
\end{equation}
However, Lemma \ref{11/05/01/23:05} shows that this is a contradiction. Thus, $Q$ is non-trivial. 
\par 
Next, we shall show that $\mathcal{K}(Q)\le 0$.  Suppose the contrary that $\mathcal{K}(Q)>0$. Then, it follows 
from  \eqref{11/03/04/2:13} and \eqref{11/09/13/17:03} that $\mathcal{K}(u_{n}^{*}-Q)<0$ for any sufficiently large $n \in \mathbb{N}$, so that \eqref{11/05/01/22:06} in Lemma \ref{11/07/20/18:01} shows that there exists a unique $\lambda_{n}\in (0,1)$ such that 
$\mathcal{K}(T_{\lambda_{n}}(u_{n}^{*}-Q))=0$. Then, we see from \eqref{11/05/01/22:43} in Lemma \ref{11/07/20/18:01} and \eqref{11/09/13/17:02} that  
\begin{equation}\label{11/02/25/7:42}
\begin{split}
\widetilde{m}_{\omega}
&\le 
\mathcal{I}_{\omega}(T_{\lambda_{n}}(u_{n}^{*}-Q))
\\[6pt]
&\le \mathcal{I}_{\omega}(u_{n}^{*}-Q)
=\mathcal{I}_{\omega}(u_{n}^{*})
-
\mathcal{I}_{\omega}(Q)+o_{n}(1)
\\[6pt]
&= \widetilde{m}_{\omega} -
\mathcal{I}_{\omega}(Q)+o_{n}(1).
\end{split}
\end{equation}
Hence, we conclude that $\mathcal{I}_{\omega}(Q)=0$. However, this contradicts that $Q$ is non-trivial. Thus, $\mathcal{K}(Q)\le 0$. 
\par 
Since $Q$ is non-trivial and $\mathcal{K}(Q)\le 0$, we see from the definition of $\widetilde{m}_{\omega}$  that 
\begin{equation}\label{11/02/25/7:57}
\widetilde{m}_{\omega}
\le \mathcal{I}_{\omega}(Q). 
\end{equation}
Moreover, it follows from \eqref{11/09/13/17:02} and \eqref{11/03/04/1:30} that 
\begin{equation}\label{11/04/10/21:53}
\mathcal{I}_{\omega}(Q) \le \liminf_{n\to \infty}\mathcal{I}_{\omega}(u_{n}^{*})
\le \widetilde{m}_{\omega}.
\end{equation} 
Combining (\ref{11/02/25/7:57}) and (\ref{11/04/10/21:53}), we obtain that $\mathcal{I}_{\omega}(Q)=m_{\omega}$. In particular, we have $m_{\omega}=\widetilde{m}_{\omega}>0$. Thus, we have completed the proof.  
\end{proof}
\section{Well-posedness}
\label{10/10/04/21:43}
In this section, we discuss the local well-posedness result in the energy critical case (see \cite{Cazenave-Weissler1, Kato1995, Tao-Visan}). \par 
Let us begin with the notion of admissible pairs: A pair of space-time indices $(q,r) \in [2,\infty]\times [2,\infty]$ is said to be $L^{2}$-admissible, if $\frac{1}{r}=\frac{d}{2}\left( \frac{1}{2}-\frac{1}{q} \right)$. In particular, $(2,\infty)$, $(\frac{2(d+2)}{d},\frac{2(d+2)}{d})$, $(\frac{2d(d+2)}{d^{2}+4}, \frac{2(d+2)}{d-2})$, $(\frac{2(d+2)(p-1)}{d(d+2)(p-1)-8}, \, 
\frac{(d+2)(p-1)}{2})$ and $(2^{*},2)$ are $L^{2}$-admissible pairs. 
 For any $L^{2}$-admissible pair $(q,r)$, we have 
\begin{equation}
\label{10/10/29/18:07}
\left\| e^{it\Delta}u \right\|_{L^{r}(\mathbb{R},L^{q})}
\lesssim 
\left\| u \right\|_{L^{2}},
\end{equation}
and for any $L^{2}$-admissible pairs $(q_{1},r_{1})$ and $(q_{2},r_{2})$,  
\begin{equation}
\label{10/10/29/18:11}
\begin{split}
&
\left\| \int_{t_{0}}^{t}
e^{i(t-t')\Delta}v(t')\,dt' \right\|_{L^{r_{1}}(\mathbb{R},L^{q_{1}})}\lesssim 
\left\| v \right\|_{L^{r_{2}'}(\mathbb{R},L^{q_{2}'})}
\quad 
\mbox{for any $t_{0} \in \mathbb{R}$}.
\end{split}
\end{equation}
These estimates are called the Strichartz estimates. 
\par 
Generally,  a pair $(q,r) \in [2,\infty]\times [2,\infty]$ is called $\dot{H}^{s}$-admissible, if $\frac{1}{r}=\frac{d}{2}\left( \frac{1}{2}-\frac{1}{q}-\frac{s}{d} \right)$ (cf. \cite{Holmer-Roudenko, Keraani}).  In particular, the pair $(\frac{(d+2)(p-1)}{2}, \frac{(d+2)(p-1)}{2})$ is the diagonal $\dot{H}^{s_{p}}$-admissible pair. 
\par 
In order to mention the local well-posedness result, we need to introduce several space-time function spaces. Let $I$ be an interval. Then, we introduce Strichartz-type function spaces:    
\begin{align}
\label{11/05/12/17:22}
S(I)
&:=
L^{\infty}(I,L^{2})\cap L^{2}(I,L^{2^{*}}),
\\[6pt]
\label{11/05/27/21;34}
V_{p}(I)
&:=
L^{\frac{(d+2)(p-1)}{2}}(I, L^{\frac{2(d+2)(p-1)}{d(d+2)(p-1)-8}}),
\\[6pt]
\label{11/05/14/17:34}
W_{p}(I)&
:=L^{\frac{(d+2)(p-1)}{2}}(I,L^{\frac{(d+2)(p-1)}{2}}),
\\[6pt]
\label{11/05/12/18:34}
V(I)
&:=V_{1+\frac{4}{d-2}}(I)
=
L^{\frac{2(d+2)}{d-2}}(I, L^{\frac{2d(d+2)}{d^{2}+4}}),
\\[6pt]
\label{11/05/12/17:23} 
W(I)&:=W_{1+\frac{4}{d-2}}(I)=L^{\frac{2(d+2)}{d-2}}(I,L^{\frac{2(d+2)}{d-2}}).
\end{align}

\begin{figure}[h]
\begin{center}
\input{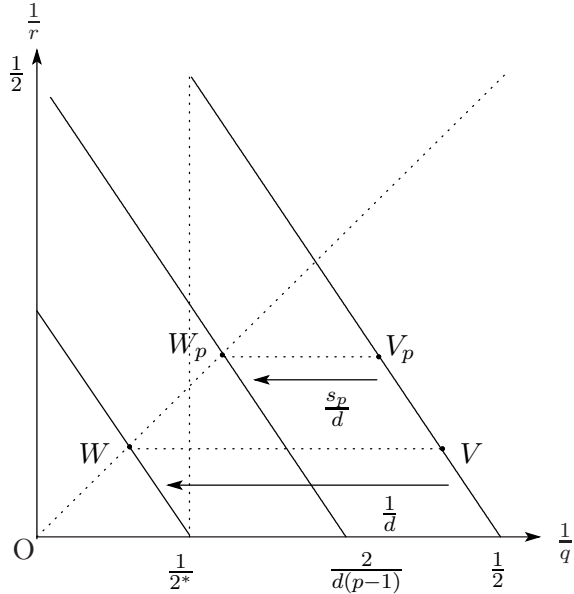}
\end{center}
\caption{Strichartz type spaces}
\end{figure}

It is worthwhile noting that for any $1+\frac{4}{d}<p\le q\le 1+ \frac{4}{d-2}$, we have
\begin{equation}
\label{11/05/11/23:53}
\left\| |\nabla|^{s} \left( |u|^{q-1}u\right) \right\|_{L^{\frac{(d+2)(p-1)}{2p}}(I,L^{\frac{2d(d+2)(p-1)}{d(d+6)(p-1)-8}})}
\lesssim 
\left\| |\nabla|^{s} u \right\|_{V_{p}(I)}
\left\| u \right\|_{W_{q}(I)}^{q-1}
\quad 
\mbox{for $s=0,1$}
\end{equation}
and 
\begin{equation}
\label{11/06/01/18:11}
\left\| u \right\|_{W_{q}(I)}
\lesssim 
\left\| |\nabla|^{s_{q}} u \right\|_{V_{q}(I)}
\le \left\| |\nabla|^{s_{q}} u \right\|_{L^{\infty}(I,L^{2})}^{1-\frac{p-1}{q-1}}
\left\| |\nabla|^{s_{q}} u \right\|_{V_{p}(I)}^{\frac{p-1}{q-1}}.
\end{equation}
Here, the H\"older conjugate of $\bigm(\frac{2d(d+2)(p-1)}{d(d+6)(p-1)-8}, \frac{(d+2)(p-1)}{2p}\bigm)$ is $L^{2}$-admissible. 
\begin{theorem}[Well-posedness in $H^{1}$]
\label{10/10/04/21:44}
Assume $d\ge 3$ and \eqref{11/07/18/9:45}--\eqref{11/05/02/8:40}. Then, we have:
\begin{enumerate}
\item[{\rm (i)}] Let $\psi_{0} \in H^{1}(\mathbb{R}^{d})$, $I$ be an interval, $t_{0}\in I$ and $A>0$. Assume that 
\begin{equation}\label{11/06/01/14:07}
\left\| \psi_{0} \right\|_{H^{1}}\le A.
\end{equation}
Then, there exists $\delta>0$ depending on $A$ with the following property; If  
\begin{equation}
\label{11/05/15/20:19}
\left\| \langle \nabla \rangle e^{i(t-t_{0})\Delta}\psi_{0} 
\right\|_{V_{p_{1}}(I)}
\le \delta,
\end{equation} 
then there exists a solution $\psi \in C(I,H^{1}(\mathbb{R}^{d}))$ to  (\ref{11/06/12/9:08}) such that 
\begin{align}
\label{11/05/29/11:29}
&\psi(t_{0})=\psi_{0}, 
\\[6pt]
\label{11/05/15/20:31}
&\left\| \langle \nabla \rangle \psi \right\|_{S(I)}
\lesssim 
\left\| \psi_{0} \right\|_{H^{1}}, 
\\[6pt]
\label{11/05/15/20:30}
&\left\| \langle \nabla \rangle \psi \right\|_{V_{p_{1}}(I)}
\le 2 
\left\| \langle \nabla \rangle e^{i(t-t_{0})\Delta}\psi_{0} 
\right\|_{V_{p_{1}}(I)}.
\end{align}

\item[{\rm (ii)}] Let $I$ be an interval. Suppose that $\psi_{1}, \psi_{2} \in C(I,H^{1}(\mathbb{R}^{d}))$ are two solutions of 
(\ref{11/06/12/9:08}) with $\psi_{1}(t_{0})=\psi_{2}(t_{0})$ for some $t_{0} \in I$. Then, $\psi_{1}=\psi_{2}$. 
\end{enumerate}

Furthermore, let $\psi \in C(I_{\max},H^{1}(\mathbb{R}^{d}))$ be a solution to (\ref{11/06/12/9:08}), where $I_{\max}$ is the maximal time interval on which the solution $\psi$ exists (Note here that by {\rm (i)}, $I_{\max}$ must be open). Then, we have: 
\begin{enumerate}
\item[{\rm (iii)}] The following conservation laws hold; For any $t, t_{0} \in I_{\max}$, 
\begin{align}
\label{10/08/31/15:26}
\mathcal{M}(\psi(t))
&
:=
\left\| \psi(t) \right\|_{L^{2}}^{2}
=\mathcal{M}(\psi(t_{0})),
\\[6pt]
\label{10/08/31/15:33}
\mathcal{H}(\psi(t))
&
=\mathcal{H}(\psi(t_{0})),
\\[6pt]
\label{10/09/18/15:50}
\mathcal{S}_{\omega}(\psi(t))
&=\mathcal{S}_{\omega}(\psi(t_{0}))
\qquad 
\mbox{for any $\omega \in \mathbb{R}$}, 
\\[6pt]
\label{10/08/31/15:34}
\mathcal{P}(\psi(t))
&
:=
\Im \int_{\mathbb{R}^{d}}\nabla \psi(x,t)\overline{\psi(x,t)}\,dx
=\mathcal{P}(\psi(t_{0})).
\end{align}

\item[{\rm (iv)}]
If $T_{\max}:=\sup{I_{\max}} < +\infty$, then    
\begin{equation}
\label{10/10/04/22:42}
\left\| \psi \right\|_{
W_{p_{1}}([T,\, T_{\max}))
\cap W([T,\, T_{\max}))
}= \infty
\qquad 
\mbox{for any $T \in I_{\max}$}.
\end{equation}
A similar result holds, when $T_{\min}:=\inf{I_{\max}} > -\infty$.
\item[{\rm (v)}] If 
\begin{equation}
\label{11/05/15/20:47}
\left\| \psi \right\|_{W_{p_{1}}(I_{\max})\cap W(I_{\max})}<\infty, 
\end{equation}  
then $T_{\max}=+\infty$, $T_{\min}=-\infty$, 
 and there exist $\phi_{+}, \phi_{-} \in H^{1}(\mathbb{R}^{d})$ such that 
\begin{equation}\label{11/05/15/20:49}
\lim_{t\to + \infty}
\left\| \psi(t) -e^{it\Delta}\phi_{+} \right\|_{H^{1}}
=
\lim_{t\to -\infty}
\left\| \psi(t) -e^{it\Delta}\phi_{-} \right\|_{H^{1}}
=
0.
\end{equation}
\end{enumerate}
\end{theorem}

\begin{proof}[Proof of Theorem \ref{10/10/04/21:44}] 
We first prove (i). Let $\delta >0$ be a sufficiently small constant specified later; In particular, we take $\delta<1$ . Suppose that 
\begin{equation}\label{11/06/01/14:22}
\left\| \langle \nabla \rangle e^{i(t-t_{0})\Delta} \psi_{0} 
\right\|_{V_{p_{1}}(I)}\le \delta .
\end{equation}
We see from the Strichartz estimate that there exists $C_{st}>0$ depending only on $d$ such that 
\begin{equation}\label{11/05/27/17:55}
\left\| e^{it\Delta} u \right\|_{S(\mathbb{R})}
\le C_{st}\left\| u \right\|_{L^{2}}.
\end{equation} 
We define the space $Y(I)$ and the map $\mathcal{T}$ by 
\begin{align}
\label{11/05/27/17:57}
&Y(I)
:=\left\{ u \in C(I,H^{1}(\mathbb{R}^{d}))
\bigg| 
\begin{array}{l}
\left\| \langle \nabla \rangle u \right\|_{V_{p_{1}}(I)}
\le 
2\left\| \langle \nabla \rangle e^{i(t-t_{0})\Delta} \psi_{0} 
\right\|_{V_{p_{1}}(I)},
\\[6pt]
\left\| \langle \nabla \rangle u \right\|_{S(I)}
\le 
2C_{St} \left\| \psi_{0} \right\|_{H^{1}}
\end{array}
\right\},
\\[6pt]
\label{11/05/27/18:25}
&\mathcal{T}(u)
:=e^{i(t-t_{0})\Delta}\psi_{0}
-
i\int_{t_{0}}^{t}e^{i(t-t')\Delta}
\bigm\{ f(u)+|u|^{2^{*}-2}u \bigm\}(t')\,dt'.
\end{align}
We can verify that $Y(I)$ becomes a metric space with the metric $\rho(u,v):=\left\| u-v \right\|_{S(I)}$. 
\par 
We shall show that $\mathcal{T}$ maps $Y(I)$ into itself for sufficiently small $\delta$.  The Strichartz estimate, together with (\ref{11/04/30/22:54}), (\ref{11/05/02/8:40}), (\ref{11/05/11/23:53}), (\ref{11/06/01/18:11}) and (\ref{11/06/01/14:22}), gives us that 
\begin{equation}\label{11/05/27/21:05}
\begin{split}
\left\| \langle \nabla \rangle \mathcal{T}(u) \right\|_{V_{p_{1}}(I)}
&\le 
\left\| \langle \nabla \rangle e^{i(t-t_{0})\Delta} \psi_{0}  
\right\|_{V_{p_{1}}(I)}
\\[6pt]
&\qquad 
+
C \left( 
1
+
\left\| u \right\|_{L^{\infty}(I,H^{1})}^{p_{2}-p_{1}}
+
\left\| u \right\|_{L^{\infty}(I,H^{1})}^{2^{*}-(p_{1}+1)}
\right)
\left\| \langle \nabla \rangle u \right\|_{V_{p_{1}}(I)}^{p_{1}}
\\[6pt]
&\le  
\left\| \langle \nabla \rangle e^{i(t-t_{0})\Delta} \psi_{0}  
\right\|_{V_{p_{1}}(I)}
\\[6pt]
&\qquad 
+
C(1+A^{p_{2}-p_{1}}+A^{2^{*}-(p_{1}+1)}) \delta^{p_{1}-1}\left\| \langle \nabla \rangle e^{i(t-t_{0})\Delta} \psi_{0}  \right\|_{V_{p_{1}}(I)},
\end{split}
\end{equation}
where $C$ is some universal constant. Hence, taking $\delta$ sufficiently small depending on $A$, we obtain 
\begin{equation}\label{11/06/01/14:15}
\left\| \langle \nabla \rangle \mathcal{T}(u) \right\|_{V_{p_{1}}(I)}
\le 2\left\| \langle \nabla \rangle e^{i(t-t_{0})\Delta} \psi_{0}  
\right\|_{V_{p_{1}}(I)}. 
\end{equation}
Similarly, taking $\delta$ sufficiently small depending on $A$, we have  
\begin{equation}\label{11/05/27/22:09}
\left\| \langle \nabla \rangle \mathcal{T}(u) \right\|_{S(I)}
\le 2C_{St}
\left\| \psi_{0}  \right\|_{H^{1}}.
\end{equation}

Next, we shall show that $\mathcal{T}$ is contraction in $(Y(I),\rho)$ for sufficiently small $\delta>0$.  Take any $u,v \in Y(I)$.  Then, we have in a way similar to the estimate \eqref{11/05/27/21:05} that 
\begin{equation}\label{11/06/01/14:26}
\begin{split}
&\rho(\mathcal{T}(u), \mathcal{T}(v))=\left\| \mathcal{T}(u)-\mathcal{T}(v) \right\|_{S(I)}
\\[6pt]
&\lesssim 
(1+A^{p_{2}-p_{1}}+A^{2^{*}-(p_{1}+1)})
\delta^{p_{1}-1}
\left\| u-v \right\|_{V_{p_{1}}(I)},
\end{split} 
\end{equation}
which shows that if $\delta$ is sufficiently small depending on $A$, then $\mathcal{T}$ is contraction. 
\par 
Thus, the claim (i) follows from the contraction mapping principle in $(Y(I),\rho)$.  
\par 
We omit the proof of (ii) and (iii). 
\par 
We prove (iv) by contradiction. Assume $T_{\max}<+\infty$, and suppose the contrary that    
\begin{equation}\label{11/05/13/18:01}
\left\| \psi \right\|_{W_{p_{1}}([T_{0},\,T_{\max}))
\cap W([T_{0},\,T_{\max}))}< \infty
\quad 
\mbox{for some $T_{0} \in I_{\max}$}.
\end{equation}
In particular, we have 
\begin{equation}\label{11/05/13/18:04}
\lim_{T\to T_{\max}}\left\| \psi \right\|_{W_{p_{1}}([T,\,T_{\max}))
\cap W_{p_{2}}([T,\,T_{\max})) \cap W([T,\,T_{\max}))}=0.
\end{equation}
In view of (i) in this theorem, it suffices to show that  $\displaystyle{\lim_{t\to T_{\max}}\psi(t)}$ exists strongly in $H^{1}(\mathbb{R}^{d})$. We further reduce this to proving that for any sequence $\{t_{n}\}$ in $[T_{0},T_{\max})$ with $\displaystyle{\lim_{n\to \infty}t_{n}=T_{\max}}$, $\{e^{-it_{n}\Delta}\psi(t_{n})\}$ is Cauchy sequence in $H^{1}(\mathbb{R}^{d})$. Let us prove this. Put $p_{3}:=1+\frac{4}{d-2}$ as well as the above. 
The Strichartz estimate together with (\ref{11/05/11/23:53}) yields that 
\begin{equation}\label{11/05/11/21:36}
\begin{split}
&
\left\| 
e^{-it_{n}\Delta}\psi(t_{n})
-e^{-it_{m}\Delta}\psi(t_{m}) 
\right\|_{H^{1}}
\\[6pt]
&\le
\sup_{t \in [t_{m},t_{n}]}
\left\|
\int_{t_{m}}^{t}
e^{-it'\Delta}
\bigm\{ 
f(\psi) + |\psi|^{2^{*}-2}\psi 
\bigm\}(t')\,dt'
\right\|_{H^{1}}
\\[6pt]
&\lesssim 
\left\| \langle \nabla \rangle 
\psi 
\right\|_{V_{p_{1}}([t_{m},\, T_{\max}))} 
\sum_{j=1}^{3}
\left\| 
 \psi 
\right\|_{W_{p_{j}}([t_{m},\, T_{\max}))}^{p_{j}-1}.
\end{split}
\end{equation}
We see from (\ref{11/05/13/18:04}) and (\ref{11/05/11/21:36}) that $\{e^{-it_{n}\Delta}\psi(t_{n})\}$ is Cauchy sequence in $H^{1}(\mathbb{R}^{d})$, provided that   
\begin{equation}\label{11/05/12/18:07}
\left\| \langle \nabla \rangle \psi \right\|_{
V_{p_{1}}([T,\, T_{\max}))
}<\infty
\quad 
\mbox{for some $T\in [T_{0},T_{\max})$}. 
\end{equation}
We shall prove (\ref{11/05/12/18:07}). By (\ref{11/05/13/18:04}), we can take $T_{1} \in [T_{0}, T_{\max})$ such that 
\begin{equation}\label{11/05/13/18:12}
\sum_{j=1}^{3}
\left\| 
 \psi 
\right\|_{W_{p_{j}}([T_{1},\, T_{\max}))}^{p_{j}-1}
\ll 1.
\end{equation} 
Then, an estimate similar to (\ref{11/05/11/21:36}) gives us that 
\begin{equation}\label{11/05/12/18:16}
\left\|\langle \nabla \rangle \psi \right\|_{
V_{p_{1}}([T_{1},\, T_{\max}))}
\le  
C \left\|  \psi(T_{1}) \right\|_{H^{1}}
+
\frac{1}{2}
\left\| \langle \nabla \rangle 
 \psi 
\right\|_{V_{p_{1}}([T_{1},\, T_{\max}))}
\end{equation}
for some universal constant $C>0$, from which we immediately obtain the desired result (\ref{11/05/12/18:07}). 
\par 
Next, we prove (v). We see from the contraposition of (iv) that $T_{\max}=+\infty$ and $T_{\min}=-\infty$. Moreover, an argument similar to  the proof of (iv) shows that there exists $\phi_{+}, \phi_{-} \in H^{1}(\mathbb{R}^{d})$ such that (\ref{11/05/15/20:49}) holds. 
\end{proof}

\begin{proposition}[see \cite{Cazenave-Weissler1}]
\label{11/07/13/16:27}
Assume $d\ge 3$. Let $\phi \in \dot{H}^{1}(\mathbb{R}^{d})$, let $t_{0}\in \mathbb{R}\cup \{\pm \infty\}$ and let $A>0$. Assume that 
\begin{equation}\label{11/06/01/14:07}
\left\| \nabla \phi \right\|_{L^{2}}\le A.
\end{equation}
Then, there exists $\delta>0$ depending on $A$ with the following property; If  
\begin{equation}
\label{11/08/14/15:28}
\left\| \nabla \phi 
\right\|_{L^{2}(\mathbb{R})}
\le \delta,
\end{equation} 
then there exists a global solution $\psi \in C(\mathbb{R},\dot{H}^{1}(\mathbb{R}^{d}))$ to  (\ref{11/07/17/16:01}) such that 
\begin{align}
\label{11/07/13/16:35}
&\left\{ 
\begin{array}{ccc}
\psi(t_{0})=\phi &\mbox{if }& t_{0} \in \mathbb{R}, 
\\[6pt]
\displaystyle{
\lim_{t\to t_{0}}\left\| \psi(t) -e^{it\Delta}\phi \right\|_{H^{1}}=0
}
&\mbox{if}& t_{0} \in \{\pm \infty\}, 
\end{array}
\right.
\\[6pt]
\label{11/07/13/16:34}
&\left\| \nabla \psi \right\|_{S(\mathbb{R})}
\lesssim 
\left\| \nabla \phi 
\right\|_{L^{2}}.
\end{align}
\end{proposition}

We can control the Strichartz spaces by a few spaces. Indeed, we have: \begin{lemma}\label{11/05/14/14:42}
Assume $d\ge 3$. Let $I$ be an interval, $A,B>0$, and
 let $u$ be a function such that 
\begin{equation}\label{11/05/14/14:46}
\left\| u \right\|_{L^{\infty}(I,H^{1})}\le A,
\qquad 
\left\| u \right\|_{W_{p_{1}}(I) \cap W(I)}\le B.
\end{equation}
Let $\varepsilon>0$ and suppose that  
\begin{equation}\label{11/05/27/4:09}
\left\|\langle \nabla \rangle \left( 
i\frac{\partial }{\partial t}u +\Delta u + f(u)+|u|^{2^{*}-2}u
\right) 
\right\|_{L^{\frac{2(d+2)}{d+4}}(I,L^{\frac{2(d+2)}{d+4}})}
\le \varepsilon.
\end{equation} 
Then, we have 
\begin{equation}\label{11/05/14/14:47}
\left\| \langle \nabla \rangle u \right\|_{S(I)}
\lesssim C(A,B) + \varepsilon,
\end{equation}
where $C(A,B)$ is some constant depending on $A$ and $B$. 
\end{lemma}
\begin{proof}[Proof of Lemma \ref{11/05/14/14:47}]
We rewrite the function $u$ by  
\begin{equation}\label{10/10/31/22:18}
u(t)=e^{i(t-t_{0})\Delta}u(t_{0})
+i
\int_{t_{0}}^{t}e^{i(t-t')\Delta}
\left(2i\frac{\partial u}{\partial t}+\Delta u\right)(t')\,dt
\quad 
\mbox{for any $t_{0}\in I$}, 
\end{equation}
where the equality is taken in the weak sense. Then, the Strichartz estimate together with (\ref{11/05/11/23:53}) and (\ref{11/05/27/4:09}) gives us that 
\begin{equation}\label{11/05/14/14:54}
\begin{split}
\left\|\langle \nabla \rangle u \right\|_{S(I)}
&\lesssim 
\left\| u(t_{0}) \right\|_{H^{1}}
+
\sum_{k=1}^{2}
\left\|  \langle \nabla \rangle u \right\|_{V_{p_{k}}(I)}
\left\| u \right\|_{W_{p_{k}}(I)}^{p_{k}-1}
+
\left\|  \langle \nabla \rangle u \right\|_{V(I)}
\left\| u \right\|_{W(I)}^{\frac{4}{d-2}}
+
\varepsilon
\\[6pt] 
&\le 
A
+
\sum_{j=1}^{2}
\left\| \langle \nabla \rangle u \right\|_{V_{p_{k}}(I)}
B^{p_{k}-1}
+
\left\| \langle \nabla \rangle u \right\|_{V(I)}
B^{\frac{4}{d-2}}
+\varepsilon
.
\end{split} 
\end{equation}
Here, we see from the H\"older inequality that  
\begin{equation}
\label{11/05/14/15:12}
\begin{split}
\left\| \langle \nabla \rangle  u \right\|_{V_{p}(I)}
&\le 
\left\| u 
\right\|_{L^{\infty}(I,H^{1})}^{1-\frac{4}{(d+2)(p-1)}}
\left\| \langle \nabla \rangle u 
\right\|_{L^{2}(I,L^{2^{*}})}^{\frac{4}{(d+2)(p-1)}}
\\[6pt]
&\le 
A^{1-\frac{4}{(d+2)(p-1)}}
\left\|\langle \nabla \rangle  u \right\|_{S(I)}^{\frac{4}{(d+2)(p-1)}}
\qquad 
\mbox{for any $1+\frac{4}{d}<p\le 1+\frac{4}{d-2}$}. 
\end{split}
\end{equation}
Combining (\ref{11/05/14/14:54}) with (\ref{11/05/14/15:12}), we easily verify that the desired result (\ref{11/05/14/14:47}) holds.
\end{proof}

\section{Perturbation Theory}
\label{10/11/04/14:21}
The derivative of our nonlinearity is no longer Lipschitz 
 continuous when $d \ge 5$; It is merely H\"older continuous of order $p-1>\frac{4}{d}$. Thus, to establish the long-time perturbation theory (see Propositions \ref{11/05/14/12:26}), we need some idea. We will employ the exotic Strichartz estimate (see  \cite{Foschi} and  \cite{Tao-Visan}). 
\par 
Assume that $d\ge 5$. We define $\alpha$ by  $s_{\alpha}=1-\frac{4}{d}s_{p_{1}}$, i.e.,  
\begin{equation}
\label{11/05/21/11:26}
\alpha:=1+\frac{4d(p_{1}-1)}{d(d+2)(p_{1}-1)-16}
\in 
\biggm(1+\frac{4d}{d^{2}-2d+8},\ 1+\frac{4}{d-2} \biggm)
.
\end{equation}
Let  $(\rho, \gamma)$ and $(\rho^{*},\gamma^{*})$ be the pairs such that  
\begin{align}
\rho&:=\frac{\alpha+1+2^{*}}{2},
\qquad 
\frac{1}{\gamma}
=\frac{d}{2}\left( \frac{1}{2}-\frac{1}{\rho}
-\frac{s_{\alpha}}{d}\right), 
\\[6pt]
\label{11/05/21/15:35}
\rho^{*}&:=\rho',
\qquad 
\frac{1}{r^{*}}=1-\frac{d}{2}\left( \frac{1}{2}-\frac{1}{\rho}
+\frac{s_{\alpha}}{d}
\right).
\end{align}
Note here that $\alpha$ and $\rho$ are monotone decreasing 
 with respect to $p_{1}$. 
\par 
Our exotic Strichartz norms are as follows\footnote{When $d\le 4$, we take $\left\| u \right\|_{ES(I)}=\left\| \nabla u \right\|_{V(I)}$ and 
$\left\| u \right\|_{ES^{*}(I)}=\left\| \nabla u \right\|_{L^{\frac{2(d+2)}{d+4}}(I,L^{\frac{2(d+2)}{d+4}})}$.}:  
\begin{align}
\label{11/05/21/12:00}
\left\| u \right\|_{ES(I)}
&:=
\left\| |\nabla|^{\frac{4}{d}s_{p_{1}}}u \right\|_{
L^{\gamma}(I,L^{\rho})}, 
\\[6pt] 
\label{11/05/21/14:05}
\left\| u \right\|_{ES^{*}(I)}
&:=
\left\| |\nabla|^{\frac{4}{d}s_{p_{1}}}u \right\|_{
L^{\gamma^{*}}(I,L^{\rho^{*}})}
.
\end{align}
Since $(\rho, \gamma)$ are $\dot{H}^{1-\frac{4}{d}s_{p_{1}}}$-admissible ($\dot{H}^{s_{\alpha}}$-admissible) with $\frac{(d+2)(p_{1}-1)}{2}<\gamma <\infty$ ($d\ge 5$), we see that 
\begin{align}
\label{11/06/02/22:24}
&\left\| u \right\|_{ES(I)}
\lesssim \left\| u \right\|_{L^{\infty}(I,H^{1})}^{1-\theta_{ES}}
\left\| \langle \nabla \rangle u \right\|_{V_{p_{1}}(I)}^{\theta_{ES}}
\qquad 
\mbox{for some $0< \theta_{ES} <1$},
\\[6pt]
\label{11/06/03/9:59}
&\left\| u \right\|_{L^{\gamma}(I,L^{\frac{d^{2}\rho}{d^{2}+d\rho-4\rho s_{p_{1}}}})}
\lesssim 
\left\| u \right\|_{ES(I)},
\end{align}
where $(\frac{d^{2}\rho}{d^{2}+d\rho-4\rho s_{p_{1}}},\gamma)$ is 
$L^{2}$-admissible.
\begin{lemma}[Exotic Strichartz estimate]
\label{10/10/16/22:32}
Let $I$ be an interval, $t_{0} \in I$, and let $u$ be a function on $\mathbb{R}^{d}\times I$. Then, we have 
\begin{equation}\label{10/10/16/15:57}
\left\| u \right\|_{ES(I)}
\le   
\left\| e^{i(t-t_{0})\Delta}u(t_{0}) \right\|_{ES(I)}
+
C
\left\| i\frac{\partial u}{\partial t}+\Delta u \right\|_{ES^{*}(I)}
\end{equation}
for some constant $C>0$ depending only on $d$. 
\end{lemma}
\begin{proof}[Proof of Lemma \ref{10/10/16/22:32}]
We can write any function $u$ by  
\begin{equation}\label{11/07/30/11:16}
u(t)=e^{i(t-t_{0})\Delta}u(t_{0})
+i
\int_{t_{0}}^{t}e^{i(t-t')\Delta}
\left(i\frac{\partial u}{\partial t}+\Delta u\right)(t')\,dt,
\end{equation}
where the equality is taken in the weak sense. Then, the claim follows from the inhomogeneous Strichartz estimate by Foschi \cite{Foschi}. 
\end{proof}

In order to treat the fractional differential operator $|\nabla|^{\frac{4}{d}s_{p_{1}}}$ in the exotic Strichartz norms, we employ the following estimates:
\begin{lemma}[see \cite{Christ-Weinstein}]\label{10/10/29/18:36}
Let $s \in (0,1]$ and let $1<q,q_{1},q_{2},q_{3},q_{4}<\infty$ with  $\frac{1}{q}=\frac{1}{q_{1}}+\frac{1}{q_{2}}=\frac{1}{q_{3}}+\frac{1}{q_{4}}$. Then, we have  
\begin{equation}\label{10/10/29/18:38}
\left\| |\nabla|^{s}(uv)\right\|_{L^{q}}
\lesssim 
\left\| u\right\|_{L^{q_{1}}}
\left\| |\nabla|^{s}v \right\|_{L^{q_{2}}}
+
\left\||\nabla|^{s}u \right\|_{L^{q_{3}}}
\left\| v \right\|_{L^{q_{4}}}.
\end{equation}
\end{lemma}

\begin{lemma}[see \cite{Visan}]
\label{10/10/29/18:41}
Let $h\colon \mathbb{C}\to \mathbb{C}$ be a H\"older continuous function of order $\alpha \in (0,1)$. Then, for any $s \in (0,\alpha)$, $q \in (1,\infty)$ and $\sigma \in (\frac{s}{\alpha},1)$, we have  
\begin{equation}\label{10/10/29/18:44}
\left\| |\nabla|^{s}h(u)\right\|_{L^{q}}
\lesssim 
\left\| u \right\|_{L^{\left( \alpha-\frac{s}{\sigma}\right)q_{1}}}^{\alpha-\frac{s}{\sigma}}
\left\| |\nabla|^{\sigma} u\right\|_{L^{\frac{s}{\sigma}q_{2}}}^{\frac{s}{\sigma}},
\end{equation}
provided that $\frac{1}{q}=\frac{1}{q_{1}}+\frac{1}{q_{2}}$ and 
$(\alpha-\frac{s}{\sigma})q_{1}>\alpha$. 
\end{lemma}

\begin{lemma}\label{11/05/16/0:44}
Assume that $d\ge 5$. Let $h_{1}$ and $h_{*}$ be H\"older continuous functions of order $p_{1}-1$ and $\frac{4}{d-2}$, respectively. Let $I$ be an interval. 
 Then, we have 
\begin{align}
&\label{11/05/17/18:45}
\left\| h_{1}(v) w \right\|_{ES^{*}(I)}
\lesssim 
\left\| |\nabla|^{s_{p_{1}}}v \right\|_{L^{r_{0}}(I,L^{q_{0}})
\cap L^{\widetilde{r}_{0}}(I,L^{\widetilde{q}_{0}})}^{p_{1}-1}
\left\| w \right\|_{ES(I)},
\\[6pt]
\label{11/05/21/22:32}
&\left\| h_{*}(v) w \right\|_{ES^{*}(I)}
\lesssim 
\left\| \nabla v \right\|_{L^{r_{0}}(I,L^{q_{0}})\cap L^{\widetilde{r}_{0}}(I,L^{\widetilde{q}_{0}})}^{\frac{4}{d-2}}
\left\| w \right\|_{ES(I)},
\end{align}
where $(q_{0},r_{0})$ and $(\widetilde{q}_{0},\widetilde{r}_{0})$ are some $L^{2}$-admissible pairs with $\frac{(d+2)(p_{1}-1)}{2}<r_{0}, \widetilde{r}_{0}<\infty$. 
\end{lemma}
\begin{proof}[Proof of Lemma \ref{11/05/16/0:44}]
Using Lemma \ref{10/10/29/18:36},  we obtain that 
\begin{equation}\label{10/10/30/23:39}
\begin{split}
&\left\|h_{1}(v) w \right\|_{ES^{*}(I)}
\\[6pt]
&\lesssim 
\left\|  h_{1}(v) \right\|_{
L^{\frac{2\rho}{2d-(d-2)\rho}}(I,L^{\frac{\rho}{\rho-2}})}
\left\|  w \right\|_{ES(I)}
\\[6pt]
&\quad  
+
\left\| |\nabla|^{\frac{4}{d}s_{p_{1}}} h_{1}(v)  
\right\|_{
L^{\frac{2\rho}{2d-(d-2)\rho}}
(I,L^{\frac{d^{2}(p_{1}-1)\rho}{d(d+2)(p_{1}-1)\rho-8\rho-2d^{2}(p_{1}-1)}
})}
\left\| w \right\|_{L^{\gamma}(I,L^{\frac{d^{2}\rho}{d^{2}-4s_{p_{1}}\rho}})}.
\end{split}
\end{equation}
Since $\dot{W}^{\frac{4}{d}s_{p_{1}},\rho}\hookrightarrow 
L^{\frac{d^{2}\rho}{d^{2}-4s_{p_{1}}\rho}}$, we further obtain that 
\begin{equation}\label{11/05/21/17:22}
\begin{split}
\left\|h_{1}(v) w \right\|_{ES^{*}(I)}
&\lesssim 
\left\|  v 
\right\|_{L^{ \frac{2(p_{1}-1)\rho}{2d-(d-2)\rho} }
(I,L^{\frac{(p_{1}-1)\rho}{\rho-2}})}^{p_{1}-1} \left\|  w \right\|_{ES(I)}
\\[6pt]
&\quad  
+
\left\| |\nabla|^{\frac{4}{d}s_{p_{1}}} h_{1}(v)  
\right\|_{L^{\frac{2\rho}{2d-(d-2)\rho}}
(I,L^{ \frac{d^{2}(p_{1}-1)\rho}{d(d+2)(p_{1}-1)\rho-8\rho-2d^{2}(p_{1}-1)}
})}\!\!\!
\left\| w \right\|_{ES(I)}.
\end{split}
\end{equation}
Here, $(\frac{(p_{1}-1)\rho}{\rho-2}, \frac{2(p_{1}-1)\rho}{2d-(d-2)\rho})$ is an $\dot{H}^{s_{p_{1}}}$-admissible pair in $(\frac{d(p_{1}-1)}{2}, \frac{(d+2)(p_{1}-1)}{2})  \times (\frac{(d+2)(p_{1}-1)}{2}, \infty)$. Moreover, it follows from Lemma \ref{10/10/29/18:41}
 and the H\"older inequality in time
 that 
\begin{equation}\label{11/05/17/18:22}
\begin{split}
&\left\| |\nabla|^{\frac{4}{d}s_{p_{1}}} h_{1}(v)  
\right\|_{L^{\frac{2\rho}{2d-(d-2)\rho}}
(I,L^{\frac{d^{2}(p_{1}-1)\rho}
{d(d+2)(p_{1}-1)\rho-8\rho-2d^{2}(p_{1}-1)}
})}
\\[6pt]
&\lesssim 
\left\| v \right\|_{
L^{(p_{1}-1-\frac{4}{d})r_{1}}
(I,L^{(p_{1}-1-\frac{4}{d})q_{1}})
}^{p-1-\frac{4}{d}}
\left\| |\nabla|^{s_{p_{1}}} v 
\right\|_{
L^{\frac{4}{d}r_{2}}
(I,L^{\frac{4}{d}q_{2}})}^{\frac{4}{d}},
\end{split}
\end{equation}
where 
\begin{align}
\label{11/05/19/16:52}
&q_{1}:=\frac{d(d+2)(p_{1}-1)}{2d(p_{1}-1)-8},
\qquad 
q_{2}
:=
\frac{d^{2}(d+2)(p_{1}-1)\rho}
{d(p_{1}-1)\left\{ (d^{2}+2d+4)\rho-2(d^{2}+2d)\right\}-16\rho},
\\[6pt]
\label{11/05/19/10:07}
&
r_{1}:=\frac{d(d+2)(p_{1}-1)}{2d(p_{1}-1)-8}, 
\qquad 
r_{2}:=\frac{2d(d+2)(p_{1}-1)\rho}
{16\rho-d^{2}(p_{1}-1)\left\{d\rho- 2(d+2) \right\}}.
\end{align} 
Here, $((p_{1}-1-\frac{4}{d})q_{1},\, (p_{1}-1-\frac{4}{d})r_{1})=(\frac{(d+2)(p_{1}-1)}{2},\, \frac{(d+2)(p_{1}-1)}{2})$ is 
 a diagonal $\dot{H}^{s_{p_{1}}}$-admissible pair, and 
$(\frac{4}{d}q_{2}, 
 \frac{4}{d}r_{2})$ is an $L^{2}$-admissible pair in $(2,\frac{(d+2)(p_{1}-1)}{2})\times (\frac{(d+2)(p_{1}-1)}{2},\infty)$. 
\par 
Combining the estimates (\ref{11/05/21/17:22}) and (\ref{11/05/17/18:22}), we obtain the desired estimate (\ref{11/05/17/18:45}). 
\par 
Next, we consider the estimate (\ref{11/05/21/22:32}). 
Since $(\frac{4}{d-2}\frac{\rho}{\rho-2}, \frac{4}{d-2}\frac{\rho}{\rho-2})$ is $\dot{H}^{1}$-admissible, 
we see from the same estimate as the above that 
\begin{equation}\label{11/05/26/17:16}
\begin{split}
\left\|h_{*}(v) w \right\|_{ES^{*}(I)}
&\lesssim 
\left\| \nabla v \right\|_{L^{r_{3}}(I,L^{q_{3}})}^{\frac{4}{d-2}}\left\|  w \right\|_{ES(I)}
\\[6pt]
&\quad 
+
\left\| |\nabla|^{\frac{4}{d}s_{p_{1}}} h_{*}(v)  
\right\|_{L^{\frac{2\rho}{2d-(d-2)\rho}}
(I,L^{\frac{d^{2}(p_{1}-1)\rho}
{d(d+2)(p_{1}-1)\rho-8\rho-2d^{2}(p_{1}-1)}
})}
\left\| w \right\|_{ES(I)}
\end{split}
\end{equation}
for some $L^{2}$-admissible pair $(q_{3},r_{3})\neq (2,\infty), (2^{*},2)$. 
Moreover, it follows from Lemma \ref{10/10/29/18:41}
 and the H\"older inequality in time
 that 
\begin{equation}\label{11/05/26/17:26}
\begin{split}
&\left\| |\nabla|^{\frac{4}{d}s_{p_{1}}} h_{*}(v)  
\right\|_{L^{\frac{2\rho}{2d-(d-2)\rho}}
(I,L^{\frac{d^{2}(p_{1}-1)\rho}
{d(d+2)(p_{1}-1)\rho-8\rho-2d^{2}(p_{1}-1)}
})}
\\[6pt]
&\lesssim 
\left\| v \right\|_{
L^{(\frac{4}{d-2}-\frac{4}{d})r_{1}^{*}}
(I,L^{(\frac{4}{d-2}-\frac{4}{d})q_{1}^{*}})
}^{\frac{4}{d-2}-\frac{4}{d}}
\left\| |\nabla|^{s_{p_{1}}} v 
\right\|_{
L^{\frac{4}{d}r_{2}^{*}}
(I,L^{\frac{4}{d}q_{2}^{*}})}^{\frac{4}{d}},
\end{split}
\end{equation}
where 
\begin{align}
\label{11/06/01/17:54}
&q_{1}^{*}:=\frac{d(d+2)}{4},
\qquad 
q_{2}^{*}
:=
\frac{d^{2}(d+2)(p_{1}-1)\rho}
{d^{2}(d+4)(p_{1}-1)\rho -8(d+2)\rho -2d^{2}(d+2)(p_{1}-1)},
\\[6pt]
\label{11/06/01/17:55}
&
r_{1}^{*}:=\frac{d(d+2)}{4}, 
\qquad 
r_{2}^{*}:=\frac{2d(d+2)\rho}
{2d^{2}(d+2) -d(d-2)(d+2)\rho -8 \rho}.
\end{align} 
Since $((\frac{4}{d-2}-\frac{4}{d})q_{1}^{*},\, (\frac{4}{d-2}-\frac{4}{d})r_{1}^{*})=(\frac{2(d+2)}{d-2},\, \frac{2(d+2)}{d-2})$ is a diagonal $\dot{H}^{1}$-admissible pair and 
$(\frac{4}{d}q_{2}^{*}, 
 \frac{4}{d}r_{2}^{*})$ is an $\dot{H}^{1-s_{p_{1}}}$-admissible pair in $(2,2^{*})\times (2,\infty)$, the Sobolev embedding gives us 
the desired estimate (\ref{11/05/21/22:32}). 
\end{proof}

Let $(q_{0},r_{0})$ and $(\widetilde{r}_{0},\widetilde{q_{0}})$ be $L^{2}$-admissible pairs  found in Lemma \ref{11/05/16/0:44}. In our proof of the perturbation theories below, we  will need the auxiliary space   
\begin{equation}
\label{11/06/01/17:01}
X(I):=L^{r_{0}}(I,L^{q_{0}})\cap L^{\widetilde{r}_{0}}(I,L^{\widetilde{q}_{0}})
.
\end{equation}

\begin{proposition}[Short-time perturbation theory, \cite{Tao-Visan}]
\label{11/06/02/17:55}
Assume $d\ge 5$. Let $I$ be an interval, $\psi \in C(I,H^{1}(\mathbb{R}^{d}))$ be a solution to (\ref{11/06/12/9:08}), and let $u$ be a function in $C(I,H^{1}(\mathbb{R}^{d}))$. Put 
\begin{equation}\label{11/07/16/15:35}
e:= i\frac{\partial u}{\partial t} +\Delta u + 
f(u)+|u|^{2^{*}-2}u.
\end{equation}
 Let $A_{1}, A_{2}>0$ and $t_{0} \in I$, and assume that 
\begin{align}
\label{11/06/02/17:56}
\left\| \psi \right\|_{L^{\infty}(I,H^{1})}
&\le A_{1},
\\[6pt]
\label{11/08/14/15:14}
\left\| \psi(t_{0})-u(t_{0}) \right\|_{H^{1}}
&\le A_{2}.
\end{align}
Then, there exists $\delta_{0}>0$ depending on $A_{1}$ and $A_{2}$ with the following property: If  
\begin{equation}
\label{11/05/30/9:45}
\left\|\langle \nabla \rangle u \right\|_{V_{p_{1}}(I) \cap V(I) \cap X(I)} \le \delta_{0},
\end{equation}
and 
\begin{align}
\label{11/05/30/9:53}
&\left\| \langle \nabla \rangle 
e 
\right\|_{L^{\frac{2(d+2)}{d+4}}(I,L^{\frac{2(d+2)}{d+4}})}
\le \delta, 
\\[6pt]
\label{11/05/30/9:54}
&\left\| \langle \nabla \rangle  
e^{i(t-t_{0})\Delta}\left(u(t_{0})-\psi(t_{0})\right)
\right\|_{V_{p_{1}}(I)}
\le \delta 
\end{align}
for some $0<\delta \le \delta_{0}$, then we have 
\begin{align}
\label{11/05/30/10:31}
&\left\| \langle \nabla \rangle \left( u- \psi \right) \right\|_{
V_{p_{1}}(I)\cap V(I)\cap X(I)}\lesssim 
\delta 
+ 
\delta^{\frac{4}{d}\theta_{ES}},
\\[6pt]
\label{11/06/07/22:59}
&\left\| \langle \nabla \rangle 
\left( i\frac{\partial }{\partial t}+\Delta \right) (u-\psi)
+
\langle \nabla \rangle e 
\right\|_{L^{2}(I,L^{\frac{2d}{d+2}})}
\lesssim 
\delta 
+ 
\delta^{\frac{4}{d}\theta_{ES}}
,
\\[6pt]
\label{11/05/30/10:32}
&\left\| \langle \nabla \rangle \left( u-\psi \right) \right\|_{S(I)}
\lesssim 
A_{2} + \delta 
+ 
\delta^{\frac{4}{d}\theta_{ES}}
,
\\[6pt]
\label{11/05/30/10:33}
&\left\| \langle \nabla \rangle \psi \right\|_{S(I)}\lesssim A_{1}+A_{2}.
\end{align} 
\end{proposition}

\begin{proof}[Proof of Proposition \ref{11/06/02/17:55}]
Put \begin{equation}\label{11/05/29/12:28} 
w:=\psi-u.
\end{equation} 
Then, $w$ satisfies that 
\begin{equation}\label{11/05/15/23:41}
\begin{split}
i\frac{\partial }{\partial t}w +\Delta w 
&=
-\bigm(
f(w+u)- f(u)
\bigm) 
\\[6pt]
&\qquad 
-
\bigm( |w+u|^{2^{*}-2}(w+u)
-
 |u|^{2^{*}-2}u
\bigm)
-e .
\end{split}
\end{equation}
We can divide 
\begin{align}
\label{11/06/12/14:42}
&\frac{\partial}{\partial z}
\left(
f(u)
+
|u|^{2^{*}-2}u
\right)
=g_{1}(u)+g_{*}(u), 
\\[6pt]
\label{11/06/12/14:49}
&\frac{\partial}{\partial \bar{z}}
\left(
f(u)+ |u|^{2^{*}-2}u
\right)
=
h_{1}(u)+h_{*}(u),
\end{align}
where $g_{1}$ and $h_{1}$ are H\"older continuous functions of order $p_{1}-1$, and $g_{*}$ and $h_{*}$ are ones of order $2^{*}-2=\frac{4}{d-2}$. 
\par 
The exotic Strichartz estimate (Lemma \ref{10/10/16/22:32}) together with (\ref{11/06/02/22:24}), (\ref{11/05/30/9:53}), (\ref{11/05/30/9:54}) and Lemma \ref{11/05/16/0:44} shows 
\begin{equation}
\label{11/05/16/0:00}
\begin{split}
\left\| w \right\|_{ES(I)}
&\lesssim 
A_{2}^{1-\theta_{ES}}\delta^{\theta_{ES}}
+
\left\| 
\int_{0}^{1}
\bigg(g_{1}+h_{1}\bigg)(u+\theta w)\,d\theta \, w
\right\|_{ES^{*}(I)}
\\[6pt]
&\qquad 
+
\left\| 
\int_{0}^{1}
\bigg(g_{*}+h_{*}\bigg)(u+\theta w)\,d\theta \, \overline{w}
\right\|_{ES^{*}(I)}
+\left\| \langle \nabla  \rangle e \right\|_{L^{\frac{2(d+2)}{d+4}}(I,L^{\frac{2(d+2)}{d+4}})}
\\[6pt]
&\lesssim
A_{2}^{1-\theta_{ES}}\delta^{\theta_{ES}}
+
\left( \left\| |\nabla|^{s_{p_{1}}}u \right\|_{X(I)}^{p_{1}-1}
+
\left\| |\nabla|^{s_{p_{1}}}\psi \right\|_{X (I)}^{p_{1}-1}
\right)\left\| w \right\|_{ES(I)}
\\[6pt]
&\qquad 
+
\left( 
\left\| \nabla u \right\|_{X(I)}^{2^{*}-2}
+
\left\| \nabla \psi \right\|_{X(I)}^{2^{*}-2}
\right)
\left\| w \right\|_{ES(I)}
+
\delta.
\end{split}
\end{equation}
We shall derive an estimate for $\left\|\langle \nabla \rangle \psi \right\|_{X(I)}$. 
Using the triangle inequality and \eqref{11/05/30/9:54}, we have  
\begin{equation}\label{11/08/14/15:33}
\left\|\langle \nabla \rangle e^{i(t-t_{0})\Delta} \psi(t_{0}) 
\right\|_{V_{p_{1}}(I)}
\le 
\left\| \langle \nabla \rangle e^{i(t-t_{0})\Delta} u(t_{0}) 
\right\|_{V_{p_{1}}(I)}
+
\delta. 
\end{equation}
Here, the Strichartz estimate together with \eqref{11/05/11/23:53}, \eqref{11/05/30/9:45} and \eqref{11/05/30/9:53} gives us that 
\begin{equation}\label{11/06/02/9:14}
\begin{split}
&\left\| \langle \nabla \rangle 
e^{i(t-t_{0})\Delta} u(t_{0}) 
\right\|_{V_{p_{1}}(I)}
\\[6pt]
&=
\left\| \langle \nabla \rangle 
\left( u 
-i\int_{t_{0}}^{t} e^{i(t-t')\Delta}\left\{ f(u)+|u|^{2^{*}-2}u+e 
\right\}(t')dt' 
\right)
\right\|_{V_{p_{1}}(I)}
\\[6pt]
&\lesssim 
\delta_{0}
+
\left( \delta_{0}^{p_{1}}
+
\delta_{0}^{p_{2}}
+
\delta_{0}^{2^{*}-1}
\right)
+
\delta 
\lesssim \delta_{0},
\qquad 
\mbox{provided that $\delta_{0}\le 1$}.
\end{split}
\end{equation}
Combining \eqref{11/08/14/15:33} and \eqref{11/06/02/9:14}, we have 
\begin{equation}\label{11/08/14/16:08}
\left\| \langle \nabla \rangle 
e^{i(t-t_{0})\Delta} \psi(t_{0}) 
\right\|_{V_{p_{1}}(I)}\lesssim \delta_{0}. 
\end{equation}
Hence, taking $\delta_{0}$ sufficiently small depending on $A_{1}$, we find from Theorem \ref{10/10/04/21:44} that  
\begin{equation}\label{11/06/02/18:23}
\left\| \langle \nabla \rangle \psi \right\|_{V_{p_{1}}(I)}\lesssim 
\delta_{0},
\end{equation}
which together with the Sobolev embedding, the H\"older inequality and \eqref{11/06/02/17:56} also yields 
\begin{align}
\label{11/08/15/12:20}
&\left\| \psi \right\|_{W_{p_{1}}(I)}
\lesssim 
\delta_{0},
\\[6pt]
\label{11/08/18/16:13}
&
\left\| \psi \right\|_{W(I)}
\lesssim 
\left\| \langle \nabla \rangle \psi \right\|_{V(I)}
\le 
A_{1}^{1-\theta_{V}}\delta_{0}^{\theta_{V}}
\qquad 
\mbox{for some $0<\theta_{V}<1$}, 
\\[6pt]
\label{11/08/18/16:11}
&\left\|\langle \nabla \rangle \psi \right\|_{X(I)}
\lesssim 
A_{1}^{1-\theta_{X}}\delta_{0}^{\theta_{X}}
\qquad 
\mbox{for some $0<\theta_{X}<1$}.
\end{align}
Returning to (\ref{11/05/16/0:00}), we see from \eqref{11/05/30/9:45} and \eqref{11/08/18/16:11} that 
\begin{equation}\label{11/06/13/8:17}
\begin{split}
\left\| w \right\|_{ES(I)}
&\lesssim 
A_{2}^{1-\theta_{ES}}\delta^{\theta_{ES}} +\delta
\\[6pt]
&\qquad + \bigm\{
\delta_{0}^{p_{1}-1}
+
(A^{1-\theta_{X}}\delta_{0})^{p_{1}-1}
+ \delta_{0}^{2^{*}-2}
+
(A^{1-\theta_{X}} \delta_{0}^{\theta_{X}})^{2^{*}-2}
\bigm\}
\left\| w \right\|_{ES(I)},
\end{split}
\end{equation}
so that 
\begin{equation}\label{11/06/03/9:32}
\left\| w \right\|_{ES(I)}
\lesssim A_{2}^{1-\theta_{ES}}\delta^{\theta_{ES}},
\end{equation}
provided that $\delta_{0}$ is sufficiently small depending  on $A_{1}$ and $A_{2}$.
\par 
Now, we shall show \eqref{11/05/30/10:31}. Note that we have by \eqref{11/04/30/22:54} and \eqref{11/05/02/8:40} that 
\begin{equation}\label{11/08/15/10:54}
\left| f(w+u)-f(u) \right|
\lesssim 
\left( |\psi|^{p_{1}-1}+|u|^{p_{1}-1} +|\psi|^{p_{2}-1}+|u|^{p_{2}-1} \right)|w|
\end{equation}
and 
\begin{equation}\label{11/08/15/10:42}
\begin{split}
&\left| \nabla (f(w+u)-f(u)) \right|
\le \left| \nabla f(w+u) -\nabla f(u) \right|
\\[6pt]
&\lesssim 
\left( |\psi|^{p_{1}-1}+|\psi|^{p_{2}-1} \right)|\nabla w|
+
\left( \left| w \right|^{p_{1}-1}+|w|^{p_{2}-1}\right)
|\nabla u |.
\end{split}
\end{equation}
It follows from \eqref{11/05/30/9:54} and the inhomogeneous Strichartz estimate  together with  \eqref{11/08/15/10:54}, \eqref{11/08/15/10:42}, \eqref{11/06/03/9:59} and \eqref{11/05/30/9:53}  that \begin{equation}\label{11/06/03/9:37}
\begin{split}
&\left\| \langle \nabla \rangle w \right\|_{V_{p_{1}}(I)\cap V(I)\cap X(I)}
\\[6pt]
&\lesssim \delta +  
\left( 
\left\| \psi \right\|_{W_{p_{1}}(I)}^{p_{1}-1}
+
\left\| \psi \right\|_{W_{p_{2}}(I)}^{p_{2}-1}
+
\left\| \psi \right\|_{W(I)}^{2^{*}-2}
\right)
\left\| \langle \nabla \rangle w \right\|_{V_{p_{1}}(I)\cap V(I) \cap X(I)}
\\[6pt]
&\qquad 
+
\left( 
\left\| u \right\|_{W_{p_{1}}(I)}^{p_{1}-1}
+
\left\| u \right\|_{W_{p_{2}}(I)}^{p_{2}-1}
+
\left\| u \right\|_{W(I)}^{2^{*}-2}
\right)
\left\| \langle \nabla \rangle w \right\|_{V_{p_{1}}(I)\cap V(I) \cap X(I)}
\\[6pt]
&\qquad 
+ 
\left( \left\| w \right\|_{ES(I)}^{p_{1}-1}
+
\left\| w \right\|_{ES(I)}^{p_{2}-1} 
+
\left\| w \right\|_{ES(I)}^{2^{*}-2}
\right)\left\| \nabla u \right\|_{S(I)}.
\end{split}
\end{equation}
Here, the Strichartz estimate together with \eqref{11/06/02/17:56}, \eqref{11/08/14/15:14}, \eqref{11/05/11/23:53} and \eqref{11/05/30/9:53} yields that 
\begin{equation}\label{11/08/15/14:01}
\left\| \langle \nabla \rangle u \right\|_{S(I)}
\lesssim 
A_{1}+A_{2}+
\left( 
\left\| u \right\|_{W_{p_{1}}(I)}^{p_{1}-1}
+
\left\| u \right\|_{W_{p_{2}}(I)}^{p_{2}-1}
+
\left\| u \right\|_{W(I)}^{2^{*}-2}
\right)
\left\|\langle \nabla \rangle u \right\|_{V_{p_{1}}(I)}
+
\delta
.
\end{equation}
Hence, we see from \eqref{11/08/15/14:01} together with 
 \eqref{11/05/30/9:45}, and then taking $\delta_{0}$ sufficiently small depending on $A_{1}$ and $A_{2}$, we have 
\begin{equation}\label{11/06/04/10:11}
\left\|\langle \nabla \rangle u \right\|_{S(I)}
\lesssim 
A_{1}+A_{2}.
\end{equation}
Taking $\delta_{0}$ sufficiently small depending on $A_{1}$ and $A_{2}$, we find from \eqref{11/06/03/9:37} together with \eqref{11/05/30/9:45}, \eqref{11/08/15/12:20}, \eqref{11/06/03/9:32} and \eqref{11/06/04/10:11} that  
\begin{equation}\label{11/06/04/10:09}
\left\|\langle \nabla \rangle  w \right\|_{V_{p_{1}}(I)\cap V(I)\cap X(I)}
\lesssim 
\delta + 
\delta^{\frac{4}{d}\theta_{ES}},
\end{equation}
which gives \eqref{11/05/30/10:31}.
\par 
Next, we shall show (\ref{11/06/07/22:59}). Using \eqref{11/05/15/23:41}, \eqref{11/08/15/10:54}, \eqref{11/08/15/10:42} and the H\"older inequality, we verify that 
\begin{equation}\label{11/07/16/15:44}
\begin{split}
&\left\| \langle \nabla \rangle 
\left( i\frac{\partial }{\partial t}+\Delta \right) (u-\psi)
+
\langle \nabla \rangle e 
\right\|_{L^{2}(I,L^{\frac{2d}{d+2}})}
\\[6pt]
&\lesssim  
\left( \left\| \psi \right\|_{W_{p_{1}}(I)}^{p_{1}-1}
+
\left\| \psi \right\|_{W_{p_{2}}(I)}^{p_{2}-1}
+
\left\| \psi \right\|_{W(I)}^{2^{*}-2}
\right)
\left\| \langle \nabla \rangle w \right\|_{V(I)}
\\[6pt]
&\qquad +
\left( \left\| u \right\|_{W_{p_{1}}(I)}^{p_{1}-1}
+
\left\| u \right\|_{W_{p_{2}}(I)}^{p_{2}-1}
+
\left\| u \right\|_{W(I)}^{2^{*}-2}
\right)
\left\| \langle \nabla \rangle w \right\|_{V(I)}
\\[6pt]
&\qquad +
\left( \left\| w \right\|_{W_{p_{1}}(I)}^{p_{1}-1}
+
\left\| w \right\|_{W_{p_{2}}(I)}^{p_{2}-1}
+
\left\| w \right\|_{W(I)}^{2^{*}-2}
\right)
\left\| \langle \nabla \rangle u \right\|_{V(I)}
.
\end{split}
\end{equation}
Moreover, it follows from \eqref{11/07/16/15:44} together with \eqref{11/05/30/9:45},   \eqref{11/08/15/12:20} and \eqref{11/06/04/10:09} that\begin{equation}\label{11/06/07/23:20}
\left\| \langle \nabla \rangle 
\left( i\frac{\partial }{\partial t}+\Delta \right) w 
+
\langle \nabla \rangle e 
\right\|_{L^{2}(I,L^{\frac{2d}{d+2}})}
\lesssim 
\delta 
+ 
\delta^{\frac{4}{d}\theta_{ES}},
\end{equation}
provided that $\delta_{0}$ is sufficiently small depending on $A_{1}$ and $A_{2}$. Hence, we have proved (\ref{11/06/07/22:59}).
\par 
Finally. we prove \eqref{11/05/30/10:32} and \eqref{11/05/30/10:33}. 
Combining the Strichartz estimate together with \eqref{11/08/14/15:14}, (\ref{11/05/30/9:53}) and (\ref{11/06/07/23:20}), we obtain  
\begin{equation}\label{11/06/07/22:53}
\left\| \langle \nabla \rangle w \right\|_{S(I)}
\lesssim 
A_{2} + \delta 
+ 
\delta^{\frac{4}{d}\theta_{ES}},
\end{equation}
so that \eqref{11/05/30/10:32} holds. Moreover, \eqref{11/05/30/10:33} follows from (\ref{11/06/04/10:11}) and (\ref{11/06/07/22:53}).
\end{proof}

\begin{proposition}[Long-time perturbation theory]
\label{11/05/14/12:26}
Assume $d\ge 3$ and $2_{*}-1<p<2^{*}-1$. Let $I$ be an interval, $\psi \in C(I,H^{1}(\mathbb{R}^{d}))$ be a solution to (\ref{11/06/12/9:08}), and let $u$ be a function in $C(I,H^{1}(\mathbb{R}^{d}))$.
 Let $A_{1},A_{2},B>0$ and $t_{1} \in I$, and assume that 
\begin{align}
\label{11/05/15/23:09}
&\left\| \psi \right\|_{L^{\infty}(I,H^{1})}
\le A_{1},
\qquad  
\\[6pt]
\label{11/08/14/11:31}
&\left\| \psi(t_{1})-u(t_{1}) \right\|_{H^{1}}
\le A_{2},
\\[6pt]
\label{11/05/30/9:30}
&\left\| u \right\|_{W_{p_{1}}(I)\cap W(I)}\le B.
\end{align}
Then, there exists $\delta>0$ depending on $A_{1}$, $A_{2}$ and $B$ such that if\begin{align}
\label{11/05/15/23:15}
&\left\| \langle \nabla \rangle 
\biggm(
i\frac{\partial }{\partial t}u +\Delta u + f(u)
+|u|^{2^{*}-2}u
\biggm)
\right\|_{L^{\frac{2(d+2)}{d+4}}(I,L^{\frac{2(d+2)}{d+4}})}
\le \delta, 
\\[6pt]
\label{11/05/15/23:24}
&\left\| \langle \nabla \rangle  
e^{i(t-t_{1})\Delta}\left(u(t_{1})-\psi(t_{1})\right)
\right\|_{V_{p_{1}}(I)}
\le \delta
\qquad 
\mbox{for some $t_{1}\in I$},
\end{align}
then we have 
\begin{equation}\label{11/06/10/9:41}
\left\| \langle \nabla \rangle  \psi \right\|_{S(I)}<\infty .
\end{equation}
\end{proposition}
\begin{remark}\label{11/09/06/10:35}
In the long-time perturbation theory above, we assume the uniform boundedness of $\psi$ in $H^{1}(\mathbb{R}^{d})$, instead of the one of $u$ (cf. \cite{IMN, Killip-Visan}); For, it is easier to obtain the bound for $\psi$ than the approximate solution $u$.  
\end{remark}
\begin{proof}[Proof of Proposition \ref{11/05/14/12:26}]
We consider the case $\inf{I}=t_{1}$ only. The other cases can be proven in the same way as this case. 
\par 
Our first step is to derive a bound of $\left\| \langle \nabla \rangle u \right\|_{S(I)}$.  Let $\eta>0$ be a universal constant specified later, and assume $\delta\le \eta$. We see from \eqref{11/05/30/9:30} that: there exist 
\\
{\rm (i)} a number $N'$ depending on $B$ (and $\eta$), and 
\\
{\rm (ii)} disjoint intervals $I_{1}',\ldots, I_{N'}'$ of the form $I_{j}'=[t_{j}',t_{j+1}')$ ($t_{1}'=t_{1}$ and $t_{N+1}':=\max{I}$), 
\\
such that 
\begin{align}
\label{11/05/26/18:28}
&I=\bigcup_{j=1}^{N'}I_{j}', 
\\[6pt]
\label{11/05/26/18:35}
&\left\| u \right\|_{W_{p_{1}}(I_{j}')\cap W(I_{j}')}\le 
\eta
\qquad 
\mbox{for any $j=1, \ldots, N'$}.
\end{align}
Then, it follows from \eqref{11/05/15/23:15} and \eqref{11/05/26/18:35} that 
\begin{equation}\label{11/08/15/20:01}
\begin{split}
\left\| \langle \nabla \rangle u \right\|_{S(I_{j}')}
&\le C_{1}\left\| u(t_{j}')\right\|_{H^{1}} +C_{2}\delta
\\[6pt]
&\quad +
C_{3}\left( \left\| u \right\|_{W_{p_{1}}(I_{j}')}^{p_{1}-1}
+
\left\| u \right\|_{W_{p_{2}}(I_{j}')}^{p_{2}-1}
+
\left\| u \right\|_{W(I_{j}')}^{2^{*}-2}
\right)
\left\| \langle \nabla \rangle u \right\|_{V(I_{j}')}
\\[6pt]
&\le 
C_{1}\left\| u(t_{j}')\right\|_{H^{1}}+C_{2}\delta
\\[6pt]
&\quad + 
C_{3}'\left( \eta^{p_{1}-1}
+\eta^{p_{2}-1}+\eta^{2^{*}-2}
\right)\left\| \langle \nabla \rangle u \right\|_{S(I_{j}')}
\qquad 
\mbox{for any $1\le j \le N'$},
\end{split}
\end{equation}
where $C_{1}$, $C_{2}$, $C_{3}$ and $C_{3}'$ are some universal constants. We choose $\eta$ so small that 
\begin{equation}\label{11/08/16/11:08}
C_{3}'\left( \eta^{p_{1}-1}
+\eta^{p_{2}-1}+\eta^{2^{*}-2}
\right)\le \frac{1}{2}. 
\end{equation}
Here, \eqref{11/05/15/23:09} and \eqref{11/08/14/11:31} shows that\begin{equation}\label{11/08/16/11:00}
\left\| u(t_{1}')\right\|_{H^{1}}
\le 
\left\| \psi \right\|_{L^{\infty}(I,H^{1})}
+
\left\| u(t_{1})-\psi(t_{1}) \right\|_{H^{1}}
\le A_{1}+A_{2},
\end{equation}
so that, taking $\delta$ so small that $C_{2}\delta\le A_{1}+A_{2}$, we see from \eqref{11/08/15/20:01} together with \eqref{11/08/16/11:08} that   
\begin{equation}\label{11/08/15/20:20}
\left\| \langle \nabla \rangle u \right\|_{S(I_{1}')}
\le 2(C_{1}+1)(A_{1}+A_{2}).
\end{equation}
In particular, we have 
\begin{equation}\label{11/08/15/20:23}
\left\| u(t_{2}') \right\|_{H^{1}}\le 2(C_{1}+1)(A_{1}+A_{2}). 
\end{equation}
Hence, using \eqref{11/08/15/20:01} again, we obtain  
\begin{equation}\label{11/08/15/20:20}
\left\| \langle \nabla \rangle u \right\|_{S(I_{2}')}
\le 2^{2}C_{1}(C_{1}+1)(A_{1}+A_{2})+2(A_{1}+A_{2})
\lesssim A_{1}+A_{2}.
\end{equation}
Iterating this, we consequently have 
\begin{equation}\label{11/08/16/11:59}
\left\| \langle \nabla \rangle u \right\|_{S(I)}
\le C(A_{1}, A_{2}, B)
\end{equation}
for some constant $C(A_{1}, A_{2}, B)>0$ depending on $A_{1}$, $A_{2}$ and $B$. In particular, we have 
\begin{equation}\label{11/08/15/20:20}
\sup_{t\in I}\left\| \psi(t)-u(t) \right\|_{H^{1}}
\le A_{1}+C(A_{1},A_{2},B)=:A_{2}'. 
\end{equation} 
\par 
Now, let $\delta_{0}$ be the constant found in Theorem \ref{11/06/02/17:55} which is determined by $A_{1}$ and $A_{2}'$ given in \eqref{11/05/15/23:09} and \eqref{11/08/16/11:59}, and suppose that $\delta<\min\{ \delta_{0}, \eta \}$. 
 Then, we see from \eqref{11/08/16/11:59} that there exist 
\\
{\rm (i)} a number $N$ depending $A_{1}$, $A_{2}$ and $B$, and 
\\
{\rm (ii)} disjoint intervals $I_{1},\ldots, I_{N}$ of the form $I_{j}=[t_{j},t_{j+1})$ ($t_{N+1}:=\max{I}$), 
\\
such that 
\begin{align}
\label{11/08/15/17:11}
&I=\bigcup_{j=1}^{N}I_{j}, 
\\[6pt]
\label{11/08/15/17:14}
&\left\| \langle \nabla \rangle u \right\|_{V_{p_{1}}(I_{j})\cap 
V(I_{j})\cap X(I_{j})}\le \delta_{0}
\qquad 
\mbox{for any $j=1, \ldots, N$}.
\end{align}
Proposition \ref{11/06/02/17:55} together with (\ref{11/05/15/23:15}) and (\ref{11/05/15/23:24}) shows 
\begin{align}
\label{11/08/18/13:52}
&\left\|\langle \nabla \rangle \left( u- \psi \right) 
\right\|_{V_{p_{1}}(I_{1})\cap V(I_{1})\cap X(I_{1})}
\le C_{0}
\left\{  
\delta 
+ 
\delta^{\frac{4}{d}\theta_{EX}}
\right\},
\\[6pt]
\label{11/08/18/13:53}
&\left\| \langle \nabla \rangle 
\left( i\frac{\partial }{\partial t}+\Delta \right) (u-\psi)
+
\langle \nabla \rangle e 
\right\|_{L^{2}(I_{1},L^{\frac{2d}{d+2}})}
\le C_{0}
\left\{
\delta 
+ 
\delta^{\frac{4}{d}\theta_{ES}}
\right\},
\\[6pt]
\label{11/08/18/13:54}
&\left\| \langle \nabla \rangle \left( u-\psi \right) 
\right\|_{S(I_{1})}
\le C_{0} 
\left\{  
A_{2}' + \delta 
+ 
\delta^{\frac{4}{d}\theta_{ES}}
\right\},
\\[6pt]
\label{11/08/18/13:55}
&\left\| \langle \nabla \rangle\psi \right\|_{S(I_{1})}
\le C_{0} 
\left\{   
A_{1}+A_{2}'\right\},
\end{align} 
where $C_{0}$ is some universal constant. Here, we have the formula\begin{equation}\label{11/06/08/15:00}
\begin{split}
e^{i(t-t_{j+1})\Delta}\left(u(t_{j+1})-\psi(t_{j+1})
\right)
&=e^{i(t-t_{j})\Delta}\left(u(t_{j})-\psi(t_{j})
\right)
\\[6pt]
&\quad + 
i
\int_{I_{j}}
\hspace{-2pt}e^{i(t-t')\Delta}
\left\{ \left(
i\frac{\partial }{\partial t}+\Delta \right) (u-\psi) +e 
\right\}
(t')  \,dt'.
\end{split}
\end{equation}
Using the Strichartz estimate, the formula \eqref{11/06/08/15:00}, \eqref{11/05/15/23:24} and \eqref{11/08/18/13:53}, we obtain that\begin{equation}\label{11/06/08/14:28}
\begin{split}
&\left\|  \langle \nabla \rangle 
e^{i(t-t_{2})\Delta}\left(u(t_{2})-\psi(t_{2})\right)
\right\|_{V_{p_{1}}(I)}
\\[6pt]
&\le 
\left\|  \langle \nabla \rangle 
e^{i(t-t_{1})\Delta}\left(u(t_{1})-\psi(t_{1})\right)
\right\|_{V_{p_{1}}(I)}
\hspace{-6pt}+
\left\|\langle \nabla \rangle \left( i\frac{\partial }{\partial t}+\Delta \right) (u-\psi)
+
\langle \nabla \rangle e \right\|_{L^{2}(I_{1},L^{\frac{2d}{d+2}})}
\\[6pt]
&\le  
\delta 
+
C_{0} \left\{ 
\delta
+
\delta^{\frac{4}{d}\theta_{ES}}
\right\}.
\end{split}
\end{equation}
We choose $\delta$ so small that 
\begin{equation}\label{11/08/18/14:18}
\delta_{1}:=\delta 
+
C_{0} \left\{ 
\delta
+
\delta^{\frac{4}{d}\theta_{ES}}
\right\}
< \delta_{0}.
\end{equation} 
Then, it follows from Proposition \ref{11/06/02/17:55} that 
\begin{align}
\label{11/06/08/14:03}
&\left\|\langle \nabla \rangle \left( u- \psi \right) 
\right\|_{V_{p_{1}}(I_{2})\cap V(I_{2})\cap X(I_{2})}
\le C_{0} 
\left\{ 
\delta_{1} 
+ 
\delta_{1}^{\frac{4}{d}\theta_{EX}}
\right\},
\\[6pt]
\label{11/06/08/14:04}
&\left\| \langle \nabla \rangle 
\left( i\frac{\partial }{\partial t}+\Delta \right) (u-\psi)
+
\langle \nabla \rangle e 
\right\|_{L^{2}(I_{2},L^{\frac{2d}{d+2}})}
\le 
C_{0}
\left\{ 
\delta_{1} 
+ 
\delta_{1}^{\frac{4}{d}\theta_{ES}}
\right\},
\\[6pt]
\label{11/06/08/14:05}
&\left\| \langle \nabla \rangle \left( u-\psi \right) 
\right\|_{S(I_{2})}
\le C_{0}
\left\{ 
A_{2}' + \delta_{1} 
+ 
\delta_{1}^{\frac{4}{d}\theta_{ES}}
\right\} 
,
\\[6pt]
\label{11/06/08/14:06}
&\left\| \langle \nabla \rangle\psi \right\|_{S(I_{2})}\le 
C_{0}\left\{  A_{1}+A_{2}' \right\}.
\end{align} 
Moreover, we see from the formula \eqref{11/06/08/15:00} together with \eqref{11/06/08/14:28} and \eqref{11/08/18/13:53} that 
\begin{equation}\label{11/08/18/14:32}
\begin{split}
&\left\|  \langle \nabla \rangle 
e^{i(t-t_{3})\Delta}\left(u(t_{3})-\psi(t_{3})\right)
\right\|_{V_{p_{1}}(I)}
\\[6pt]
&\le 
\left\|  \langle \nabla \rangle 
e^{i(t-t_{2})\Delta}\left(u(t_{2})-\psi(t_{2})\right)
\right\|_{V_{p_{1}}(I)}
\hspace{-6pt}+
\left\|\langle \nabla \rangle \left( i\frac{\partial }{\partial t}+\Delta \right) (u-\psi)
+
\langle \nabla \rangle e \right\|_{L^{2}(I_{2},L^{\frac{2d}{d+2}})}
\\[6pt]
&\le  
\delta_{1} 
+
C_{0} \left\{ 
\delta_{1}
+
\delta_{1}^{\frac{4}{d}\theta_{ES}}
\right\}.
\end{split}
\end{equation}
Hence, taking further small $\delta$ such that
\begin{equation}\label{11/08/18/14:37}
\delta_{2}:=\delta_{1} 
+
C_{0} \left\{ 
\delta_{1}
+
\delta_{1}^{\frac{4}{d}\theta_{ES}}
\right\}<\delta_{0},
\end{equation} 
we can employ Proposition \ref{11/06/02/17:55} on the interval $I_{3}$.  Repeating this procedure $N$ times, we find that if $\delta$ is sufficiently small depending on $A_{1}$, $A_{2}$ and $B$, then we have 
\begin{equation}\label{11/08/18/14:46}
\left\| \langle \nabla \rangle \psi 
\right\|_{S(I_{j})}
\le C_{0}(A_{1}+A_{2}')
\qquad 
\mbox{for any $1\le j\le N$}. 
\end{equation}
Hence, we have  
\begin{equation}\label{11/06/08/15:26}
\left\| \psi \right\|_{W_{p_{1}}(I)}^{\frac{(d+2)(p_{1}-1)}{2}}
=
\sum_{j=1}^{N}\left\| \psi \right\|_{W_{p_{1}}(I_{j})}^{\frac{(d+2)(p_{1}-1)}{2}}
\le 
\sum_{j=1}^{N}C_{0}(A_{1}+A_{2}') \le C_{0}(A_{1}+A_{2}')N.
\end{equation} 
Similarly, we have 
\begin{equation}\label{11/06/08/15:26}
\left\| \psi \right\|_{W(I)}^{\frac{2(d+2)}{d-2}}
\le C_{0}(A_{1}+A_{2}')N.
\end{equation} 
Then, the desired result (\ref{11/06/10/9:41}) follows from Lemma \ref{11/05/14/14:42}. 
\end{proof}

\section{Scattering result}\label{11/06/10/17:05}
In this section, we prove Theorem \ref{11/06/10/17:09}.  
\subsection{Analysis on $A_{\omega,+}$}
\label{11/06/28/10:28}
We discuss basic properties of the set $A_{\omega,+}$. 
\par 
We first observe  a relation between $\mathcal{K}$ and $\mathcal{H}$: 
\begin{lemma}\label{11/06/28/10:12}
Let $\{u_{n}\}$ be a bounded sequence in $H^{1}(\mathbb{R}^{d})$ with 
\begin{equation}\label{10/09/16/17:12}
\mathcal{K}(u_{n})\ge 0
\qquad 
\mbox{for any $n\in \mathbb{N}$}
.
\end{equation}
Then, we have 
\begin{equation}\label{10/09/16/16:54}
\mathcal{H}(u_{n})> 0
\qquad 
\mbox{for any $n \in \mathbb{N}$}.
\end{equation}
Furthermore, if the sequence $\{u_{n}\}$ satisfies that 
\begin{equation}\label{10/09/17/17:03}
\liminf_{n\to \infty}\left\| \nabla u_{n} \right\|_{L^{2}}>0,
\end{equation}
then 
\begin{equation}\label{10/09/16/16:54}
\liminf_{n\to \infty}\mathcal{H}(u_{n})> 0.
\end{equation}
\end{lemma}

\begin{proof}[Proof of Lemma \ref{11/06/28/10:12}]
Lemma \ref{11/07/20/18:01} together with (\ref{10/09/16/17:12}) shows   that 
\begin{equation}\label{10/08/13/15:51}
\mathcal{K}(T_{\lambda}u_{n})> 0 
\quad 
\mbox{for any $\lambda \in (0,1)$ and $n \in \mathbb{N}$}.
\end{equation}
Hence,  we see from the relation \eqref{11/06/28/9:59} that   
\begin{equation}\label{10/08/13/15:26}
\mathcal{H}(u_{n})>  \mathcal{H}(T_{\lambda}u_{n})
\qquad 
\mbox{for any $\lambda \in (0,1)$ and $n \in \mathbb{N}$}. 
\end{equation}
Here, it follows from \eqref{11/04/30/22:54}, \eqref{11/05/02/8:40} and the Sobolev embedding that there exists a constant $C>0$ independent of $n$ and 
$\lambda$ such that  
\begin{equation}\label{10/09/08/16:14}
\begin{split}
\mathcal{H}(T_{\lambda}u_{n})
&=\lambda^{2}\left\| \nabla u_{n}\right\|_{L^{2}}^{2}
-F(T_{\lambda}u_{n}(x))-\frac{1}{2^{*}}\left\| u_{n} 
\right\|_{L^{2^{*}}}^{2^{*}}
\\[6pt]
&\ge  
\lambda^{2}\left\| \nabla u_{n}\right\|_{L^{2}}^{2}
- C \left(
\sum_{j=1}^{2}
\lambda^{\frac{d}{2}(p_{j}-1)}
\left\| u_{n} \right\|_{H^{1}}^{p_{j}+1}
+\lambda^{2^{*}}
\left\| u_{n} \right\|_{H^{1}}^{2^{*}}
\right).
\end{split}
\end{equation}
Since $\{u_{n}\}$ is bounded in $H^{1}(\mathbb{R}^{d})$ and 
 $2< \frac{d}{2}(p_{1}-1)<\frac{d}{2}(p_{2}-1)<2^{*}$, we obtain the conclusions from the estimates (\ref{10/08/13/15:26}) and (\ref{10/09/08/16:14}).
\end{proof}

Next, we show that for each $\omega>0$, $A_{\omega,+}$ is bounded in $H^{1}(\mathbb{R}^{d})$: 
\begin{lemma}\label{11/06/27/22:43}
Let $\omega>0$, $m>0$ and let $u$ be a function in $H^{1}(\mathbb{R}^{d})$. Assume that 
\begin{equation}
\label{11/07/30/15:47}
\mathcal{K}(u)\ge 0,
\qquad 
\mathcal{S}_{\omega}(u)\le m. 
\end{equation}
Then, we have 
\begin{equation}
\label{11/07/30/15:48}
\left\| u \right\|_{L^{2}}^{2} \le \frac{2m}{\omega},
\qquad 
\left\| \nabla u \right\|_{L^{2}}^{2}
\lesssim 
m+\frac{m}{\omega}.
\end{equation}
In particular, we have 
\begin{equation}
\sup_{u \in A_{\omega,+}}
\left\| u \right\|_{H^{1}}^{2}
\lesssim m_{\omega}+\frac{m_{\omega}}{\omega}.
\end{equation}
\end{lemma}
\begin{proof}[Proof of Lemma \ref{11/06/27/22:43}]
It follows from $\mathcal{K}(u)\ge 0$ that  
\begin{align}\label{11/07/10/9:23}
&\left\| u \right\|_{L^{2^{*}}}^{2^{*}}
\le 
\left\| \nabla u \right\|_{L^{2}}^{2},
\\[6pt]
\label{11/07/10/9:53}
&\mathcal{I}_{\omega}(u)
\le \mathcal{S}_{\omega}(u)\le m. 
\end{align}
We see from \eqref{11/07/10/9:53} that 
\begin{equation}\label{11/07/10/10:04}
\left\| u \right\|_{L^{2}}^{2}\le \frac{2m}{\omega},
\end{equation}
which gives the first claim in \eqref{11/07/30/15:48}. 
\par 
Using \eqref{11/04/30/22:54}, \eqref{11/05/02/8:40} and \eqref{11/07/10/9:23}, we obtain that  
\begin{equation}\label{11/07/10/10:07}
m \ge \mathcal{S}_{\omega}(u) \ge \mathcal{H}(u)
\ge 
\frac{1}{d}\left\|\nabla u \right\|_{L^{2}}^{2}
-C_{1}\left\| u \right\|_{L^{p_{1}+1}}^{p_{1}+1}
-C_{2}\left\| u \right\|_{L^{p_{2}+1}}^{p_{2}+1}, 
\end{equation}
where $C_{1}$ and $C_{2}$ are universal positive constants. Moreover, using the H\"older inequality, \eqref{11/07/10/9:23} and \eqref{11/07/10/10:04}, we have
\begin{equation}\label{11/07/10/9:39}
\begin{split}
m &\ge 
\frac{1}{d}\left\|\nabla u \right\|_{L^{2}}^{2}
-C_{1}
\left\| u \right\|_{L^{2}}^{\frac{(d+2)-(d-2)p_{1}}{2}}
\left\| u \right\|_{L^{2^{*}}}^{\frac{d(p_{1}-1)}{2}}
\\[6pt]
&\hspace{78pt}
-C_{2}
\left\| u \right\|_{L^{2}}^{\frac{(d+2)-(d-2)p_{2}}{2}}
\left\| u \right\|_{L^{2^{*}}}^{\frac{d(p_{2}-1)}{2}}
\\[6pt]
&\ge 
\frac{1}{d}\left\|\nabla \psi(t) \right\|_{L^{2}}^{2}
-C_{1}
\left( \frac{2m_{\omega}}{\omega}\right)^{\frac{(d+2)-(d-2)p_{1}}{4}}
\left( 
\left\| \nabla \psi(t)\right\|_{L^{2}}^{2}
\right)^{\frac{(d-2)(p_{1}-1)}{4}}
\\[6pt]
&\hspace{78pt}
-C_{2}
\left( \frac{2m_{\omega}}{\omega}
\right)^{\frac{(d+2)-(d-2)p_{2}}{2}}
\left( \left\|\nabla \psi(t)\right\|_{L^{2}}^{2}\right)^{\frac{(d-2)(p_{2}-1)}{4}}.
\end{split}
\end{equation}
Since $\frac{(d-2)(p_{1}-1)}{4}\le \frac{(d-2)(p_{2}-1)}{4}<1$, \eqref{11/07/10/9:39} together with the Young inequality yields the second claim in \eqref{11/07/30/15:48}.
\end{proof}

For a function $u \in H^{1}(\mathbb{R}^{d})$ with $\mathcal{K}(u)\ge 0$, we can compare $\mathcal{H}(u)$ and $\left\| \nabla u \right\|_{L^{2}}^{2}$: 
 
\begin{lemma}\label{11/08/13/12:54}
Assume $d\ge 3$ and conditions \eqref{11/07/18/9:45}--\eqref{11/05/02/8:40}. Let $u \in H^{1}(\mathbb{R}^{d})$ be a function with $\mathcal{K}(u)\ge 0$. Then, we have  
\begin{equation}\label{11/08/13/12:55}
\left\| \nabla u \right\|_{L^{2}}^{2}
\lesssim \mathcal{H}(u).
\end{equation}
\end{lemma}

\begin{proof}[Proof of Lemma \ref{11/08/13/12:54}]
Using the assumption \eqref{11/04/30/22:39}, we have 
\begin{equation}\label{11/08/13/12:59}
\int_{\mathbb{R}^{d}}(DF-2F)(u)
\ge 
\left(\frac{4}{d}+\varepsilon_{0} \right) \int_{\mathbb{R}^{d}}F(u)\ge 0. 
\end{equation}
Put $C_{0}':=\max\{ \frac{4}{4+d\varepsilon_{0}}, \frac{d-2}{d}\}$. Clearly, $C_{0}'<1$. Then, it follows from \eqref{11/08/13/12:59} and $\mathcal{K}(u)\ge 0$  that 
\begin{equation}\label{11/08/13/13:03}
\begin{split}
\mathcal{H}(u)
&=
\frac{1}{2}\left\| \nabla u \right\|_{L^{2}}^{2}
-\frac{1}{2}\int_{\mathbb{R}^{d}}F(u) 
-\frac{1}{2^{*}}\left\| u \right\|_{L^{2^{*}}}^{2^{*}}
\\[6pt]
&\ge 
\frac{1}{2}\left\| \nabla u \right\|_{L^{2}}^{2}
- 
\frac{1}{2} 
\left( \frac{d}{4+d \varepsilon_{0}}
\right)
\int_{\mathbb{R}^{d}}(DF-2F)(\psi(t))
-\frac{1}{2^{*}}\left\| u \right\|_{L^{2^{*}}}^{2^{*}}
\\[6pt]
&\ge 
\frac{1}{2}\left( 1 -C_{0}'\right) 
 \left\| \nabla u \right\|_{L^{2}}^{2},
\end{split}
\end{equation} 
which gives the desired result. 
\end{proof}

The following lemma tells us that $A_{\omega,+}$ is invariant under the flow defined by \eqref{11/06/12/9:08}. Strongly, $\mathcal{K}$ of a solution in $A_{\omega,+}$ is positive uniformly in time:
\begin{lemma}\label{11/06/27/22:29}
Let $\psi$ be a solution to \eqref{11/06/12/9:08} starting from $A_{\omega,+}$, and let $I_{\max}$ be the maximal interval where $\psi$ exists.  Then, we have 
\begin{align}\label{11/06/27/22:31}
&\psi(t) \in A_{\omega,+} 
\qquad 
\mbox{for any $t \in I_{\max}$}, 
\\[6pt]
\label{11/06/27/22:34}
&\inf_{t \in I_{\max}} \mathcal{K}(\psi(t))>0.
\end{align}
\end{lemma}
\begin{proof}[Proof of Lemma \ref{11/06/27/22:29}]
The claim \eqref{11/06/27/22:31} easily follows from the action conservation law and the definition of $m_{\omega}$. 
\par 
We shall prove \eqref{11/06/27/22:34}. Put 
\begin{equation}\label{11/09/04/14:20}
\widetilde{F}(u):=F(u)+\frac{d-2}{d}|u|^{2^{*}}. 
\end{equation}
Then, we have 
\begin{align}
\label{11/09/04/14:22}
&(D-2_{*}-\widetilde{\varepsilon}_{0})\widetilde{F} \ge 0, 
\\[6pt]
\label{11/09/04/14:23}
&(D-2)(D-2_{*}-\widetilde{\varepsilon}_{0})\widetilde{F} \ge 0,
\end{align}
where $\widetilde{\varepsilon}_{0}:=\min\{\varepsilon_{0}, \frac{2}{d-2}\}$.
Let $\psi$ be a solution to \eqref{11/06/12/9:08} starting from $A_{\omega,+}$. It is easy to verify that  
\begin{equation}\label{11/09/02/16:04}
\begin{split}
&\frac{d^{2}}{d\lambda^{2}}\mathcal{S}_{\omega}(T_{\lambda}\psi(t))
\\[6pt]
&=
-\frac{1}{\lambda^{2}}\mathcal{K}(T_{\lambda}\psi(t))
+\frac{2}{\lambda^{2}}\left\|\nabla T_{\lambda}\psi(t) \right\|_{L^{2}}^{2}
-\frac{d}{2\lambda^{2}}
\int_{\mathbb{R}^{d}}\frac{d}{4}(D^{2}\widetilde{F}-4D\widetilde{F}+4\widetilde{F})(T_{\lambda}u).
\end{split}
\end{equation}
Combining \eqref{11/09/02/16:04} with \eqref{11/09/04/14:23}, we obtain 
\begin{equation}\label{11/09/04/14:12}
\begin{split}
&\frac{d^{2}}{d\lambda^{2}}\mathcal{S}_{\omega}(T_{\lambda}\psi(t))
\\[6pt]
&\le 
-\frac{1}{\lambda^{2}}\mathcal{K}(T_{\lambda}\psi(t))
+\frac{2}{\lambda^{2}}\left\|\nabla T_{\lambda}\psi(t) \right\|_{L^{2}}^{2}
-\frac{d}{2\lambda^{2}}
\int_{\mathbb{R}^{d}}\left(1+\frac{d\widetilde{\varepsilon}_{0}}{4}\right)(D\widetilde{F}-2\widetilde{F})(T_{\lambda}\psi(t))
\\[6pt]
&=-\frac{1}{\lambda^{2}}\mathcal{K}(T_{\lambda}\psi(t))
+\frac{2}{\lambda^{2}}
\left\{ \mathcal{K}(T_{\lambda}\psi(t))
-\frac{d}{4}
\int_{\mathbb{R}^{d}}\frac{d\widetilde{\varepsilon}_{0}}{4}(D\widetilde{F}-2\widetilde{F})(T_{\lambda}\psi(t))
\right\}
\\[6pt]
&\hspace{240pt} 
\mbox{for any $\lambda>0$ and $t \in I_{\max}$}.
\end{split}
\end{equation}
Suppose here that  
\begin{equation}\label{11/09/04/14:43}
\mathcal{K}(\psi(t))
-\frac{d}{4}
\int_{\mathbb{R}^{d}}\frac{d\widetilde{\varepsilon}_{0}}{4}(D\widetilde{F}-2\widetilde{F})(\psi(t))\ge 0
\qquad 
\mbox{for any $t \in I_{\max}$}.
\end{equation}
Then,  we have     
\begin{equation}\label{11/09/04/14:49}
\left(1+\frac{d\widetilde{\varepsilon}_{0}}{4}\right)\mathcal{K}(\psi(t))
\ge \frac{d\widetilde{\varepsilon}_{0}}{4}\left\| \nabla \psi(t)\right\|_{L^{2}}^{2}
\gtrsim 
\mathcal{H}(\psi(t))
\gtrsim 1
\qquad 
\mbox{for any $t \in I_{\max}$},
\end{equation}
where we have used the Hamiltonian conservation law and Lemma \ref{11/06/28/10:12} to derive the final inequality. Thus, \eqref{11/06/27/22:34} holds in this case. 
\par 
On the other hand, suppose that there exists $t_{0}\in I_{\max}$ such that 
\begin{equation}\label{11/09/04/14:51}
\mathcal{K}(\psi(t_{0}))
-\frac{d}{4}
\int_{\mathbb{R}^{d}}\frac{d\widetilde{\varepsilon}_{0}}{4}(D\widetilde{F}-2\widetilde{F})(\psi(t_{0}))
< 0. 
\end{equation}
We see from the Sobolev embedding, \eqref{11/09/04/14:51}, \eqref{11/04/30/22:54}, \eqref{11/05/02/8:40} and the H\"older inequality  that  
\begin{equation}\label{11/09/04/22:23}
\left\| \psi(t_{0})\right\|_{L^{2^{*}}}^{2}
\lesssim 
\left\| \nabla \psi(t_{0})\right\|_{L^{2}}^{2}
\lesssim
\sum_{j=1}^{2} 
\left\| \psi(t_{0}) \right\|_{L^{2}}^{p_{j}+1-\frac{d(p_{j}-1)}{2}}
\left\| \psi(t_{0}) \right\|_{L^{2^{*}}}^{\frac{d(p_{j}-1)}{2}}
+
\left\| \psi(t_{0})\right\|_{L^{2^{*}}}^{2^{*}},
\end{equation}
which together with the mass conservation law gives us that 
\begin{equation}\label{11/09/04/22:26}
\left\|\psi(t_{0}) \right\|_{L^{2^{*}}}\gtrsim 1. 
\end{equation}
Let $\lambda(t_{0})$ be a number such that 
\begin{equation}\label{11/09/04/21:58}
\mathcal{K}(T_{\lambda(t_{0})}\psi(t_{0}))=0 .
\end{equation}
Then, Lemma \ref{11/06/27/22:43}, \eqref{11/09/04/21:58} and \eqref{11/09/04/22:26} show that    
\begin{equation}\label{11/09/04/21:59}
\lambda(t_{0})^{2}
\gtrsim 
\lambda(t_{0})^{2}\left\|\nabla \psi(t_{0}) \right\|_{L^{2}}^{2}
\ge 
\lambda(t_{0})^{2^{*}}
\left\| \psi(t_{0}) \right\|_{L^{2^{*}}}^{2^{*}}
\gtrsim 
\lambda(t_{0})^{2^{*}}. 
\end{equation}
Hence, we have 
\begin{equation}\label{11/09/04/22:02}
\lambda(t_{0})\lesssim 1. 
\end{equation}
Now, we see from \eqref{11/09/04/14:51} and \eqref{11/09/05/10:54} that\begin{equation}\label{11/09/04/22:28}
\begin{split}
&\frac{1}{\lambda^{2}}
\left\{ \mathcal{K}(T_{\lambda}\psi(t_{0}))
-\frac{d}{4}
\int_{\mathbb{R}^{d}}\frac{d\widetilde{\varepsilon}_{0}}{4}(D\widetilde{F}-2\widetilde{F})(T_{\lambda}\psi(t_{0}))
\right\}
\\[6pt]
&=
\left\| \nabla \psi(t_{0})) \right\|_{L^{2}}^{2}
-\frac{d}{4}\left( 1+ \frac{d\widetilde{\varepsilon}_{0}}{4}\right)
\frac{1}{\lambda^{2}}
\int_{\mathbb{R}^{d}}(D\widetilde{F}-2\widetilde{F})(T_{\lambda}\psi(t_{0}))
< 0
\quad 
\mbox{for any $\lambda \ge1$}. 
\end{split}
\end{equation}
Hence, \eqref{11/09/04/14:12} together with \eqref{11/05/01/22:06} in Lemma \ref{11/07/20/18:01} and \eqref{11/09/04/22:28} shows   
\begin{equation}\label{11/09/04/15:22}
\frac{d^{2}}{d\lambda^{2}}\mathcal{S}_{\omega}(T_{\lambda}\psi(t_{0}))
< 
-\frac{1}{\lambda^{2}}\mathcal{K}(T_{\lambda}\psi(t_{0}))
< 0
\quad 
\mbox{for any $1\le \lambda \le \lambda(t_{0}) $}. 
\end{equation}
Combining \eqref{11/09/04/15:22} with \eqref{11/09/04/22:02}, we obtain that 
\begin{equation}\label{11/09/04/15:34}
\begin{split}
\mathcal{K}(\psi(t_{0}))  
&\gtrsim 
(\lambda(\psi(t_{0}))-1)\mathcal{K}(\psi(t_{0}))
=
 (\lambda(\psi(t_{0}))-1)\frac{d}{d\lambda}\mathcal{S}_{\omega}(T_{\lambda}\psi(t_{0}))\bigg|_{\lambda=1}
\\[6pt]
&\ge 
\mathcal{S}_{\omega}(T_{\lambda(t_{0})}\psi(t_{0}))
- 
\mathcal{S}_{\omega}(\psi(t_{0}))
\ge m_{\omega}-\mathcal{S}_{\omega}(\psi(t_{0}))\gtrsim 1.
\end{split}
\end{equation}
This completes the proof. 
\end{proof}

\subsection{Extraction of critical element}
\label{11/06/10/17:06}
In view of Theorem \ref{10/10/04/21:44} (v), it suffices for Theorem \ref{11/06/10/17:09} to show that any solution $\psi$ to \eqref{11/06/12/9:08} starting from $A_{\omega,+}$ satisfies $\|\psi \|_{W_{p_{1}}(I_{\max})\cap W(I_{\max})}<\infty$, where $I_{\max}$ denotes the maximal interval where $\psi$ exists. To this end,  for $m>0$, we put   
\begin{equation}
\label{11/06/10/17:13}
\tau_{\omega}(m):=\sup
\left\{ \left\| \psi \right\|_{W_{p_{1}}(I_{\max})\cap W(I_{\max})}
\colon 
\begin{array}{l}
\mbox{$\psi$ is a solution to (\ref{11/06/12/9:08}) such that} 
\\[6pt]
\mbox{$\psi \in A_{\omega,+}$ and $\mathcal{S}_{\omega}(\psi)\le m$}
\end{array}
\right\}  
\end{equation}
and define 
\begin{equation}
\label{11/06/10/17:26}
m_{\omega}^{*}:=\sup\left\{ m>0 \colon \tau_{\omega}(m)<\infty 
\right\}.
\end{equation}
These quantities were used in \cite{IMN, Killip-Visan}. It follows from the existence of a ground state for \eqref{11/05/01/17:30} that $m_{\omega}^{*}\le m_{\omega}$. 
\par 
Our aim is to show that $m_{\omega}^{*}= m_{\omega}$. Here, let $\psi$ be a solution to \eqref{11/06/12/9:08} such that $\psi \in A_{\omega,+}$ and $\mathcal{S}_{\omega}(\psi)\le m$. If $m$ is sufficiently small, then Lemma \ref{11/06/27/22:43} shows 
$\left\|\psi(t) \right\|_{H^{1}}\ll 1$. Hence, we see from Theorem \ref{10/10/04/21:44} (i) that $m_{\omega}^{*}>0$. 
\par 
Now, we suppose the contrary that $m_{\omega}^{*}<m_{\omega}$. Then, we shall show the existence of the so-called critical element.
To this end, we employ the following  result for the equation \eqref{11/07/17/16:01} (see Corollary 1.9 in \cite{Killip-Visan}. See also \cite{Kenig-Merle}), which causes the restriction $d\ge 5$ in Theorem \ref{11/06/10/17:09}:
\begin{theorem}\label{11/07/24/20:56}
Assume $d\ge 5$. Put    
\begin{equation}\label{11/07/31/18:17}
A_{0}:=\left\{
u \in \dot{H}^{1}(\mathbb{R}^{d})
\biggm| 
\mathcal{H}_{0}(u)<\frac{1}{d}\sigma^{\frac{d}{2}},
\ 
\left\| \nabla u \right\|_{L^{2}}^{2}< \sigma^{\frac{d}{2}}
\right\}.
\end{equation}
Then, any solution $\psi$ to \eqref{11/07/17/16:01} starting from $A_{0}$ exists globally in time and satisfies 
\begin{equation}
\psi(t) \in A_{0} \quad 
\mbox{for any $t \in \mathbb{R}$},
\qquad 
\label{11/07/31/23:00}
\sup_{t\in \mathbb{R}}\left\| \nabla \psi(t) \right\|_{L^{2}}^{2}\le \sigma^{\frac{d}{2}}
\end{equation}
and 
\begin{equation}\label{11/07/24/21:03}
\left\| \psi \right\|_{W(\mathbb{R})}<\infty. 
\end{equation}
\end{theorem}

Under the hypothesis $m_{\omega}^{*}<m_{\omega}$, we can take a sequence $\{\psi_{n}\}$ of solutions to \eqref{11/06/12/9:08} such that 
\begin{align}
\label{11/06/29/16:03}
&\psi_{n}(t)\in A_{\omega,+} 
\qquad 
\mbox{for any $t \in I_{n}$},
\\[6pt] 
\label{11/06/11/6:16}
&\lim_{n\to \infty}\mathcal{S}_{\omega}(\psi_{n})=m_{\omega}^{*},
\\[6pt]
\label{11/06/29/16:04}
&\left\| \psi_{n} \right\|_{W_{p_{1}}(I_{n})\cap W(I_{n})}=\infty,\end{align}
where $I_{n}$ denotes the maximal interval where $\psi_{n}$ exists (by time-translation, we may assume that each $I_{n}$ contains $0$). We also see from Lemma \ref{11/06/27/22:43} that 
\begin{equation}
\label{11/07/03/14:33}
\sup_{n\in \mathbb{N}}\left\| \psi_{n}(t)\right\|_{L^{\infty}(I_{n},H^{1})}^{2}
\lesssim m_{\omega}+\frac{m_{\omega}}{\omega}.
\end{equation}
We apply the profile decomposition (see Theorem 1.6 in \cite{Keraani}) to the sequence $\displaystyle{
\bigm\{|\nabla|^{-1} \langle \nabla \rangle \psi_{n}(0)\bigm\}}$ and obtain some subsequence of $\{\psi_{n}(0)\}$ (still denoted by the same symbol) with the following property: there exists   
\\
{\rm (i)} a family $\{ \widetilde{u}^{1}, \widetilde{u}^{2},\ldots \}$ of functions in $H^{1}(\mathbb{R}^{d})$ (each $\widetilde{u}^{j}$ is called the {\it linear profile}),   
\\
{\rm (ii)} a family $\displaystyle{
\bigm\{ 
\{(x_{n}^{1}, t_{n}^{1},\lambda_{n}^{1}) \},
\{(x_{n}^{2}, t_{n}^{2},\lambda_{n}^{2}) \},
\ldots 
\bigm\}}$ of sequences in 
$\mathbb{R}^{d}\times \mathbb{R} \times (0,1]$ with 
\begin{align}
\label{10/09/21/11:15}
&\lim_{n\to \infty}t_{n}^{j}=t_{\infty}^{j} \in \mathbb{R}
\cup \{\pm \infty \},
\\[6pt]
&
\lim_{n\to \infty}\lambda_{n}^{j}=\lambda_{\infty}^{j} \in \{0,1\}
, 
\qquad 
\lambda_{n}^{j}\equiv 1 \quad \mbox{if $\lambda_{\infty}^{j}=1$},
\\[6pt]
\label{11/07/03/10:21}
&\lim_{n\to \infty}
\left\{ 
 \frac{\lambda_{n}^{j'}}{\lambda_{n}^{j}}
+
\frac{\lambda_{n}^{j}}{\lambda_{n}^{j'}}
+
\frac{ \left| x_{n}^{j}-x_{n}^{j'} \right|}{\lambda_{n}^{j}} 
+ 
\frac{\left| t_{n}^{j}-t_{n}^{j'} \right|}
{\left( \lambda_{n}^{j}\right)^{2}}
\right\}=+\infty
\quad 
\mbox{for any $j'\neq j$},
\end{align}
{\rm (iii)} a family $\{w_{n}^{1},w_{n}^{2},\ldots\}$ of functions in $H^{1}(\mathbb{R}^{d})$ with 
\begin{equation}\label{11/07/06/17:55}
\lim_{j\to \infty}\lim_{n\to \infty}
\left\| |\nabla|^{-1}\langle \nabla \rangle e^{it\Delta} w_{n}^{j}\right\|_{L^{r}(\mathbb{R},L^{q})}
=0
\qquad 
\mbox{for any $\dot{H}^{1}$-admissible pair $(q,r)$},
\end{equation}
such that, defining the transformations $g_{n}^{j}$ and $G_{n}^{j}$ by  
\begin{align}
\label{11/07/25/9:17}
(g_{n}^{j}u)(x)
&:=\frac{1}{(\lambda_{n}^{j})^{\frac{d-2}{2}}}
u \biggm(
\frac{x-x_{n}^{j}}{\lambda_{n}^{j}}
\biggm),
\\[6pt]
\label{11/07/24/16:12}
(G_{n}^{j}v)(x,t)
&:=\frac{1}{(\lambda_{n}^{j})^{\frac{d-2}{2}}}
v \biggm(
\frac{x-x_{n}^{j}}{\lambda_{n}^{j}},
\frac{t-t_{n}^{j}}{(\lambda_{n}^{j})^{2}}
\biggm), 
\end{align}
we have 
\begin{equation}\label{11/06/22/22:27}
\begin{split}
e^{it\Delta}\psi_{n}(0)
&=
\sum_{j=1}^{k}  \langle \nabla \rangle^{-1}|\nabla|
G_{n}^{j}\bigm(e^{it \Delta}   |\nabla|^{-1} \langle \nabla \rangle \widetilde{u}^{j}\bigm)
+
e^{it\Delta}w_{n}^{k}
\\[6pt]
&=
\sum_{j=1}^{k}
G_{n}^{j}\biggm(
\frac{\langle (\lambda_{n}^{j})^{-1}\nabla \rangle^{-1}
\langle \nabla \rangle}{\lambda_{n}^{j}}
e^{it \Delta}\widetilde{u}^{j}\biggm)
+
e^{it\Delta}w_{n}^{k}
\\[6pt]
&=
\sum_{j=1}^{k}\frac{
\langle \nabla \rangle^{-1}\langle \lambda_{n}^{j}\nabla \rangle}{\lambda_{n}^{j}}
e^{i(t-t_{n}^{j})\Delta}g_{n}^{j}\widetilde{u}^{j}
+
e^{it\Delta}w_{n}^{k}
\qquad \mbox{for any $k \in \mathbb{N}$}.
\end{split}
\end{equation} 
Note here that  for any Fourier multiplier $\mu(\nabla)$ and the $L^{2}$-scaling operator $T_{\lambda}$, we have 
\begin{equation}\label{11/07/24/15:19}
\mu(\nabla )T_{\lambda} = T_{\lambda} \mu(\lambda \nabla). 
\end{equation}
Besides, putting  
\begin{equation}
\label{11/07/24/12:34}
\sigma_{n}^{j}
:=
\frac{ 
\langle (\lambda_{n}^{j})^{-1} \nabla \rangle^{-1}\langle \nabla \rangle}
{\lambda_{n}^{j}},
\end{equation}
for any $k \in \mathbb{N}$ and $s=0,1$, we have the expansions:
\begin{align}
\label{10/09/24/21:40} 
&\lim_{n\to \infty}
\left\{ 
\left\| |\nabla|^{s} \psi_{n}(0) \right\|_{L^{2}}^{2}
-
\sum_{j=1}^{k}
\left\||\nabla|^{s} g_{n}^{j}
\biggm(
\sigma_{n}^{j}
e^{-i\frac{t_{n}^{j}}{(\lambda_{n}^{j})^{2}}\Delta}
\widetilde{u}^{j}
\biggm) \right\|_{L^{2}}^{2}
-
\left\| |\nabla|^{s} w_{n}^{k} \right\|_{L^{2}}^{2}
\right\}
=0,
\\[6pt]
\label{11/08/13/16:22}
&
\lim_{n\to \infty}
\left\{ 
\mathcal{H}(\psi_{n}(0))
-
\sum_{j=1}^{k} \mathcal{H}_{\omega}
\biggm(
g_{n}^{j}
\biggm(
\sigma_{n}^{j}
e^{-i\frac{t_{n}^{j}}{(\lambda_{n}^{j})^{2}}\Delta}
\widetilde{u}^{j}
\biggm)
\biggm)
-
\mathcal{H}(w_{n}^{k})
\right\}
=
0,
\\[6pt]
\label{11/06/27/22:18}
&
\lim_{n\to \infty}
\left\{ 
\mathcal{S}_{\omega}(\psi_{n}(0))
-
\sum_{j=1}^{k} \mathcal{S}_{\omega}
\biggm(
g_{n}^{j}
\biggm(
\sigma_{n}^{j}
e^{-i\frac{t_{n}^{j}}{(\lambda_{n}^{j})^{2}}\Delta}
\widetilde{u}^{j}
\biggm)
\biggm)
-
\mathcal{S}_{\omega}(w_{n}^{k})
\right\}
=
0,
\\[6pt]
\label{11/06/27/22:20}
&
\lim_{n\to \infty}
\left\{ 
\mathcal{I}_{\omega}(\psi_{n}(0))
-
\sum_{j=1}^{k} \mathcal{I}_{\omega}
\biggm(
g_{n}^{j}
\biggm(
\sigma_{n}^{j}
e^{-i\frac{t_{n}^{j}}{(\lambda_{n}^{j})^{2}}\Delta}
\widetilde{u}^{j}
\biggm)
\biggm)
-
\mathcal{I}_{\omega}(w_{n}^{k})
\right\}
=
0,
\\[6pt]
\label{11/06/27/22:22}
&
\lim_{n\to \infty}
\left\{ 
\mathcal{K}(\psi_{n}(0))
-
\sum_{j=1}^{k} \mathcal{K}
\biggm(
g_{n}^{j}
\biggm(
\sigma_{n}^{j}
e^{-i\frac{t_{n}^{j}}{(\lambda_{n}^{j})^{2}}\Delta}
\widetilde{u}^{j}
\biggm)
\biggm)
-
\mathcal{K}(w_{n}^{k})
\right\}
=
0.
\end{align}
Note here that \eqref{10/09/24/21:40}  together with \eqref{11/07/03/14:33} yields
\begin{equation}\label{11/07/17/10:57}
\sup_{k\in \mathbb{N}}
\limsup_{n\to \infty}
\left\|\langle \nabla \rangle e^{it\Delta} w_{n}^{k} \right\|_{S(\mathbb{R})}
\lesssim 
\sup_{k\in \mathbb{N}}
\limsup_{n\to \infty}
\left\| w_{n}^{k} \right\|_{H^{1}}
< \infty .
\end{equation}

Next, we define the {\it nonlinear profile}.  
Let $U_{n}^{j}$ be the solution to \eqref{11/06/12/9:08} with $U_{n}^{j}(0)=\frac{\langle \nabla \rangle^{-1}\langle \lambda_{n}^{j}\nabla \rangle}{\lambda_{n}^{j}}e^{-it_{n}^{j}\Delta}g_{n}^{j}\widetilde{u}^{j}$, so that $e^{it\Delta}U_{n}^{j}(0)=G_{n}^{j}(\sigma_{n}^{j}e^{it\Delta}\widetilde{u}^{j})$ (see \eqref{11/06/22/22:27}). Thus, $U_{n}^{j}$ satisfies 
\begin{equation}\label{11/07/24/16:08}
U_{n}^{j}(t)=
G_{n}^{j}
(\sigma_{n}^{j}e^{it\Delta}\widetilde{u}^{j})
+
i\int_{0}^{t}e^{i(t-t')\Delta}\bigm\{ f(U_{n}^{j}) + |U_{n}^{j}|^{2^{*}-2}U_{n}^{j}\bigm\}(t')\,dt'.
\end{equation}
Undoing the transformations $G_{n}^{j}$ and $\sigma_{n}^{j}$ in \eqref{11/07/24/16:08}, we have the equation 
\begin{equation}\label{11/07/24/16:57}
\begin{split}
\widetilde{\psi}_{n}^{j}(t)
&=
e^{it \Delta}\widetilde{u}^{j}
\\[6pt]
&\quad +i
\int_{-\frac{t_{n}^{j}}{(\lambda_{n}^{j})^{2}}}^{t}
\hspace{-4pt}
e^{i(t-t')\Delta}
(\sigma_{n}^{j})^{-1}
\bigg\{ 
(\lambda_{n}^{j})^{\frac{d+2}{2}}f\biggm(
(\lambda_{n}^{j})^{-\frac{d-2}{2}}\sigma_{n}^{j}\widetilde{\psi}_{n}^{j}
\biggm)
+
|\sigma_{n}^{j}\widetilde{\psi}_{n}^{j}|^{2^{*}-2}\sigma_{n}^{j}\widetilde{\psi}_{n}^{j}
\bigg\}(t')\,dt',
\end{split}
\end{equation}
where 
\begin{equation}\label{11/07/24/16:54}
\widetilde{\psi}_{n}^{j}:=(\sigma_{n}^{j})^{-1}(G_{n}^{j})^{-1}U_{n}^{j}. 
\end{equation}
We define the nonlinear profile $\widetilde{\psi}^{j}$ as a solution to the limit equation of \eqref{11/07/24/16:57}: 
\begin{equation}\label{11/07/02/11:49}
\begin{split}
\widetilde{\psi}^{j}(t)
&=
e^{it \Delta}
\widetilde{u}^{j}
\\[6pt]
&\quad +i
\int_{-\frac{t_{\infty}^{j}}{(\lambda_{\infty}^{j})^{2}}}^{t}
\hspace{-3pt}e^{i(t-t')\Delta}
(\sigma_{\infty}^{j})^{-1}\bigg\{ 
(\lambda_{\infty}^{j})^{\frac{d+2}{2}}f\biggm(
(\lambda_{\infty}^{j})^{-\frac{d-2}{2}}\sigma_{\infty}^{j}\widetilde{\psi}^{j}
\biggm)
+
|\sigma_{\infty}^{j}\widetilde{\psi}^{j}|^{2^{*}-1}
\sigma_{\infty}^{j}\widetilde{\psi}^{j}
\bigg\}(t')\,dt'
\\[6pt]
&=
e^{it \Delta}
\widetilde{u}^{j}
+i
\int_{-\frac{t_{\infty}^{j}}{(\lambda_{\infty}^{j})^{2}}}^{t}e^{i(t-t')\Delta}
(\sigma_{\infty}^{j})^{-1}
\mathcal{N}_{j}(\sigma_{\infty}^{j}\widetilde{\psi}^{j}(t'))\,dt',
\end{split}
\end{equation}
where 
\begin{equation}\label{11/07/24/18:07}
\sigma_{\infty}^{j}
:=
\left\{ \begin{array}{ccc}
1 &\mbox{if}& \lambda_{\infty}^{j}=1,
\\[6pt]
|\nabla |^{-1} \langle \nabla \rangle &\mbox{if}& \lambda_{\infty}^{j}=0
\end{array} \right.
\end{equation}
and 
\begin{equation}\label{11/07/02/13:50}
\mathcal{N}_{j}(u)
:=
\left\{ \begin{array}{ccc}
f(u) +|u|^{2^{*}-2}u
&\mbox{if}& \lambda_{\infty}^{j}=1,
\\[6pt]
|u|^{2^{*}-2}u
&\mbox{if}& \lambda_{\infty}^{j}=0.
\end{array}
\right. 
\end{equation}
When $-\displaystyle{\frac{t_{\infty}^{j}}{(\lambda_{\infty}^{j})^{2}}} \in \{\pm \infty\}$, we regard \eqref{11/07/02/11:49} as the final value problem at $\pm \infty$. 
\par 
Let $I^{j}:=(T_{\min}^{j},T_{\max}^{j})$ be the maximal interval where the nonlinear profile $\widetilde{\psi}^{j}$ exists. Note that $T_{\min}^{j}=-\infty$ when  $t_{\infty}^{j}/ (\lambda_{\infty}^{j})^{2}=+\infty$, and $T_{\max}^{j}=+\infty$ when 
$t_{\infty}^{j}/(\lambda_{\infty}^{j})^{2}=-\infty$. We see from the construction of the nonlinear profile that $\widetilde{\psi}^{j} \in C(I^{j},H^{1}(\mathbb{R}^{d}))$ and   
\begin{equation}
\label{11/06/22/22:37}
\lim_{n\to \infty}\left\| \widetilde{\psi}^{j}\biggm(-\frac{t_{n}^{j}}{(\lambda_{n}^{j})^{2}}\biggm)-
\, e^{-i\frac{t_{n}^{j}}{(\lambda_{n}^{j})^{2}}\Delta} 
\widetilde{u}^{j}
 \right\|_{H^{1}}=0.
\end{equation} 

\begin{lemma}\label{11/08/11/11:56}
There exists $\delta>0$ with the following property: Let $j \in \mathbb{N}$ and assume that 
\begin{equation}\label{11/08/11/11:57}
\left\| \widetilde{u}^{j} \right\|_{H^{1}}\le \delta.
\end{equation}
Then, we have $I^{j}=\mathbb{R}$ and 
\begin{equation}\label{11/08/11/11:58}
\left\|\langle \nabla \rangle \widetilde{\psi}^{j} \right\|_{S(\mathbb{R})}
\lesssim 
\left\|\widetilde{u}^{j} \right\|_{H^{1}}.
\end{equation}
\end{lemma}

\begin{proof}[Proof of Lemma \ref{11/08/11/11:56}] 
This lemma follows from the standard small-data well-posedness theory.
\end{proof}
Moreover, we can verify the following lemma in a way similar to Lemma \ref{11/05/14/14:42}:
\begin{lemma}\label{11/08/07/22:04}
Assume $d\ge 3$. Let $I$ be an interval, $A,B>0$, and
 let $u$ be a function such that 
\begin{equation}\label{11/08/07/22:08}
\left\| u \right\|_{L^{\infty}(I,H^{1})}\le A,
\qquad 
\left\| u \right\|_{W_{p_{1}}(I) \cap W(I)}\le B.
\end{equation}
Let $\varepsilon>0$ and suppose that  
\begin{equation}\label{11/08/07/22:09}
\left\|\langle \nabla \rangle \left( 
i\frac{\partial }{\partial t}u +\Delta u + \sigma_{\infty}^{-1}\mathcal{N}_{j}(\sigma_{\infty}^{j}u)
\right) 
\right\|_{L^{\frac{2(d+2)}{d+4}}(I,L^{\frac{2(d+2)}{d+4}})}
\le \varepsilon.
\end{equation} 
Then, we have 
\begin{equation}\label{11/08/07/22:10}
\left\| \langle \nabla \rangle u \right\|_{S(I)}
\lesssim C(A,B) + \varepsilon,
\end{equation}
where $C(A,B)$ is some constant depending on $A$ and $B$. 
\end{lemma}

\begin{lemma}[Properties of nonlinear profiles]
\label{11/07/07/20:34}
Let $\widetilde{\psi}^{j_{1}}$ be a nonlinear profile, and suppose that it is non-trivial. Then,  we have 
\begin{equation}\label{11/07/07/20:37}
\sigma_{\infty}^{j_{1}} \widetilde{\psi}^{j_{1}}(t)
\in \left\{ \begin{array}{lcc}
 A_{\omega,+}
&\mbox{if } & \lambda_{\infty}^{j_{1}}=1,
\\[6pt]
 A_{0}
&\mbox{if } & \lambda_{\infty}^{j_{1}}=0
\end{array}
\right.
\qquad 
\mbox{for any $t \in I^{j_{1}}$}.
\end{equation}
\end{lemma}
\begin{proof}[Proof of Lemma \ref{11/07/07/20:34}] 
Note first that when $\widetilde{\psi}^{j_{1}}$ is non-trivial, we see from \eqref{11/06/22/22:37} that the corresponding linear profile $\widetilde{u}^{j_{1}}$ is also non-trivial.
\par  
We shall show that 
\begin{align}
\label{11/07/07/21:05}
&\mathcal{I}_{\omega}
\bigg( g_{n}^{j} \bigg( 
\sigma_{n}^{j}e^{-i\frac{t_{n}^{j}}{(\lambda_{n}^{j})^{2}}\Delta}
\widetilde{u}^{j}
\bigg)\bigg)
< \frac{m_{\omega}+m_{\omega}^{*}}{2} 
\quad
\mbox{for any $j$ and sufficiently large $n \in \mathbb{N}$}
,
\\[6pt]
\label{11/07/07/21:36}
&\mathcal{K}\biggm( g_{n}^{j} \bigg( 
\sigma_{n}^{j}e^{-i\frac{t_{n}^{j}}{(\lambda_{n}^{j})^{2}}\Delta}
\widetilde{u}^{j}
\bigg)\biggm)
>0
\quad 
\mbox{for any $j\ge 1$ and sufficiently large $n \in \mathbb{N}$},\\[6pt]
&\label{11/07/07/21:32}
\mathcal{S}_{\omega}
\bigg( g_{n}^{j} \bigg( 
\sigma_{n}^{j}e^{-i\frac{t_{n}^{j}}{(\lambda_{n}^{j})^{2}}\Delta}
\widetilde{u}^{j}
\bigg)\bigg)
< \frac{m_{\omega}+m_{\omega}^{*}}{2} 
\quad
\mbox{for any $j$ and sufficiently large $n \in \mathbb{N}$}
.
\end{align}
It follows from \eqref{11/06/11/6:16}, $\mathcal{K}(\psi_{n}(0))>0$ and \eqref{11/06/27/22:20} that 
\begin{equation}\label{11/06/27/22:41}
\begin{split}
m_{\omega}^{*}+o_{n}(1) 
&= \mathcal{S}_{\omega}(\psi_{n}(0))
\\[6pt]
&
\ge 
\mathcal{I}_{\omega}(\psi_{n}(0))
=\sum_{j=1}^{k}\mathcal{I}_{\omega}
\biggm( g_{n}^{j} \bigg( 
\sigma_{n}^{j}e^{-i\frac{t_{n}^{j}}{(\lambda_{n}^{j})^{2}}\Delta}
\widetilde{u}^{j}
\bigg)\biggm)
+\mathcal{I}_{\omega}(w_{n}^{k})+o_{n}(1).
\end{split}
\end{equation}
Hence, \eqref{11/06/27/22:41} together with $m_{\omega}^{*}<m_{\omega}$ and the positivity of $\mathcal{I}_{\omega}$ shows 
\eqref{11/07/07/21:05}. Moreover, \eqref{11/07/07/21:05} together with the definition of $\widetilde{m}_{\omega}$ (see \eqref{11/04/30/23:54}) shows \eqref{11/07/07/21:36}.
\par 
Note here that 
\begin{equation}\label{11/07/09/10:00}
\liminf_{n\to \infty}\left\| 
\nabla g_{n}^{j} \bigg( 
\sigma_{n}^{j}e^{-i\frac{t_{n}^{j}}{(\lambda_{n}^{j})^{2}}\Delta}
\widetilde{u}^{j}
\bigg)
\right\|_{L^{2}}
\hspace{-6pt}=
\left\| \nabla \sigma_{\infty}^{j}
\widetilde{u}^{j}\right\|_{L^{2}}
>0,
\quad 
\mbox{provided that $\widetilde{u}^{j}$ is non-trivial}.
\end{equation}
Hence, we see from Lemma \ref{11/06/28/10:12} together with \eqref{11/07/09/10:00} that 
\begin{equation}
\label{11/07/07/21:47}
\mathcal{S}_{\omega}
\biggm( g_{n}^{j} \bigg( 
\sigma_{n}^{j}e^{-i\frac{t_{n}^{j}}{(\lambda_{n}^{j})^{2}}\Delta}
\widetilde{u}^{j}
\bigg)\biggm)\ge 0
\qquad 
\mbox{for any $j\ge 1$ and sufficiently large $n \in \mathbb{N}$}
.
\end{equation}
Thus, \eqref{11/06/27/22:18} together with 
\eqref{11/06/11/6:16} and \eqref{11/07/07/21:47} gives us 
\eqref{11/07/07/21:32}. 
\par 
We shall prove \eqref{11/07/07/20:37}. Suppose that $\lambda_{\infty}^{j_{1}}=1$, so that $\sigma_{\infty}^{j}=1$.  Then, it follows from \eqref{11/06/22/22:37} that 
\begin{equation}\label{11/07/09/20:20}
\lim_{n\to \infty}
\left\| (\sigma_{\infty}^{j}
\widetilde{\psi}^{j_{1}})\biggm( -\frac{t_{n}^{j}}{(\lambda_{n}^{j})^{2}}\biggm)
-
g_{n}^{j_{1}}
\biggm( \sigma_{n}^{j_{1}}
e^{-i\frac{t_{n}^{j_{1}}}{(\lambda_{n}^{j_{1}})^{2}}\Delta}\widetilde{u}^{j_{1}}
\biggm)
\right\|_{H^{1}}=0.
\end{equation} 
This together with \eqref{11/07/07/21:05} gives us that 
\begin{equation}\label{11/07/09/20:22}
\lim_{n\to \infty}\mathcal{I}_{\omega}
\biggm(\sigma_{\infty}^{j_{1}} \widetilde{\psi}^{j_{1}}\bigg(
-\frac{t_{n}^{j}}{(\lambda_{n}^{j})^{2}}\bigg)\biggm)
\le 
\frac{m_{\omega}+m_{\omega}^{*}}{2}
<\widetilde{m}_{\omega}
,
\end{equation} 
which together with the definition of $\widetilde{m}_{\omega}$ shows 
\begin{equation}\label{11/07/09/20:24}
\mathcal{K}
\biggm(\sigma_{\infty}^{j_{1}} \widetilde{\psi}^{j_{1}}\bigg(
-\frac{t_{n}^{j}}{(\lambda_{n}^{j})^{2}}\bigg)\biggm)
>0
\qquad 
\mbox{for any sufficiently large $n \in \mathbb{N}$}. 
\end{equation}
Moreover, \eqref{11/07/07/21:32} together with 
\eqref{11/07/09/20:20} yields that 
\begin{equation}\label{11/07/09/20:27}
\mathcal{S}_{\omega}
\biggm(\sigma_{\infty}^{j_{1}} \widetilde{\psi}^{j_{1}}\bigg(
-\frac{t_{n}^{j}}{(\lambda_{n}^{j})^{2}}\bigg)\biggm)<m_{\omega}
\qquad 
\mbox{for any sufficiently large $n \in \mathbb{N}$}.
\end{equation}
Hence, $\displaystyle{
\sigma_{\infty}^{j_{1}} \widetilde{\psi}^{j_{1}}\bigg(
-\frac{t_{n}^{j}}{(\lambda_{n}^{j})^{2}}\bigg) \in A_{\omega,+}}$ for any sufficiently large $n \in \mathbb{N}$. Then, \eqref{11/07/07/20:37} immediately follows from Lemma \ref{11/06/27/22:29}. 
\par 
On the other hand, suppose $\lambda_{\infty}^{j_{1}}=0$. Then, we see from \eqref{11/07/07/21:32} that 
\begin{equation}\label{11/07/31/22:01}
\frac{m_{\omega}+m_{\omega}^{*}}{2}
> 
\frac{1}{2}\left\|  
\nabla g_{n}^{j_{1}} \bigg( 
\sigma_{n}^{j_{1}}e^{-i\frac{t_{n}^{j_{1}}}{(\lambda_{n}^{j_{1}})^{2}}\Delta}
\widetilde{u}^{j_{1}}
\bigg)
\right\|_{L^{2}}^{2}
-
\frac{1}{2^{*}}
\left\| 
g_{n}^{j_{1}} \bigg( 
\sigma_{n}^{j_{1}}e^{-i\frac{t_{n}^{j_{1}}}{(\lambda_{n}^{j_{1}})^{2}}\Delta}
\widetilde{u}^{j_{1}}
\bigg)
\right\|_{L^{2^{*}}}^{2^{*}},
\end{equation}
which together with Lemma \ref{11/05/01/23:05} and \eqref{11/06/22/22:37} shows 
\begin{equation}\label{11/07/31/22:23}
\frac{1}{d}\sigma^{\frac{d}{2}}>\mathcal{H}_{0}\biggm( 
\sigma_{\infty}^{j}
\widetilde{\psi}^{j_{1}}\biggm( -\frac{t_{n}^{j_{1}}}{(\lambda_{n}^{j_{1}})^{2}}\biggm)\biggm).
\end{equation}
Moreover, \eqref{11/07/31/22:01} together with \eqref{11/07/07/21:36} yields that 
\begin{equation}\label{11/07/31/22:27}
\frac{m_{\omega}+m_{\omega}^{*}}{2}>
\frac{1}{d}\left\|  
\nabla g_{n}^{j_{1}} \bigg( 
\sigma_{n}^{j_{1}}e^{-i\frac{t_{n}^{j_{1}}}{(\lambda_{n}^{j_{1}})^{2}}\Delta}
\widetilde{u}^{j_{1}}
\bigg)
\right\|_{L^{2}}^{2}.
\end{equation}
Hence, we see from \eqref{11/06/22/22:37} that 
\begin{equation}\label{11/07/31/22:39}
\sigma^{\frac{d}{2}}>\left\| \nabla \sigma_{\infty}^{j}\widetilde{\psi}^{j_{1}}
\biggm( -\frac{t_{n}^{j_{1}}}{(\lambda_{n}^{j_{1}})^{2}}\biggm)
\right\|_{L^{2}}^{2}.
\end{equation}
Thus, we have shown \eqref{11/07/07/20:37}. 
\end{proof}

\begin{lemma}\label{11/08/11/21:15}
There exists $j_{0} \in \mathbb{N}$ such that  $I^{j}=\mathbb{R}$ for any $j>j_{0}$, where $I^{j}$ denotes the maximal interval where the nonlinear profile $\widetilde{\psi}^{j}$ exists, and 
\begin{equation}\label{11/08/12/11:04}
\sum_{j>j_{0}}
\left\| \langle \nabla \rangle \widetilde{\psi}^{j} \right\|_{S(\mathbb{R})}^{2}
\lesssim 
\sum_{j>j_{0}}
\left\| \widetilde{u}^{j} \right\|_{H^{1}}^{2}
<\infty .
\end{equation}
\end{lemma}
\begin{proof}[Proof of Lemma \ref{11/08/11/21:15}]
It follows from \eqref{10/09/24/21:40} together with \eqref{11/07/03/14:33} that   
\begin{equation}\label{11/08/10/17:06}
\sum_{j=1}^{\infty}\lim_{n\to \infty}
\left\||\nabla|^{s} g_{n}^{j}
\biggm(
\sigma_{n}^{j}
e^{-i\frac{t_{n}^{j}}{(\lambda_{n}^{j})^{2}}\Delta}
\widetilde{u}^{j}
\biggm) \right\|_{L^{2}}^{2}
\lesssim 
m_{\omega}+\frac{m_{\omega}}{\omega}
\qquad 
\mbox{for $s=0,1$}.
\end{equation}
We see from \eqref{11/08/10/17:06} that  
\begin{equation}\label{11/08/12/11:24}
\sum_{j=1}^{\infty}
\left\|
\widetilde{u}^{j}
\right\|_{H^{1}}^{2}
<\infty
\end{equation}
and therefore     
\begin{equation}\label{11/08/12/11:11}
\lim_{j\to \infty}
\left\|
\widetilde{u}^{j}
\right\|_{H^{1}}
=
0.
\end{equation}
Hence, Lemma \ref{11/08/11/11:56} together with \eqref{11/06/22/22:37} and \eqref{11/08/12/11:11} shows the desired result. 
\end{proof}

Now, for any $j \in \mathbb{N}$, we define the space $W^{j}$ by 
\begin{equation}\label{11/08/01/10:07}
W^{j}:=
\left\{ \begin{array}{ccc}
W_{p_{1}}\cap W &\mbox{if}& \lambda_{\infty}^{j}=1,
\\[6pt]
W &\mbox{if}& \lambda_{\infty}^{j}=0.
\end{array}\right.
\end{equation}
Put 
\begin{equation}\label{11/07/02/14:46}
\psi_{n}^{j}:=
G_{n}^{j} \sigma_{n}^{j} \widetilde{\psi}^{j}.
\end{equation}
The maximal interval where $\psi_{n}^{j}$ exists is 
$\displaystyle{
I_{n}^{j}
:=
\left((\lambda_{n}^{j})^{2}T_{\min}^{j}+t_{n}^{j}, 
(\lambda_{n}^{j})^{2}T_{\max}^{j}+t_{n}^{j}\right)
}$. 

\begin{lemma}\label{11/07/02/15:03}
Let $k \in \mathbb{N}$ and assume that  
\begin{equation}\label{11/08/07/20:19}
\left\| \sigma_{\infty}^{j} \widetilde{\psi}^{j} 
\right\|_{W^{j}(I^{j})}
<\infty 
\qquad 
\mbox{for any $1\le j\le k$},
\end{equation}
where $I^{j}$ denotes the maximal interval where $\widetilde{\psi}^{j}$ exists. Then, we have $I^{j}=\infty$,  
\begin{equation}
\label{11/08/09/21:56}
\left\| \nabla \psi_{n}^{j} \right\|_{W_{1+\frac{4}{d}}(\mathbb{R})}+
\left\| 
\psi_{n}^{j}\right\|_{W_{p_{1}}(\mathbb{R})\cap W(\mathbb{R})}
\lesssim 
\left\| 
\langle \nabla \rangle 
\widetilde{\psi}^{j}\right\|_{S(\mathbb{R})}
\lesssim 1 
\qquad 
\mbox{for any $1\le j \le k$},
\end{equation}
and there exists $B>0$ with the following property: For any $k \in \mathbb{N}$, there exists $N_{k}\in \mathbb{N}$ such that 
\begin{equation}\label{11/08/10/14:07}
\sup_{n\ge N_{k}}
\left( 
\left\|\nabla \psi_{n}^{j} \right\|_{W_{1+\frac{4}{d}}(\mathbb{R})}
+
\left\| \sum_{j=1}^{k} \psi_{n}^{j} 
\right\|_{W_{p_{1}}(\mathbb{R})\cap W(\mathbb{R})}
\right)
\le B.
\end{equation}
Furthermore, if \eqref{11/08/07/20:19} holds for any $k$, then we have 
\begin{equation}\label{11/07/16/14:23}
\lim_{k\to \infty}\lim_{n\to \infty}
\left\| \psi_{n}^{j} |\nabla|^{s} e^{it\Delta} w_{n}^{k} 
\right\|_{L_{t,x}^{\frac{2(d+2)(p-1)}{d(p-1)+4}}}
=0
\quad 
\mbox{for any $j \ge 1$ and $p_{1}\le p \le 2^{*}-1$}.
\end{equation}

\end{lemma}

\begin{proof}[Proof of Lemma \ref{11/07/02/15:03}]
Assume that $\lambda_{\infty}^{j}=0$. Then, $\sigma_{\infty}^{j} \widetilde{\psi}^{j}$ is a solution to \eqref{11/07/17/16:01} (see \eqref{11/07/02/11:49}). Since $I^{j}$ coincides with the maximal interval where $\sigma_{\infty}^{j} \widetilde{\psi}^{j}$ exists, Theorem \ref{11/07/24/20:56} together with Lemma \ref{11/07/07/20:34} shows $I^{j}=\mathbb{R}$. On the other hand, when $\lambda_{\infty}^{j}=1$, $\sigma_{\infty}^{j} \widetilde{\psi}^{j}$ is a solution to \eqref{11/06/12/9:08} and therefore Theorem \ref{10/10/04/21:44} (iv) together with the hypothesis \eqref{11/08/07/20:19} shows $I^{j}=\mathbb{R}$. 
\par 
We shall show that 
\begin{equation}\label{11/08/23/10:57}
\left\|\nabla \psi_{n}^{j} \right\|_{W_{1+\frac{4}{d}}(\mathbb{R})}+
\left\| \psi_{n}^{j} \right\|_{W_{p_{1}}(\mathbb{R})\cap W(\mathbb{R})}
\lesssim 
\left\| \langle \nabla \rangle \widetilde{\psi}^{j} \right\|_{S(\mathbb{R})}. 
\end{equation}
The Mihlin multiplier theorem gives us that 
\begin{equation}\label{11/08/10/15:57}
\begin{split}
\left\| \psi_{n}^{j} \right\|_{W(\mathbb{R})}
&=
\left\| \sigma_{n}^{j} \widetilde{\psi}^{j} \right\|_{W(\mathbb{R})}
=
\left\| \sigma_{n}^{j}(\sigma_{\infty}^{j})^{-1} \sigma_{\infty}^{j}\widetilde{\psi}^{j} \right\|_{W(\mathbb{R})}
\\[6pt]
&\lesssim 
\left\| \sigma_{\infty}^{j}\widetilde{\psi}^{j} 
\right\|_{W(\mathbb{R})}
\lesssim  
\left\| \langle \nabla \rangle \widetilde{\psi}^{j} 
\right\|_{S(\mathbb{R})}
\qquad 
\mbox{for any $j\ge 1$ and $n\ge 1$}.
\end{split}
\end{equation}
Similarly, we have 
\begin{equation}\label{11/08/24/12:20}
\left\|\nabla \psi_{n}^{j} \right\|_{W_{1+\frac{4}{d}}(\mathbb{R})}
\lesssim 
\left\| \langle \nabla \rangle \widetilde{\psi}^{j} 
\right\|_{S(\mathbb{R})}
\qquad 
\mbox{for any $j\ge 1$ and $n\ge 1$}.
\end{equation}
Moreover, using the H\"older inequality, the Mihlin multiplier theorem and \eqref{11/08/10/15:57}, we obtain   
\begin{equation}\label{11/08/09/11:23}
\begin{split}
\left\| \psi_{n}^{j} 
\right\|_{W_{p_{1}}(\mathbb{R})}^{\frac{(d+2)(p_{1}-1)}{2}}
&\le 
\left\|G_{n}^{j} \sigma_{n}^{j}\widetilde{\psi}^{j} 
\right\|_{W_{1+\frac{4}{d}}(\mathbb{R})}^{(d+2)\left\{ 
1-\frac{(d-2)(p_{1}-1)}{4}\right\}}
\left\| G_{n}^{j}\sigma_{n}^{j}\widetilde{\psi}^{j} 
\right\|_{W(\mathbb{R})}^{\frac{(d+2)\left\{d(p_{1}-1)-4\right\}}{4}}
\\[6pt]
&= 
\left\|\lambda_{n}^{j} \sigma_{n}^{j}\widetilde{\psi}^{j} 
\right\|_{W_{1+\frac{4}{d}}(\mathbb{R})}^{(d+2)\left\{ 
1-\frac{(d-2)(p_{1}-1)}{4}\right\}}
\left\| \sigma_{n}^{j}\widetilde{\psi}^{j} 
\right\|_{W(\mathbb{R})}^{\frac{(d+2)\left\{d(p_{1}-1)-4\right\}}{4}}
\\[6pt]
&\lesssim 
\left\|\langle \nabla \rangle \widetilde{\psi}^{j} 
\right\|_{S(\mathbb{R})}^{\frac{(d+2)(p_{1}-1)}{2}}.
\end{split}
\end{equation}

We shall show \eqref{11/08/09/21:56}.  In view of \eqref{11/08/23/10:57}, it suffices to prove that 
\begin{equation}\label{11/08/23/11:01}
\left\| \langle \nabla \rangle \widetilde{\psi}^{j} \right\|_{S(\mathbb{R})}\lesssim 1
\qquad 
\mbox{for any $1\le j \le k$}. 
\end{equation}

When $\lambda_{\infty}^{j}=0$,  it follows from the Strichartz estimate and \eqref{11/05/11/23:53} that  
\begin{equation}\label{11/08/09/9:42}
\begin{split}
\left\| \langle \nabla \rangle \widetilde{\psi}^{j} \right\|_{S(\mathbb{R})}
&\lesssim 
\left\|\widetilde{u}^{j} \right\|_{H^{1}}
+
\left\| \nabla \biggm\{ 
|\sigma_{\infty}^{j}\widetilde{\psi}^{j}|^{2^{*}-2}(\sigma_{\infty}^{j}\widetilde{\psi}^{j}) \biggm\}
\right\|_{L^{2}(\mathbb{R},L^{\frac{2d}{d+2}})}
\\[6pt]
&\lesssim 
\left\|\widetilde{u}^{j} \right\|_{H^{1}}
+
\left\| \nabla \sigma_{\infty}^{j}\widetilde{\psi}^{j}\right\|_{V(\mathbb{R})}
\left\| \sigma_{\infty}^{j}\widetilde{\psi}^{j} \right\|_{W(\mathbb{R})}^{2^{*}-2}
\\[6pt]
&\le 
\left\|\widetilde{u}^{j} \right\|_{H^{1}}
+
\left\|\langle \nabla \rangle \widetilde{\psi}^{j} 
\right\|_{L^{\infty}(\mathbb{R},L^{2})}^{\frac{4}{d+2}}
\left\| \langle \nabla \rangle \widetilde{\psi}^{j}\right\|_{L^{2}(\mathbb{R},L^{2^{*}})}^{\frac{d-2}{d+2}}
\left\| \sigma_{\infty}^{j}\widetilde{\psi}^{j} \right\|_{W(\mathbb{R})}^{2^{*}-2}.
\end{split} 
\end{equation}
Here, \eqref{11/07/31/23:00} together with Lemma \ref{11/07/07/20:34} shows that 
\begin{equation}\label{11/08/09/9:37}
\left\|\nabla \sigma_{\infty}^{j} \widetilde{\psi}^{j} 
\right\|_{L^{\infty}(\mathbb{R},L^{2})}
=
\left\|\langle \nabla \rangle \widetilde{\psi}^{j} 
\right\|_{L^{\infty}(\mathbb{R},L^{2})}
<\infty.
\end{equation}
Hence, \eqref{11/08/09/9:42} together with \eqref{11/08/09/9:37} and the hypothesis \eqref{11/08/07/20:19} shows 
\begin{equation}\label{11/08/12/21:39}
\left\|\langle \nabla \rangle \widetilde{\psi}^{j} \right\|_{S(\mathbb{R})}<\infty
\qquad 
\mbox{for any $1\le j \le k$ with $\lambda_{\infty}^{j}=0$}.
\end{equation}
When $\lambda_{\infty}^{j}=1$, Lemma \ref{11/05/14/14:42} together with Lemma \ref{11/07/07/20:34} and \eqref{11/08/07/20:19} shows 
\begin{equation}\label{11/08/12/21:49}
\left\| \langle \nabla \rangle \widetilde{\psi}^{j} \right\|_{S(\mathbb{R})}<\infty
\qquad 
\mbox{for any $1\le j \le k$ with $\lambda_{\infty}^{j}=1$}.
\end{equation}
Suppose here that $k\le j_{0}$; $j_{0}$ is the number found in Lemma \ref{11/08/11/21:15}. Then, \eqref{11/08/12/21:39} and \eqref{11/08/12/21:49} implies \eqref{11/08/23/11:01}. On the other hand, when $k>j_{0}$, we see from Lemma \ref{11/08/11/21:15} that  
\begin{equation}\label{11/08/12/21:44}
\left\| \langle \nabla \rangle \widetilde{\psi}^{j} \right\|_{S(\mathbb{R})}^{2}
\lesssim 
\sum_{1\le j \le j_{0}} \left\| \langle \nabla \rangle 
\widetilde{\psi}^{j} \right\|_{S(\mathbb{R})}^{2} 
+
\sum_{j>j_{0}}
\left\| \widetilde{u}^{j} \right\|_{H^{1}}^{2}
< \infty.
\end{equation}
Thus, we have shown \eqref{11/08/23/11:01}.
\par 
We shall prove \eqref{11/08/10/14:07}. It is sufficient to consider the case $k\ge j_{0}$. Using the elementary inequality 
\begin{equation}\label{11/08/10/14:47}
\bigg| 
\sum_{1\le j \le k} \psi_{n}^{j}
\bigg|^{q}
-
\sum_{1\le j\le k}
\left| \psi_{n}^{j} \right|^{q}
\le C_{k,q} 
\sum_{1\le j\le k}
\sum_{{1\le j'\le k;}\atop {j'\neq j}}
\left|
\psi_{n}^{j}
\right|^{q-1}
\left|
\psi_{n}^{j'}
\right|,
\quad 
1<q< \infty,
\end{equation}
where $C_{k,q}$ is some constant depends only on $k$ and $q$, we have 
\begin{equation}\label{11/07/02/18:18}
\left\| \sum_{1\le j \le k} \psi_{n}^{j}\right\|_{W(\mathbb{R})}^{\frac{2(d+2)}{d-2}}
\le 
\sum_{1\le j\le k}
\left\| \psi_{n}^{j} 
\right\|_{W(\mathbb{R})}^{\frac{2(d+2)}{d-2}}
+
C_{k}\sum_{1\le j\le k}
\sum_{{1\le j'\le k;}\atop {j'\neq j}}
\int_{\mathbb{R}}\int_{\mathbb{R}^{d}}
\left|
\psi_{n}^{j}
\right|^{\frac{d+6}{d-2}}
\left|
\psi_{n}^{j'}
\right|,
\end{equation}
where $C_{k}>0$ is some constant depending only on $k$ and $d$. 
 We see from Lemma \ref{11/08/11/21:15} and \eqref{11/08/09/21:56} that 
\begin{equation}\label{11/08/11/10:51}
\begin{split}
\sum_{1\le j \le k}
\left\| \psi_{n}^{j} \right\|_{W(\mathbb{R})}^{\frac{2(d+2)}{d-2}}&\lesssim 
\sum_{1\le j \le j_{0}}
\left\| \langle \nabla \rangle \widetilde{\psi}^{j} 
\right\|_{S(\mathbb{R})}^{\frac{2(d+2)}{d-2}}
+
\sum_{j > j_{0}}
\left\| \langle \nabla \rangle \widetilde{\psi}^{j} 
\right\|_{S(\mathbb{R})}^{\frac{2(d+2)}{d-2}}
\\[6pt]
&\lesssim 
\sum_{1\le j \le j_{0}}
\left\| \langle \nabla \rangle \widetilde{\psi}^{j} 
\right\|_{S(\mathbb{R})}^{\frac{2(d+2)}{d-2}}
+
\sum_{j>j_{0}}
\left\| \widetilde{u}^{j}
\right\|_{H^{1}}^{2}
< \infty
.
\end{split}
\end{equation}
Moreover, we see from \eqref{11/07/03/10:21} that 
there exists $N_{j,k}\in \mathbb{N}$ such that 
\begin{equation}\label{11/07/03/11:24}
\int_{\mathbb{R}}\int_{\mathbb{R}^{d}}
\left|
\psi_{n}^{j}
\right|^{\frac{d+6}{d-2}}
\left|
\psi_{n}^{j'}
\right|
\le \frac{1}{C_{k}k^{2}}
\qquad 
\mbox{for any $n\ge N_{j,k}$}. 
\end{equation}
Combining \eqref{11/07/02/18:18} with \eqref{11/08/11/10:51} and \eqref{11/07/03/11:24}, we can take $B_{0}>0$ with the property that for any $k \in \mathbb{N}$, there exists $N_{k} \in \mathbb{N}$ such that 
\begin{equation}\label{11/07/02/22:38}
\sup_{n\ge N_{k}}\left\| \sum_{j=1}^{k}\psi_{n}^{j} \right\|_{W(\mathbb{R})}
\le  B_{0}. 
\end{equation}
Next, we consider the estimate in $W_{p_{1}}(\mathbb{R})$. Using the elementary inequality \eqref{11/08/10/14:47} again, we have 
\begin{equation}\label{11/08/07/20:59}
\begin{split}
\left\| \sum_{1\le j \le k} \psi_{n}^{j}\right\|_{W_{p_{1}}(\mathbb{R})}^{\frac{(d+2)(p_{1}-1)}{2}}
&\le 
\sum_{1\le j\le k}
\left\| \psi_{n}^{j} 
\right\|_{W_{p_{1}}(\mathbb{R})}^{\frac{(d+2)(p_{1}-1)}{2}}
\\[6pt]
&\qquad +C_{k}'
\sum_{1\le j\le k}
\sum_{{1\le j'\le k;}\atop {j'\neq j}}
\int_{\mathbb{R}}\int_{\mathbb{R}^{d}}
\left|
\psi_{n}^{j}
\right|^{\frac{(d+2)(p_{1}-1)-2}{2}}
\left|
\psi_{n}^{j'}
\right|,
\end{split}
\end{equation}
where $C_{k}'$ is some constant depending only on $k$, $d$ and $p_{1}$.  We see from \eqref{11/08/09/21:56} and Lemma \ref{11/08/11/21:15} that 
\begin{equation}\label{11/08/12/11:53}
\sum_{1\le j\le k}
\left\| \psi_{n}^{j} 
\right\|_{W_{p_{1}}(\mathbb{R})}^{\frac{(d+2)(p_{1}-1)}{2}}
\lesssim 
\sum_{1\le j\le j_{0}}
\left\| \langle \nabla \rangle \widetilde{\psi}^{j} 
\right\|_{S(\mathbb{R})}^{\frac{(d+2)(p_{1}-1)}{2}}
+
\sum_{j> j_{0}}
\left\| \widetilde{u}^{j} 
\right\|_{H^{1}}^{2}< \infty
.
\end{equation}
Next, we consider the second term on the right-hand side of \eqref{11/08/07/20:59}. Using the condition \eqref{11/07/03/10:21}, we can take $N_{j,k}'\in \mathbb{N}$ such that
\begin{equation}
\label{11/08/07/21:02}
\int_{\mathbb{R}}\int_{\mathbb{R}^{d}}
\left|
\psi_{n}^{j}
\right|^{\frac{(d+2)(p_{1}-1)-2}{2}}
\left|
\psi_{n}^{j'}
\right|
\le \frac{1}{C_{k}'k^{2}}
\qquad 
\mbox{for any $n\ge N_{j,k}'$}. 
\end{equation}
Combining \eqref{11/08/07/20:59} with \eqref{11/08/12/11:53} and \eqref{11/08/07/21:02}, we can take $B_{1}$ with the property that  for any $k \in \mathbb{N}$, there exists $N_{k} \in \mathbb{N}$ such that 
\begin{equation}\label{11/08/12/22:57}
\sup_{n\ge N_{k}}\left\| \sum_{j=1}^{k}\psi_{n}^{j} 
\right\|_{W_{p_{1}}(\mathbb{R})}
\le B_{1}. 
\end{equation}
Similarly, we have 
\begin{equation}\label{11/08/24/12:26}
\sup_{n\ge N_{k}}\left\| \sum_{j=1}^{k} \nabla \psi_{n}^{j} 
\right\|_{W_{1+\frac{4}{d}}(\mathbb{R})}
\le B_{1}. 
\end{equation}
Thus, we have proved \eqref{11/08/10/14:07}.
\par 
Finally, we shall prove \eqref{11/07/16/14:23}. For each $1\le j \le k$, let $\{v_{m}^{j} \}$ be a sequence in $C_{c}(\mathbb{R}^{d} \times \mathbb{R})$ such that 
\begin{equation}\label{11/07/16/17:42}
\lim_{m\to \infty}
\left\| \langle \nabla \rangle \left( 
\widetilde{\psi}^{j}-v_{m}^{j} 
\right)
\right\|_{V_{p}(\mathbb{R})}=0.
\end{equation}
Then, using the H\"older inequality, the Strichartz estimate and \eqref{11/07/17/10:57}, we have   
\begin{equation}\label{11/07/16/17:41}
\begin{split}
&\left\| \psi_{n}^{j} |\nabla |^{s} e^{it\Delta} w_{n}^{k} 
\right\|_{L_{t,x}^{\frac{2(d+2)(p-1)}{d(p-1)+4}}}
\\[6pt]
&=
(\lambda_{n}^{j})^{\frac{d(p-1)+4}{2(p-1)}-(d-2)}
\left\| \left( \sigma_{n}^{j}\widetilde{\psi}^{j} \right) 
(G_{n}^{j})^{-1} |\nabla |^{s} e^{it\Delta}  w_{n}^{k}
\right\|_{L_{t,x}^{\frac{2(d+2)(p-1)}{d(p-1)+4}}}
\\[6pt]
&\le 
(\lambda_{n}^{j})^{\frac{d(p-1)+4}{2(p-1)}-(d-2)}
\left\|  \sigma_{n}^{j} \left( 
\widetilde{\psi}^{j}- v_{m}^{j} 
\right) 
\right\|_{W_{p}(\mathbb{R})}
\left\| 
(G_{n}^{j})^{-1}|\nabla |^{s} e^{it\Delta}  w_{n}^{k} 
\right\|_{L_{t,x}^{\frac{2(d+2)}{d}}}
\\[6pt]
&\qquad 
+
(\lambda_{n}^{j})^{\frac{d(p-1)+4}{2(p-1)}-(d-2)}
\left\|\left( \sigma_{n}^{j} v_{m}^{j} \right) 
(G_{n}^{j})^{-1}|\nabla |^{s}e^{it\Delta}  w_{n}^{k} 
\right\|_{L_{t,x}^{\frac{2(d+2)(p-1)}{d(p-1)+4}}}.
\end{split}
\end{equation}
Using the Strichartz estimate, the Mihlin multiplier theorem and \eqref{11/07/16/17:42}, we estimate the first term on the right-hand side of \eqref{11/07/16/17:41} as follows:  
\begin{equation}\label{11/08/21/17:37}
\begin{split}
&(\lambda_{n}^{j})^{\frac{d(p-1)+4}{2(p-1)}-(d-2)}
\left\|  \sigma_{n}^{j} \left( 
\widetilde{\psi}^{j}- v_{m}^{j} 
\right) 
\right\|_{W_{p}(\mathbb{R})}
\left\| 
(G_{n}^{j})^{-1}|\nabla |^{s}e^{it\Delta} w_{n}^{k} 
\right\|_{L_{t,x}^{\frac{2(d+2)}{d}}}
\\[6pt]
&=
(\lambda_{n}^{j})^{\frac{4-(d-2)(p-1)}{2(p-1)}}
\left\|  \sigma_{n}^{j} \left( 
\widetilde{\psi}^{j}- v_{m}^{j} 
\right) 
\right\|_{W_{p}(\mathbb{R})}
\left\| |\nabla|^{s}
e^{it\Delta}  w_{n}^{k} 
\right\|_{L_{t,x}^{\frac{2(d+2)}{d}}}
\\[6pt]
&\lesssim 
(\lambda_{n}^{j})^{\frac{4-(d-2)(p-1)}{2(p-1)}}
\left\|  \langle \nabla \rangle 
\left( 
\widetilde{\psi}^{j}- v_{m}^{j}
\right)  
\right\|_{V_{p}(\mathbb{R})}
\left\| w_{n}^{k} \right\|_{H^{1}}
=o_{m}(1). 
\end{split}
\end{equation}
We consider the second term on the right-hand side of \eqref{11/07/16/17:41}. Note here that $\frac{2(d+2)(p-1)}{d(p-1)+4}\le \frac{d+2}{d-1}\le 2$ for $d\ge 4$. When $s=0$, we have by the H\"older inequality and \eqref{11/07/06/17:55} that 
\begin{equation}\label{11/08/23/12:18}
\begin{split}
&
(\lambda_{n}^{j})^{\frac{d(p-1)+4}{2(p-1)}-(d-2)}
\left\|\left( \sigma_{n}^{j} v_{m}^{j} \right) 
(G_{n}^{j})^{-1} |\nabla |^{s} e^{it\Delta}  w_{n}^{k} 
\right\|_{L_{t,x}^{\frac{2(d+2)(p-1)}{d(p-1)+4}}}
\\[6pt]
&\le 
\lambda_{n}^{j}
\left\| \sigma_{n}^{j}v_{m}^{j} \right\|_{L_{t,x}^{\frac{2(d+2)(p-1)}{2(p-1)+4}}}
\left\| 
e^{it\Delta} w_{n}^{k} 
\right\|_{W(\mathbb{R})}
\to 0 \qquad 
\mbox{as $n\to \infty$ and $k\to \infty$}.
\end{split}
\end{equation}
Next, we consider the case where $s=1$. Let $K_{m}^{j}:=\{|x|\le R_{m}^{j}\}\times [-T_{m}^{j}, T_{m}^{j}]$  be the rectangle containing the support of $v_{m}^{j}$. Using the H\"older inequality and employing Lemma 2.5 in \cite{Killip-Visan} (see also \cite{Keraani}), we have
\footnote{
When $d=3$ and $2< \frac{2(d+2)(p-1)}{d(p-1)+4}$ (hence $p>3$), we estimate as follows:   
\begin{equation}\label{11/07/16/17:38}
\begin{split}
&\left\|\left( \sigma_{n}^{j} v_{m}^{j} \right) 
 (G_{n}^{j})^{-1} \nabla e^{it\Delta}  w_{n}^{k} 
\right\|_{L_{t,x}^{\frac{2(d+2)(p-1)}{d(p-1)+4}}}
\\[6pt]
&\le \lambda_{n}^{j}
\left\| \sigma_{n}^{j}v_{m}^{j} \right\|_{L_{t,x}^{\infty}}
\left\| 
\nabla e^{it\Delta}(g_{n}^{j})^{-1} w_{n}^{k} 
\right\|_{L_{t,x}^{2}(K_{m}^{j})}^{\frac{2}{p-1}}
\left\| 
\nabla e^{it\Delta}  (g_{n}^{j})^{-1} w_{n}^{k} 
\right\|_{L_{t,x}^{\frac{2(d+2)}{d}}}^{\frac{p-3}{p-1}}.
\end{split}
\end{equation}
}
\begin{equation}\label{11/07/16/18:18}
\begin{split}
&
(\lambda_{n}^{j})^{\frac{d(p-1)+4}{2(p-1)}-(d-2)}
\left\|\left( \sigma_{n}^{j} v_{m}^{j} \right) 
(G_{n}^{j})^{-1} \nabla  e^{it\Delta}  w_{n}^{k} 
\right\|_{L_{t,x}^{\frac{2(d+2)(p-1)}{d(p-1)+4}}}
\\[6pt]
&\le 
(\lambda_{n}^{j})^{\frac{d(p-1)+4}{2(p-1)}-(d-2)+1}
\left\| \sigma_{n}^{j} v_{m}^{j} 
\right\|_{L_{t,x}^{\frac{2(d+2)(p-1)}{4-2(p-1)}}}
\left\| 
\nabla  
e^{it\Delta} (g_{n}^{j})^{-1}  w_{n}^{k} 
\right\|_{L_{t,x}^{2}(K_{m}^{j})}
\\[6pt]
&\lesssim 
(\lambda_{n}^{j})^{\frac{d(p-1)+4}{2(p-1)}-(d-2)+1}
\left\| \sigma_{n}^{j} v_{m}^{j} 
\right\|_{L_{t,x}^{\frac{2(d+2)(p-1)}{4-2(p-1)}}}
(T_{m}^{j})^{\frac{2}{d+2}}(R_{m}^{j})^{\frac{3d+2}{2(d+2)}}
\left\| e^{it\Delta}w_{n}^{k} \right\|_{W(\mathbb{R})}^{\frac{1}{3}}
\left\| \nabla w_{n}^{k} \right\|_{L^{2}}^{\frac{2}{3}}
.
\end{split}
\end{equation}
Hence,  we see from \eqref{11/07/06/17:55} that  
\begin{equation}\label{11/08/23/11:50}
\lim_{k\to \infty}\lim_{n\to \infty}
(\lambda_{n}^{j})^{\frac{d(p-1)+4}{2(p-1)}-(d-2)}
\left\|\left( \sigma_{n}^{j} v_{m}^{j} \right) 
(G_{n}^{j})^{-1}\nabla e^{it\Delta}  w_{n}^{k} 
\right\|_{L_{t,x}^{\frac{2(d+2)(p-1)}{d(p-1)+4}}(K_{m}^{j})}
=0. 
\end{equation}
Combining \eqref{11/07/16/17:41} with \eqref{11/08/21/17:37}, \eqref{11/08/23/12:18} and \eqref{11/08/23/11:50}, we obtain \eqref{11/07/16/14:23}. 
\end{proof}

\begin{lemma}[cf. Lemma 5.6 in \cite{IMN}]\label{11/07/03/14:31}
Let $j_{0}$ be the number found in Lemma \ref{11/08/11/21:15}. Then, there exists $1\le j \le j_{0}$ such that  
\begin{equation}\label{11/07/03/15:40}
\left\| \sigma_{\infty}^{j}\widetilde{\psi}^{j} 
\right\|_{W^{j}(I^{j})}=\infty,
\end{equation}
where $I^{j}$ denotes the maximal interval where the nonlinear profile $\widetilde{\psi}^{j}$ exists. 
\end{lemma}

\begin{proof}[Proof of Lemma \ref{11/07/03/14:31}]
We see from Lemma \ref{11/08/11/21:15} that 
\begin{equation}\label{11/08/19/11:12}
\left\| \sigma_{\infty}^{j} \widetilde{\psi}^{j} 
\right\|_{W^{j}(I^{j})}
\lesssim 
\left\|\langle \nabla \rangle \widetilde{\psi}^{j} 
\right\|_{S(\mathbb{R})}
<\infty 
\quad 
\mbox{for any $j>j_{0}$}. 
\end{equation} 

Suppose the contrary that  
\begin{equation}\label{11/08/19/11:16}
\left\| \sigma_{\infty}^{j} \widetilde{\psi}^{j} \right\|_{W^{j}(I^{j})}
< \infty
\qquad 
\mbox{for any $1\le j\le j_{0}$},
\end{equation}
so that 
\begin{equation}\label{11/07/02/15:04}
\left\| \sigma_{\infty}^{j} \widetilde{\psi}^{j} \right\|_{W^{j}(I^{j})}
< \infty
\qquad 
\mbox{for any $j\ge 1$}.
\end{equation}
Let  $k$ be a sufficiently large number to be specified later. Then, we define the function $\psi_{n}^{k\mbox{-}app}$ by
\begin{equation}\label{11/07/04/21:18}
\psi_{n}^{k\mbox{-}app}(x,t):=\sum_{j=1}^{k}
\psi_{n}^{j}(x,t)+e^{it\Delta}w_{n}^{k}(x).
\end{equation}

It follows from Lemma \ref{11/07/02/15:03} together with the 
 hypothesis \eqref{11/07/02/15:04} that $\psi_{n}^{j}$ exists globally in time for any $j\ge 1$ and hence so does $\psi_{n}^{k\mbox{-}app}$. 
\par 
We see from Lemma \ref{11/07/02/15:03} and \eqref{11/07/17/10:57} that there exists $B>0$ with the following property: there exists $N_{1,k} \in \mathbb{N}$ such that  
\begin{equation}
\label{11/07/03/16:49}
\sup_{n\ge N_{1,k}}\left\| \psi_{n}^{k\mbox{-}app} \right\|_{W_{p_{1}}(\mathbb{R})\cap W(\mathbb{R})}
\le B.
\end{equation}
Moreover, it follows from \eqref{11/06/22/22:27} that
\begin{equation}\label{11/08/18/17:22}
\begin{split}
\left\|
\psi_{n}(0)-\psi_{n}^{k\mbox{-}app}(0)
\right\|_{H^{1}}
&=
\left\|
\psi_{n}(0)-\sum_{j=1}^{k}\psi_{n}^{j}(0)
-w_{n}^{k} 
\right\|_{H^{1}}
\\[6pt]
&\lesssim 
\sum_{j=1}^{k}
\left\|
e^{-i\frac{t_{n}^{j}}{(\lambda_{n}^{j})^{2}}\Delta}\widetilde{u}^{j}
-
\widetilde{\psi}^{j}
\biggm( 
-\frac{t_{n}^{j}}{(\lambda_{n}^{j})^{2}}\biggm)
\right\|_{H^{1}}
.
\end{split}
\end{equation}
Hence, \eqref{11/06/22/22:37} shows that for any $k \in \mathbb{N}$, there exists $N_{2,k} \in \mathbb{N}$ such that 
\begin{equation}\label{11/08/18/17:31}
\sup_{n\ge N_{2,k}}\left\|
\psi_{n}(0)-\psi_{n}^{k\mbox{-}app}(0)
\right\|_{H^{1}}
\le 1.
\end{equation}
Now, let $\delta$ be the constant found in Proposition \ref{11/05/14/12:26} which is determined by the bound \eqref{11/07/03/14:33}  and $B$. 
\par 
We see from the estimate \eqref{11/08/18/17:22} together with the Strichartz estimate and \eqref{11/06/22/22:37} that for any $k \in \mathbb{N}$, there exists $N_{3,k} \in \mathbb{N}$ such that  
\begin{equation}\label{11/07/03/16:52}
\sup_{n\ge N_{3,k}}\left\|\langle \nabla \rangle e^{it\Delta}
\left( \psi_{n}(0)-\psi_{n}^{k\mbox{-}app}(0)
\right) \right\|_{V_{p_{1}}}<\delta
.
\end{equation}

We shall show that there exist $k_{0} \in \mathbb{N}$ and $N_{0} \in \mathbb{N}$ such that 
\begin{equation}\label{11/07/02/17:21}
\begin{split}
&\left\|\langle \nabla \rangle  
\left(
i\frac{\partial \psi_{n}^{k_{0} \mbox{-}app}}{\partial t}+\Delta \psi_{n}^{k_{0} \mbox{-}app} +f(\psi_{n}^{k_{0}\mbox{-}app})
+
|\psi_{n}^{k_{0}\mbox{-}app}|^{2^{*}-2}\psi_{n}^{k_{0}\mbox{-}app}\right)
\right\|_{L_{t,x}^{\frac{2(d+2)}{d+4}}}
\\[6pt]
&<\delta
\qquad 
\mbox{for any $n\ge N_{0}$}.
\end{split}
\end{equation}
Before proving this, we remark that \eqref{11/07/02/17:21} together with the long-time perturbation theory leads to an absurd conclusion; Indeed, we see from  \eqref{11/07/03/16:49}, \eqref{11/08/18/17:31} and \eqref{11/07/03/16:52} that
\begin{align}
\label{11/08/18/18:16}
&\left\| \psi_{n}(0)-\psi_{n}^{k_{0}\mbox{-}app} \right\|_{H^{1}}
\le 1,
\\[6pt]
\label{11/08/18/18:22}
&\left\| \psi_{n}^{k_{0}\mbox{-}app} \right\|_{W_{p_{1}}(\mathbb{R})\cap W(\mathbb{R})}
\le B,
\\[6pt]
\label{11/08/18/18:30}
&\left\|\langle \nabla \rangle e^{it\Delta}
\left( \psi_{n}(0)-\psi_{n}^{k_{0}\mbox{-}app}(0)
\right) \right\|_{V_{p_{1}}(\mathbb{R})}<\delta
\end{align}
for any  $n\ge \max\{N_{1,k_{0}},N_{2,k_{0}}, N_{3,k_{0}}, N_{0}\}$. Hence, employing Proposition \ref{11/05/14/12:26} (Long-time perturbation theory), we conclude that 
\begin{equation}
\label{11/07/03/14:28}
\left\|\psi_{n} \right\|_{W_{p}(\mathbb{R})\cap W(\mathbb{R})}
< \infty
\qquad 
\mbox{for any $n\ge \max\{N_{1,k_{0}},N_{2,k_{0}}, N_{3,k_{0}}, N_{0}\}$},
\end{equation}
which contradicts \eqref{11/06/29/16:04}: hence, Lemma \ref{11/07/03/14:31} holds.  
\par 
It remains to prove \eqref{11/07/02/17:21}. Note that 
\begin{equation}\label{11/07/15/9:59}
\begin{split}
&i\frac{\partial \psi_{n}^{k\mbox{-}app}}{\partial t}
+
\Delta \psi_{n}^{k\mbox{-}app}
+
f(\psi_{n}^{k\mbox{-}app})
+
|\psi_{n}^{k\mbox{-}app}|^{2^{*}-2}\psi_{n}^{k\mbox{-}app}
\\[6pt]
& =
\mathcal{N}\big(\psi_{n}^{k\mbox{-}app}\big)
-
\mathcal{N}\big(\psi_{n}^{k\mbox{-}app}-e^{it\Delta}w_{n}^{k}\big)+
\sum_{j=1}^{k} i\frac{\partial \psi_{n}^{j}}{\partial t}
+
\Delta \psi_{n}^{j}
+
\mathcal{N}\biggm(\sum_{j=1}^{k}\psi_{n}^{j}\biggm), 
\end{split}
\end{equation}
where 
\begin{align}
\label{11/08/19/10:32}
\mathcal{N}(u)&:=f(u)+|u|^{2^{*}-2}u, 
\\[6pt]
\label{11/07/15/9:49}
\mathcal{N}_{j}(u)&:=
\left\{ \begin{array}{ccc}
f(u)
+
|u|^{2^{*}-2}u
 &\mbox{if}& \lambda_{\infty}^{j}=1,
\\[6pt]
|u|^{2^{*}-2}u &\mbox{if}& \lambda_{\infty}^{j}=0.
\end{array} \right.
\end{align}
Hence, it suffices for \eqref{11/07/02/17:21} to show that   
\begin{align}
\label{11/07/16/9:46}
&\lim_{k\to \infty}\lim_{n\to \infty}
\left\| \langle \nabla \rangle \left\{
\mathcal{N}( \psi_{n}^{k\mbox{-}app})
-
\mathcal{N} \left( 
\psi_{n}^{k\mbox{-}app}-e^{it\Delta}w_{n}^{k} 
\right)
\right\}\right\|_{L_{t,x}^{\frac{2(d+2)}{d+4}}}
=0,
\\[6pt]
\label{11/07/16/9:50}
&
\lim_{k\to \infty}\lim_{n\to \infty}
\left\| \langle \nabla \rangle \left\{
\sum_{j=1}^{k}
i\frac{\partial \psi_{n}^{j}}{\partial t}
+
\Delta \psi_{n}^{j}
+
\mathcal{N}\biggm( 
\sum_{j=1}^{k} \psi_{n}^{j}
\biggm)
\right\}\right\|_{L_{t,x}^{\frac{2(d+2)}{d+4}}}
\hspace{-9pt}=0.
\end{align}
First, we prove \eqref{11/07/16/9:46}. Using the growth conditions  \eqref{11/04/30/22:54} and \eqref{11/05/02/8:40}, we verify that  
\begin{equation}\label{11/07/16/10:08}
\begin{split}
&
\left\| \langle \nabla \rangle \left\{
\mathcal{N}( \psi_{n}^{k\mbox{-}app})
-
\mathcal{N} \left( 
\psi_{n}^{k\mbox{-}app}-e^{it\Delta}w_{n}^{k} 
\right)
\right\}\right\|_{L_{t,x}^{\frac{2(d+2)}{d+4}}}
\\[6pt]
&\lesssim  
\sum_{j=1}^{2}
\left\| 
\big| \psi_{n}^{k\mbox{-}app} \big|^{p_{j}-1}
e^{it\Delta}w_{n}^{k}
\right\|_{L_{t,x}^{\frac{2(d+2)}{d+4}}}
+
\left\| 
\big| \psi_{n}^{k\mbox{-}app} \big|^{\frac{4}{d-2}}
e^{it\Delta}w_{n}^{k}
\right\|_{L_{t,x}^{\frac{2(d+2)}{d+4}}}
\\[6pt]
&\qquad 
+
\sum_{j=1}^{2}
\left\| 
e^{it\Delta}w_{n}^{k}
\right\|_{L_{t,x}^{\frac{2(d+2)p_{j}}{d+4}}}^{p_{j}}
+
\left\| 
e^{it\Delta}w_{n}^{k}
\right\|_{L_{t,x}^{\frac{2(d+2)(2^{*}-1)}{d+4}}}^{2^{*}-1}
\\[6pt]
&\qquad +
\sum_{j=1}^{2}\left\| 
\big| \psi_{n}^{k\mbox{-}app} \big|^{p_{j}-1}
\nabla e^{it\Delta}w_{n}^{k}
\right\|_{L_{t,x}^{\frac{2(d+2)}{d+4}}}
+
\left\| 
\big| \psi_{n}^{k\mbox{-}app} \big|^{2^{*}-2}
\nabla e^{it\Delta}w_{n}^{k}
\right\|_{L_{t,x}^{\frac{2(d+2)}{d+4}}}
\\[6pt]
&\qquad 
+
\sum_{j=1}^{2}
\left\| 
\big| \big| e^{it\Delta}w_{n}^{k} \big|^{p_{j}-1}
\nabla \psi_{n}^{k\mbox{-}app}
\right\|_{L_{t,x}^{\frac{2(d+2)}{d+4}}}
+
\left\| 
\big| \big| e^{it\Delta}w_{n}^{k} \big|^{2^{*}-2}
\nabla \psi_{n}^{k\mbox{-}app}
\right\|_{L_{t,x}^{\frac{2(d+2)}{d+4}}}
\\[6pt]
&\qquad 
+
\sum_{j=1}^{2}
\left\| 
\big| \big| e^{it\Delta}w_{n}^{k} \big|^{p_{j}-1}
\nabla e^{it\Delta}w_{n}^{k} 
\right\|_{L_{t,x}^{\frac{2(d+2)}{d+4}}}
+
\left\| 
\big| \big| e^{it\Delta}w_{n}^{k} \big|^{2^{*}-2}
\nabla e^{it\Delta}w_{n}^{k} 
\right\|_{L_{t,x}^{\frac{2(d+2)}{d+4}}},
\end{split}
\end{equation}
where we must add the terms 
\begin{align}
\label{11/08/21/11:18}
&
\left\| 
\big| \psi_{n}^{k\mbox{-}app} \big|^{p_{2}-2}
e^{it\Delta}w_{n}^{k} 
\nabla  \psi_{n}^{k\mbox{-}app}
\right\|_{L_{t,x}^{\frac{2(d+2)}{d+4}}}
\qquad \mbox{if $p_{2}>2$},
\\[6pt]
\label{11/08/21/11:19}
& 
\left\| 
\big| \psi_{n}^{k\mbox{-}app} \big|^{\frac{6-d}{d-2}}
e^{it\Delta}w_{n}^{k} 
\nabla  \psi_{n}^{k\mbox{-}app}
\right\|_{L_{t,x}^{\frac{2(d+2)}{d+4}}}
\qquad \mbox{if $d\le 5$}  
\end{align}
to the right-hand side of  \eqref{11/07/16/10:08}.
\par 
Since $\frac{2(d+2)}{d}<\frac{2(d+2)p}{d+4}\le \frac{2(d+2)}{d-2}$ for $p_{1}\le p \le 2^{*}-1$, we easily see from the H\"older inequality, \eqref{11/07/06/17:55} and \eqref{11/07/24/15:19} that the 3rd, 4th, the final and the 2nd final terms on the right-hand side of \eqref{11/07/16/10:08} vanish as $k\to \infty$ and $n\to \infty$.  Using Lemma \ref{11/07/02/15:03}, we also estimate the 7th and 8th terms as follows:  
\begin{equation}\label{11/08/24/12:09}
\begin{split}
&\lim_{k\to \infty}\lim_{n\to \infty}
\left\| 
\big| \big| e^{it\Delta}w_{n}^{k} \big|^{p-1}
\nabla \psi_{n}^{k\mbox{-}app}
\right\|_{L_{t,x}^{\frac{2(d+2)}{d+4}}}
\\[6pt]
&\le 
\lim_{k\to \infty}\lim_{n\to \infty}
\left\| 
e^{it\Delta}w_{n}^{k} \right\|_{W_{p}(\mathbb{R})}^{p-1}
\left\| \nabla \psi_{n}^{k\mbox{-}app}
\right\|_{W_{1+\frac{4}{d}}(\mathbb{R})}
=0.
\end{split}
\end{equation}
\par 
We consider the terms of the form 
\begin{equation}\label{11/08/21/11:35}
\left\| 
\big| \psi_{n}^{k\mbox{-}app} \big|^{p-1}
|\nabla |^{s} e^{it\Delta}w_{n}^{k}
\right\|_{L_{t,x}^{\frac{2(d+2)}{d+4}}},
\qquad 
\mbox{$1+\frac{4}{d}<p \le 2^{*}-1$ and $s=0,1$},
\end{equation}
which corresponds to the 1st, 2nd, 5th, 6th terms on the right-hand side of \eqref{11/07/16/10:08}.
Using the H\"older inequality, \eqref{11/08/10/14:07} and \eqref{11/07/17/10:57}, we have 
\begin{equation}\label{11/08/21/11:39}
\begin{split}
&\left\| 
\big| \psi_{n}^{k\mbox{-}app} \big|^{p-1}
|\nabla |^{s} e^{it\Delta}w_{n}^{k}
\right\|_{L_{t,x}^{\frac{2(d+2)}{d+4}}}
\\[6pt]
&\le 
\left\| \psi_{n}^{k\mbox{-}app} \right\|_{W_{p}(\mathbb{R})}^{\frac{3(p-1)}{4}}
\left\| 
|\nabla|^{s} e^{it\Delta} w_{n}^{k}
\right\|_{L_{t,x}^{\frac{2(d+2)}{d}}}^{\frac{5-p}{4}}
\left\|  \psi_{n}^{k\mbox{-}app} 
|\nabla|^{s} e^{it\Delta} w_{n}^{k} 
\right\|_{L_{t,x}^{\frac{2(d+2)(p-1)}{d(p-1)+4}}}^{\frac{p-1}{4}}
\\[6pt]
&\lesssim 
\left\|  \bigg( \sum_{j=1}^{k}\psi_{n}^{j}\bigg) 
|\nabla|^{s} e^{it\Delta} w_{n}^{k} 
\right\|_{L_{t,x}^{\frac{2(d+2)(p-1)}{d(p-1)+4}}}^{\frac{p-1}{4}}
+ 
\left\| 
e^{it\Delta} w_{n}^{k}
|\nabla|^{s} e^{it\Delta} w_{n}^{k} 
\right\|_{L_{t,x}^{\frac{2(d+2)(p-1)}{d(p-1)+4}}}^{\frac{p-1}{4}}
\\[6pt]
&\hspace{240pt}    
\mbox{for any sufficiently large $n$}. 
\end{split}
\end{equation}
We consider the first term on the right-hand side of \eqref{11/08/21/11:39}.  It follows from Lemmata \ref{11/08/11/21:15} and 
\ref{11/07/02/15:03} that for any $\eta>0$, there exists $J(\eta)\in \mathbb{N}$ such that  
\begin{equation}\label{11/08/24/10:36}
\left( \sum_{j>J(\eta)}\left\| \psi_{n}^{j} \right\|_{W_{p_{1}}(\mathbb{R})\cap W(\mathbb{R})}^{2}\right)^{\frac{1}{2}}< \eta.
\end{equation}
Using the triangle inequality and the H\"older inequality, we have\begin{equation}\label{11/07/16/13:48}
\begin{split}
&\left\| \bigg( \sum_{j=1}^{k}\psi_{n}^{j}\bigg) |\nabla |^{s}
e^{it\Delta} w_{n}^{k} \right\|_{L_{t,x}^{\frac{2(d+2)(p-1)}{d(p-1)+4}}}
\\[6pt]
&\le 
\sum_{j\le J(\eta)}
\left\| \psi_{n}^{j} |\nabla |^{s}
e^{it\Delta} w_{n}^{k} \right\|_{L_{t,x}^{\frac{2(d+2)(p-1)}{d(p-1)+4}}}
+
\left\| \sum_{j>J(\eta)}
 \psi_{n}^{j} \right\|_{W_{p}(\mathbb{R})}
\left\| 
\nabla 
e^{it\Delta} w_{n}^{k} \right\|_{W_{1+\frac{4}{d}}(\mathbb{R})}
.
\end{split}
\end{equation}
Lemma \ref{11/07/02/15:03} shows the first term on the right-hand side of \eqref{11/07/16/13:48} vanishes when $k$ tends to $\infty$ and then $n$ tends to $\infty$. Moreover, we see from the elementary inequality \eqref{11/08/10/14:47}, \eqref{11/07/03/10:21}, \eqref{11/07/17/10:57} and \eqref{11/08/24/10:36} that 
\begin{equation}\label{11/07/03/10:21}
\lim_{n\to \infty}\left\| \sum_{j>J(\eta)}
 \psi_{n}^{j} \right\|_{W_{p}(\mathbb{R})}
\hspace{-9pt}
\left\| 
\nabla 
e^{it\Delta} w_{n}^{J} \right\|_{W_{1+\frac{4}{d}}(\mathbb{R})}
\lesssim 
\left( 
\sum_{j>J(\eta)} \left\|
 \psi_{n}^{j} \right\|_{W_{p}(\mathbb{R})}
\right)^{\frac{1}{2}}
\left\| 
 w_{n}^{k} \right\|_{H^{1}}
\lesssim \eta. 
\end{equation}
Thus, we find that 
\begin{equation}\label{11/08/24/11:28}
\lim_{k\to \infty}\lim_{n\to \infty}
\left\|  \bigg( \sum_{j=1}^{k}\psi_{n}^{j}\bigg) 
|\nabla|^{s} e^{it\Delta} w_{n}^{k} 
\right\|_{L_{t,x}^{\frac{2(d+2)(p-1)}{d(p-1)+4}}}^{\frac{p-1}{4}}
=0.
\end{equation}
On the other hand, it follows form the H\"older inequality and 
 \eqref{11/07/06/17:55} that  
\begin{equation}\label{11/08/24/11:34}
\begin{split}
&\left\| 
e^{it\Delta} w_{n}^{k}
|\nabla|^{s} e^{it\Delta} w_{n}^{k} 
\right\|_{L_{t,x}^{\frac{2(d+2)(p-1)}{d(p-1)+4}}}
\\[6pt]
&\le 
\left\| 
e^{it\Delta} w_{n}^{k}
\right\|_{W_{1+\frac{4}{d}}(\mathbb{R})}^{\frac{2}{p-1}-\frac{d-2}{2}}
\left\| 
e^{it\Delta} w_{n}^{k}
\right\|_{W(\mathbb{R})}^{\frac{d}{2}-\frac{2}{p-1}}
\left\| |\nabla|^{s} e^{it\Delta} w_{n}^{k} 
\right\|_{W_{1+\frac{4}{d}}}
\\[6pt]
&\to 0 
\qquad 
\mbox{as $k\to \infty$ and $n\to \infty$}.
\end{split}
\end{equation}
Combining \eqref{11/08/21/11:39}  with \eqref{11/08/24/11:28} and  \eqref{11/08/24/11:34}, we obtain  
\begin{equation}\label{11/07/17/13:10}
\lim_{k\to \infty}\lim_{n\to \infty}\left\| 
\big| \psi_{n}^{k\mbox{-}app} \big|^{p-1}
|\nabla |^{s} e^{it\Delta}w_{n}^{k}
\right\|_{L_{t,x}^{\frac{2(d+2)}{d+4}}}=0,
\quad 
\mbox{$1+\frac{4}{d}<p \le 2^{*}-1$ and $s=0,1$}
\end{equation}
Hence, we have proved \eqref{11/07/16/9:46}. 
\par 
Finally, we prove \eqref{11/07/16/9:50}. Noting that 
\begin{equation}\label{11/12/04/16:09}
\begin{split}
i\frac{\partial \psi_{n}^{j}}{\partial t}
+
\Delta \psi_{n}^{j}
&=
-\frac{1}{(\lambda_{n}^{j})^{2}}G_{n}^{j}\sigma_{n}^{j}
\big[ (\sigma_{\infty}^{j})^{-1}\mathcal{N}_{j}(\sigma_{\infty}^{j}\widetilde{\psi}^{j}) \big]
\\[6pt]
&
=
\left\{ 
\begin{array}{ccc}
-\mathcal{N}(\psi_{n}^{j}) 
&\mbox{if}& \lambda_{\infty}^{j}=1,
\\[6pt]
-\langle \nabla \rangle^{-1}|\nabla |
\left\{ 
|G_{n}^{j}\sigma_{\infty}^{j}\widetilde{\psi}^{j}|^{2^{*}-2}
(G_{n}^{j}\sigma_{\infty}^{j}\widetilde{\psi}^{j})
\right\}
&\mbox{if}& \lambda_{\infty}^{j}=0, 
\end{array}
\right.
\end{split} 
\end{equation}
we can verify that 
\begin{equation}\label{11/07/17/13:49}
\begin{split}
&
\left\| \langle \nabla \rangle \left\{
\sum_{j=1}^{k} i\frac{\partial \psi_{n}^{j}}{\partial t}
+
\Delta \psi_{n}^{j}
+
\mathcal{N}\biggm( 
\sum_{j=1}^{k} \psi_{n}^{j}
\biggm)
\right\}\right\|_{L_{t,x}^{\frac{2(d+2)}{d+4}}}
\\[6pt]
&\le 
\left\| \langle \nabla \rangle  
\left\{
\mathcal{N}\biggm( 
\sum_{j=1}^{J} \psi_{n}^{j}
\biggm)
-
\sum_{j=1}^{J}  \mathcal{N}
\left(
\psi_{n}^{j}
\right)
\right\}\right\|_{L_{t,x}^{\frac{2(d+2)}{d+4}}}
+o_{n}(1).
\end{split}
\end{equation}
We see from the growth conditions \eqref{11/04/30/22:54} and \eqref{11/05/02/8:40} that 
\begin{equation}\label{11/07/17/15:04}
\begin{split}
&\left\| |\nabla |^{s} 
\left\{
\mathcal{N}\biggm( 
\sum_{j=1}^{J} \psi_{n}^{j}
\biggm)
-
\sum_{j=1}^{J}  \mathcal{N}
\left(
\psi_{n}^{j}
\right)
\right\}\right\|_{L_{t,x}^{\frac{2(d+2)}{d+4}}}
\\[6pt]
&\lesssim 
\sum_{k=1}^{3}\sum_{j=1}^{J}
\sum_{{1\le j' \le J}\atop {j'\neq j}}
\left\| 
\left| \psi_{n}^{j'} \right|^{p_{_{k}}-1}
|\nabla|^{s} \psi_{n}^{j}
\right\|_{L_{t}^{2}L_{x}^{\frac{2d}{d+2}}},
\end{split}
\end{equation}
where $p_{3}:=1+\frac{4}{d-2}$.
Suppose here that $\displaystyle{\lim_{n\to \infty}\frac{\lambda_{n}^{j'}}{\lambda_{n}^{j}}=\infty}$. Then, we have 
\begin{equation}\label{11/07/03/14:11}
\begin{split}
&\left\| 
 \left| \psi_{n}^{j'} \right|^{p_{_{k}}-1}
|\nabla|^{s} \psi_{n}^{j}
\right\|_{L_{t}^{2}L_{x}^{\frac{2d}{d+2}}}
\\[6pt]
&= 
(\lambda_{n}^{j})^{1-s+\frac{d-2}{2}
\left\{\frac{4}{d-2}-(p_{k}-1) \right\}}
\left(\frac{\lambda_{n}^{j}}{\lambda_{n}^{j'}}\right)^{-\frac{(d-2)(p_{l}-1)}{2}}
\\[6pt]
&\qquad \left\| 
\left| \widetilde{\psi}^{j'}
\biggm( 
\frac{\lambda_{n}^{j}x -(x_{n}^{j'}-x_{n}^{j})}{\lambda_{n}^{j'}}, \frac{(\lambda_{n}^{j})^{2}t -(t_{n}^{j'}-t_{n}^{j})}{(\lambda_{n}^{j'})^{2}}
\biggm)
 \right|^{p_{_{l}}-1}
|\nabla|^{s} \widetilde{\psi}^{j}
\right\|_{L_{t}^{2}L_{x}^{\frac{2d}{d+2}}}
\\[6pt]
&
\le 
(\lambda_{n}^{j})^{1-s+\frac{d-2}{2}
\left\{\frac{4}{d-2}-(p_{l}-1) \right\}}
\left(\frac{\lambda_{n}^{j}}{\lambda_{n}^{j'}}\right)^{-\frac{(d-2)(p_{l}-1)}{2}}
\left\|\langle \nabla \rangle \widetilde{\psi}^{j} \right\|_{S(\mathbb{R})}^{p_{l}}
=o_{n}(1).
\end{split}
\end{equation}
When $\displaystyle{\lim_{n\to \infty}\frac{\lambda_{n}^{j}}{\lambda_{n}^{j'}}=\infty}$, we instead have 
\begin{equation}\label{11/07/03/14:10}
\begin{split}
&\left\| 
\left| \psi_{n}^{j'} \right|^{p_{_{l}}-1}
|\nabla|^{s} \psi_{n}^{j}
\right\|_{L_{t}^{2}L_{x}^{\frac{2d}{d+2}}}
\\[6pt]
&= 
(\lambda_{n}^{j'})^{1-s+\frac{d-2}{2}\left\{ \frac{4}{d-2}-(p_{l}-1)\right\}}
\left(\frac{\lambda_{n}^{j'}}{\lambda_{n}^{j}}
\right)^{\frac{d}{2}}
\\[6pt]
&\qquad \left\| 
\biggm( 
\frac{\lambda_{n}^{j'}x -(x_{n}^{j}-x_{n}^{j'})}{\lambda_{n}^{j}}
,
\frac{(\lambda_{n}^{j'})^{2}t -(t_{n}^{j}-t_{n}^{j'})}{(\lambda_{n}^{j})^{2}}
\biggm)
\left| \widetilde{\psi}^{j'}
\right|^{p_{_{l}}-1}
|\nabla|^{s} \widetilde{\psi}^{j}
\right\|_{L_{t}^{2}L_{x}^{\frac{2d}{d+2}}}
\\[6pt]
&
\le 
(\lambda_{n}^{j'})^{1-s+\frac{d-2}{2}\left\{ \frac{4}{d-2}-(p_{l}-1)\right\}}
\left(\frac{\lambda_{n}^{j'}}{\lambda_{n}^{j}}
\right)^{\frac{d}{2}}
\left\|\langle \nabla \rangle \widetilde{\psi}^{j} \right\|_{S(\mathbb{R})}^{p_{l}}
=o_{n}(1).
\end{split}
\end{equation}
On the other hand, if $\displaystyle{\lim_{n\to \infty}
\left\{ \frac{\lambda_{n}^{j}}{\lambda_{n}^{j'}}+
\frac{\lambda_{n}^{j}}{\lambda_{n}^{j'}}
\right\}<\infty
}$, then we deduce from the dichotomy condition \eqref{11/07/03/10:21} that 
\begin{equation}\label{11/07/03/14:21}
\left\| 
\left| \psi_{n}^{j'} \right|^{p_{_{l}}-1}
|\nabla|^{s} \psi_{n}^{j}
\right\|_{L^{2}(I,L^{\frac{2d}{d+2}})}=o_{n}(1).
\end{equation}
Thus, we have shown \eqref{11/07/02/17:21}. 
\end{proof}

\begin{proposition}\label{11/06/10/17:46}
Assume $d\ge 5$ and suppose that $m_{\omega}^{*}<m_{\omega}$. Then, there exists a global solution $\Psi \in  C(\mathbb{R},H^{1}(\mathbb{R}^{d}))$ to  \eqref{11/06/12/9:08} such that 
\begin{align}
\label{11/06/10/17:48}
&
\Psi(t) \in A_{\omega,+}
\qquad 
\mbox{for any $t \in \mathbb{R}$},
\\[6pt]
\label{11/08/20/14:30}
&
\mathcal{S}_{\omega}\big(\Psi(t)\big)=m_{\omega}^{*}
\qquad 
\mbox{for any $t \in \mathbb{R}$},
\\[6pt] 
\label{11/08/20/14:02}
&\left\| \Psi 
\right\|_{W_{p}(\mathbb{R})\cap W(\mathbb{R})}=\infty.
\end{align}
\end{proposition}
The function $\Psi$ in Proposition \ref{11/06/10/17:46} is called  the critical element. We further give an important properties of  the critical element:
\begin{proposition}\label{11/08/20/14:32}
Let $\Psi$ be the solution found in Proposition \ref{11/06/10/17:46}. Then, for any $\varepsilon>0$, there exist $R_{\varepsilon}>0$ and $\gamma_{\varepsilon}\in C([0,\infty),\mathbb{R}^{d})$ with $\gamma_{\varepsilon}(0)=0$ such that 
\begin{equation}\label{11/07/17/15:53}
\int_{|x-\gamma_{\varepsilon}(t)|\le R_{\varepsilon}}
\left|\Psi(x,t)\right|^{2}+ \left|\nabla \Psi(x,t)\right|^{2}\,dx
\ge (1-\varepsilon) \left\| \Psi (t)\right\|_{H^{1}}^{2}
\qquad 
\mbox{for any $t \in [0,\infty)$}. 
\end{equation}
Furthermore, the momentum of $\Psi$ is zero: 
\begin{equation}\label{11/06/10/17:59}
\Im{
\int_{\mathbb{R}^{d}}\overline{\Psi(x,t)}\nabla \Psi(x,t)\,dx=0
}
\qquad 
\mbox{for any $t \in \mathbb{R}$}.
\end{equation}
\end{proposition}

We can prove Proposition \ref{11/08/20/14:32} in a way similar to  \cite{Akahori-Nawa, D-H-R}. Hence, we omit it. 
\par 
Now, we give a proof of Proposition \ref{11/06/10/17:46}: 
\begin{proof}[Proof of Proposition \ref{11/06/10/17:46}]
Using Lemmata \ref{11/08/11/21:15} and \ref{11/07/03/14:31} and reordering indices, we can take a number $J \le j_{1}$ such that \begin{equation}\label{11/07/17/15:20}
\begin{split}
&\left\| \sigma_{\infty}^{j} \widetilde{\psi}^{j} 
\right\|_{W_{j}(I^{j})}=\infty 
\qquad 
\mbox{for any $1\le j \le J$},
\\[6pt]
&\left\| \sigma_{\infty}^{j} \widetilde{\psi}^{j} 
\right\|_{W_{j}(\mathbb{R})}< \infty 
\qquad 
\mbox{for any $j > J$}.
\end{split}
\end{equation}
We see from Lemma \ref{11/07/07/20:34} that 
\begin{equation}\label{11/08/19/14:13}
\sigma_{\infty}^{j}\widetilde{\psi}^{j}(t) 
\in \left\{ \begin{array}{lcc}
 A_{\omega,+}
&\mbox{if } & \lambda_{\infty}^{j}=1,
\\[6pt]
 A_{0}
&\mbox{if } & \lambda_{\infty}^{j}=0
\end{array}
\right.
\qquad 
\mbox{for any $1\le j \le J$ and $t \in I^{j}$}.
\end{equation}
Since $\sigma_{\infty}^{j}\widetilde{\psi}^{j}$ is  a solution to \eqref{11/07/17/16:01} when $\lambda_{\infty}^{j}=0$, Theorem \ref{11/07/24/20:56} together with \eqref{11/08/19/14:13} shows that 
\begin{equation}\label{11/08/19/14:19}
\left\| \sigma_{\infty}^{j}\widetilde{\psi}^{j} \right\|_{W(\mathbb{R})}=
\left\| \sigma_{\infty}^{j}\widetilde{\psi}^{j} \right\|_{W_{j}(I^{j})}<\infty,
\qquad 
\mbox{if $\lambda_{\infty}^{j}=0$}.
\end{equation} 
Hence, we find that 
\begin{equation}\label{11/08/19/14:25}
\lambda_{\infty}^{j}=1, \quad 
\sigma_{\infty}^{j}=1  \qquad 
\mbox{for any $1\le j \le J$}. 
\end{equation}
Then, we also have by \eqref{10/09/24/21:40}, \eqref{11/06/27/22:18}--\eqref{11/06/27/22:22} that 
\begin{align}
\label{11/08/20/14:19}
&\lim_{n\to \infty}
\left\{ 
\left\| |\nabla|^{s} \psi_{n}(0) \right\|_{L^{2}}^{2}
-
\sum_{j=1}^{k}
\left\||\nabla|^{s}
\widetilde{u}^{j}
\right\|_{L^{2}}^{2}
-
\left\| |\nabla|^{s} w_{n}^{k} \right\|_{L^{2}}^{2}
\right\}
=0,
\\[6pt]
\label{11/08/19/14:39}
&
\lim_{n\to \infty}
\left\{ 
\mathcal{S}_{\omega}(\psi_{n}(0))
-
\sum_{j=1}^{J} \mathcal{S}_{\omega}
\biggm(
e^{-it_{n}^{j}\Delta}
\widetilde{u}^{j}
\biggm)
-
\mathcal{S}_{\omega}(w_{n}^{J})
\right\}
=
0,
\\[6pt]
\label{11/08/19/14:38}
&
\lim_{n\to \infty}
\left\{ 
\mathcal{I}_{\omega}(\psi_{n}(0))
-
\sum_{j=1}^{J} \mathcal{I}_{\omega}
\biggm(
e^{-it_{n}^{j}\Delta}
\widetilde{u}^{j}
\biggm)
-
\mathcal{I}_{\omega}(w_{n}^{J})
\right\}
=
0.
\end{align}

We shall show that $J=1$. Note that since $\mathcal{K}(\psi_{n}(0))>0$, we have $\mathcal{I}_{\omega}(\psi_{n}(0))< \mathcal{S}_{\omega}(\psi_{n}(0))$. It follows from \eqref{11/08/19/14:38} together with \eqref{11/06/11/6:16} that 
\begin{equation}\label{11/08/20/10:46}
\begin{split}
&\mathcal{I}_{\omega}\biggm(
e^{-it_{n}^{j}\Delta}
\widetilde{u}^{j}
\biggm), 
\quad 
\mathcal{I}_{\omega}(w_{n}^{J}) 
\\[6pt]
&\le \mathcal{S}_{\omega}(\psi_{n}(0))<\frac{m_{\omega}+m_{\omega}^{*}}{2}
\qquad 
\mbox{for any $1\le j \le J$ and sufficiently large $n$}.  
\end{split}
\end{equation}
Since $\widetilde{m}_{\omega}=m_{\omega}$ (see Proposition \ref{11/05/01/17:50}), we see from the the definition of $\widetilde{m}_{\omega}$ that 
\begin{equation}\label{11/08/20/10:54}
\mathcal{K}\biggm(
e^{-it_{n}^{j}\Delta}
\widetilde{u}^{j}
\biggm)>0, 
\quad 
\mathcal{K}(w_{n}^{J})
>0
\qquad 
\mbox{for any $1\le j \le J$ and sufficiently large $n$}.
\end{equation}
Moreover, we have by Lemma \ref{11/06/28/10:12} together with \eqref{11/08/20/10:54} that 
\begin{equation}\label{11/08/20/11:02}
\liminf_{n\to \infty}\mathcal{S}_{\omega}
\biggm(
e^{-it_{n}^{j}\Delta}
\widetilde{u}^{j}
\biggm)>0 \quad 
\mbox{for any $1\le j \le J$},
\qquad 
\liminf_{n\to \infty}\mathcal{S}_{\omega}(w_{n}^{J})\ge 0
.
\end{equation}

Now, suppose the contrary that $J\ge 2$. Then, it follows from \eqref{11/06/11/6:16}, \eqref{11/08/19/14:39} and \eqref{11/08/20/11:02} that   
\begin{equation}\label{11/08/20/11:15}
\limsup_{n\to \infty}\mathcal{S}_{\omega}
\biggm(
e^{-it_{n}^{j}\Delta}
\widetilde{u}^{j}
\biggm)<m_{\omega}^{*}
\qquad 
\mbox{for any $1\le j \le J$},
\end{equation}
which together with \eqref{11/06/22/22:37} and the action-conservation law yields that 
\begin{equation}\label{11/07/17/15:37}
\mathcal{S}_{\omega}\bigm(\widetilde{\psi}^{j}(t)\bigm)<m_{\omega}^{*},
\qquad 
\mbox{for any $1\le j \le J$ and $t \in I^{j}$}.
\end{equation}
Since $\widetilde{\psi}^{j}$ is a solution to \eqref{11/06/12/9:08}, it follows from the definition of $m_{\omega}^{*}$ that 
\begin{equation}\label{11/07/17/23:05}
\left\|\widetilde{\psi}^{j} \right\|_{W_{p_{1}}(I^{j})\cap W(I^{j})}< \infty 
\qquad 
\mbox{for any $1\le j \le J$}.
\end{equation}
This contradicts \eqref{11/07/17/15:20}. Thus, we have $J=1$. 
\par
Since $\big\| \widetilde{\psi}^{1} \big\|_{W_{p_{1}}(I^{1})\cap W(I^{1})}=\infty$, we have  
\begin{equation}\label{11/07/17/23:16}
\mathcal{S}_{\omega}\bigm(\widetilde{\psi}^{1}(t)\bigm)\ge m_{\omega}^{*}
\qquad 
\mbox{for any $t \in I^{1}$}.
\end{equation}
On the other hand,  we see from a proof similar to the one of \eqref{11/07/17/15:37} that 
\begin{equation}\label{11/07/17/23:20}
\mathcal{S}_{\omega}\bigm(\widetilde{\psi}^{1}(t)\bigm)\le m_{\omega}^{*}
\qquad 
\mbox{for any $t \in I^{1}$}.
\end{equation}
Combining \eqref{11/07/17/23:16} and \eqref{11/07/17/23:20}, we obtain 
\begin{equation}\label{11/07/17/23:21}
\mathcal{S}_{\omega}\bigm(\widetilde{\psi}^{1}(t)\bigm)= m_{\omega}^{*}
\qquad 
\mbox{for any $t \in I^{1}$}.
\end{equation}
Since we have by \eqref{11/06/22/22:37} that 
\begin{equation}\label{11/08/20/11:50}
\mathcal{S}_{\omega}\bigm( \widetilde{\psi}^{1}(t) \bigm)
=
\mathcal{S}_{\omega}
\biggm(
e^{-it_{n}^{1}\Delta}
\widetilde{u}^{1}
\biggm),
\end{equation}
\eqref{11/08/19/14:39} together with \eqref{11/06/11/6:16} and \eqref{11/07/17/23:21} shows   
\begin{equation}
\label{11/07/17/23:07}
\lim_{n\to \infty}\mathcal{S}_{\omega}(w_{n}^{1})
=0.
\end{equation}
Hence, Lemma \ref{11/06/28/10:12} together with \eqref{11/08/20/10:54} with $J=1$ and \eqref{11/07/17/23:07} shows  
\begin{equation}\label{11/07/17/23:36}
\lim_{n\to \infty}\left\| w_{n}^{1} \right\|_{H^{1}}
=0. 
\end{equation}
We see from \eqref{11/06/22/22:27} together with \eqref{11/07/17/23:36} that
\begin{equation}\label{11/07/17/23:37}
\lim_{n\to \infty}\left\| \psi_{n}(0) 
-e^{-it_{n}^{1}\Delta}\widetilde{u}^{1}(\cdot -x_{n}^{1})
\right\|_{H^{1}}
=0.
\end{equation}
Now, we shall show that $I^{1}=\mathbb{R}$. Suppose the contrary that $T^{1}:=\sup{I^{1}}<\infty$. Let $\{t_{n}\}$ be a sequence in $I^{1}$ such that $\displaystyle{\lim_{n\to \infty}t_{n}\uparrow T_{\max}^{1}}$, and put $\widetilde{\psi}_{n}(t):=\widetilde{\psi}^{1}(t+t_{n})$ and $\widetilde{I}_{n}:=I^{1}-\tau_{n}^{1}$. We easily verify that the sequence $\{\widetilde{\psi}_{n} \}$ satisfies that
\begin{align}
\label{11/08/20/16:39}
&\widetilde{\psi}_{n}(t)\in A_{\omega,+} 
\qquad 
\mbox{for any $t \in \widetilde{I}_{n}$},
\\[6pt] 
\label{11/08/20/16:43}
&\lim_{n\to \infty}\mathcal{S}_{\omega}(\widetilde{\psi}_{n})=m_{\omega}^{*},
\\[6pt]
\label{11/08/20/16:45}
&\left\| \widetilde{\psi}_{n} \right\|_{W_{p_{1}}(\widetilde{I}_{n})\cap W(\widetilde{I}_{n})}=\infty.
\end{align}
Then, we can apply the above argument to this sequence, and find  as well as \eqref{11/07/17/23:37} that there exists a non-trivial function $v \in H^{1}(\mathbb{R}^{d})$, a sequence $\{\tau_{n}\}$ with $\displaystyle{\tau_{\infty}:=\lim_{n\to \infty}\tau_{n} \in \mathbb{R}\cup \{\pm \infty\}}$, and a sequence $\{\xi_{n}\}$, such that 
\begin{equation}\label{11/08/20/16:53}
\lim_{n\to \infty}\left\| \widetilde{\psi}^{1}(t_{n}) 
-e^{-i\tau_{n} \Delta}v(\cdot -\xi_{n})
\right\|_{H^{1}}
=
\lim_{n\to \infty}\left\| \widetilde{\psi}_{n}(0) 
-
e^{-i\tau_{n} \Delta}v(\cdot -\xi_{n})
\right\|_{H^{1}}
=0.
\end{equation}  
This together with the Strichartz estimate also yields that  
\begin{equation}\label{11/08/20/17:37}
\lim_{n\to \infty}
\left\| e^{it\Delta}\widetilde{\psi}^{1}(t_{n}) 
- e^{i(t-\tau_{n})\Delta} v(\cdot -\xi_{n}) 
\right\|_{V_{p_{1}}(\mathbb{R})}
=0 
.
\end{equation}
When $\tau_{\infty}=\pm \infty$, it follows from the decay estimate for the free solution that, for any compact interval $I$, we have 
\begin{equation}
\lim_{n\to \infty}
\left\| 
e^{i(t-\tau_{n})\Delta} v  
\right\|_{V_{p_{1}}(I)}=0,
\end{equation}
which, with the help of \eqref{11/08/20/17:37}, also yields that 
\begin{equation}\label{11/08/20/18:29}
\lim_{n\to \infty}\left\| e^{it\Delta}\widetilde{\psi}^{1}(t_{n}) \right\|_{V_{p_{1}}(I)}
=0. 
\end{equation}
On the other hand, when $\tau_{\infty}\in \mathbb{R}$, we have  
\begin{equation}\label{11/08/20/17:40}
\lim_{n\to \infty}\left\| e^{it\Delta}\widetilde{\psi}^{1}(t_{n}) \right\|_{V_{p_{1}(I)}}
\hspace{-6pt}=
\left\| e^{i(t-\tau_{\infty})\Delta}v \right\|_{V_{p_{1}(I)}}
\hspace{-6pt}
\ll 1
\quad 
\mbox{for any interval $I$ with $|I|\ll 1$}. 
\end{equation}
Then, Theorem \ref{10/10/04/21:44} together with \eqref{11/08/20/18:29} and \eqref{11/08/20/17:40} implies that $\widetilde{\psi}^{1}$ exists beyond $T^{1}$, which is a contradiction. Thus, $\sup{I^{1}}=+\infty$. Similarly, we have $\inf{I^{1}}=-\infty$; Hence, $I^{1}=\mathbb{R}$. 
\par 
Put $\Psi:=\widetilde{\psi}^{1}$. Then, this $\Psi$ is what we want.
\end{proof}

\subsection{Completion of the proof of Theorem \ref{11/06/10/17:09}}\label{11/12/05/10:23}
Since our critical element $\Psi$ belongs to $C(\mathbb{R},H^{1}(\mathbb{R}^{d}))$ (see Proposition \ref{11/06/10/17:46}), we can derive a contradiction in the same way as the energy subcritical case (see \cite{Akahori-Nawa, D-H-R}). Thus, we have $m_{\omega}^{*}=m_{\omega}$, which together with Theorem \ref{10/10/04/21:44} {\rm (v)} completes the proof of Theorem \ref{11/06/10/17:09}.

\bibliographystyle{plain}

\vspace{24pt}
{
Takafumi Akahori,
\\
Faculty of Engineering
\\
Shizuoka University, 
\\ 
Jyohoku 3-5-1, Hamamatsu, 432-8561, Japan 
\\ 
E-mail: ttakaho@ipc.shizuoka.ac.jp
}
\\
\\
{
Slim Ibrahim,
\\ 
Department of Mathematics and Statistics 
\\
University of Victoria
\\
Victoria, British Columbia 
\\
E-mail:ibrahim@math.uvic.ca
}
\\
\\
{
Hiroaki Kikuchi,
\\ 
School of Information Environment
\\
Tokyo Denki University, 
\\
Inzai, Chiba 270-1382, Japan 
\\
E-mail: hiroaki@sie.dendai.ac.jp
}
\\
\\
{
Hayato Nawa,
\\
Division of Mathematical Science, Department of System Innovation
\\
Graduate School of Engineering Science
\\
Osaka University, 
\\
Toyonaka 560-8531, Japan 
\\
E-mail: nawa@sigmath.es.osaka-u.ac.jp
}
\end{document}